\documentclass{agn_article}
\usepackage{agn_all}

\usepackage[hidelinks]{hyperref}
\usepackage{framed}
\usepackage{tabularx}
\usepackage{tabu}
\usepackage{bigdelim}
\usepackage{rotating}
\usepackage{accents}
\usepackage{relsize}
\usepackage{multirow}
\usepackage{centernot}

\usetikzlibrary{decorations.pathreplacing}
\newcolumntype{L}[1]{>{\raggedright\arraybackslash}p{#1}}
\newcolumntype{C}[1]{>{\centering\arraybackslash}p{#1}}
\newcolumntype{R}[1]{>{\raggedleft\arraybackslash}p{#1}}

\newcommand{\citet}[2][]{\citeauthor{#2} \cite[#1]{#2}}
\newcommand{\sort}[1]{{#1}^\downarrow}
\newcommand{\Rsort}{\R^n_\downarrow}
\newcommand{\Msort}{M_\downarrow}

\newcommand{\lambdamed}{\lambda_{\mathrm{med}}}

\newcommand{\olambdamax}{\overline\lambda_{\mathrm{max}}}
\newcommand{\olambdamed}{\overline\lambda_{\mathrm{med}}}
\newcommand{\olambdamin}{\overline\lambda_{\mathrm{min}}}

\newcommand{\fpot}{g}
\newcommand{\kpow}{^{(k)}}
\renewcommand{\Symp}{\Sym^{++}}

\begin{document}
\title{Monotonicity of isotropic tensor functions on the set of\\ symmetric matrices: completing Rodney Hill's generalization of\\ the Chandler Davis convexity theorem}
\date{\today}

\defineaffiliation{insa}{
	 LGCIE, INSA-Lyon, Universit\'e de Lyon, 20 avenue Albert Einstein, 69621, Villeurbanne cedex, France
}
\defineauthor{agostino}{Marco~Valerio~d'Agostino}{insa}{marco-valerio.dagostino@insa-lyon.fr}

\knownauthors[vossd]{vossd,martin,ghiba,agostino,neff}

\maketitle
\begin{center}
\vspace{-0.8cm}
\noindent \enquote{...[an] elementary but important theorem.}
\medskip
 
\noindent Rodney Hill, Constitutive inequalities for isotropic elastic solids under finite strain,\\Proceedings of the Royal Society 1970, p.466 (b).
\end{center}

\bigskip
\begin{abstract}
	\noindent
	Motivated by classical constitutive inequalities in isotropic nonlinear elasticity theory, we investigate the monotonicity of isotropic tensor functions of the form
	\[
		\Sigma_f\colon\mathrm{Sym}(n)\to\mathrm{Sym}(n)\,,\quad \Sigma_f(Q^T\mathrm{diag}(\lambda_1,\dotsc,\lambda_n)\. Q) = Q^T\mathrm{diag}(f(\lambda_1,\dotsc,\lambda_n))\. Q \quad\forall\;Q\in\mathrm{O}(n)
	\]
	with a vector function $f=(f_1,\dotsc,f_n)\colon\mathbb{R}^n\to\mathbb{R}^n$ which is symmetric, i.e.\ satisfies
	\[
		f_i(\lambda_{\pi(1)},\dotsc,\lambda_{\pi(n)}) = f_{\pi(i)}(\lambda_1,\dotsc,\lambda_n)
	\]
	for any permutation $\pi\colon\{1,\dotsc,n\}\to\{1,\dotsc,n\}$, where $\mathrm{Sym}(n)$ denotes the space of symmetric $n\times n$ matrices, $\mathrm{O}n$ is the orthogonal group and $\mathrm{diag}(\lambda_1,\dotsc,\lambda_n)$ is the diagonal matrix with diagonal entries $\lambda_1,\dotsc,\lambda_n\in\mathbb{R}$. We prove that vector-monotonicity of $f$ on $\mathbb{R}^n$ is equivalent to matrix-monotonicity of the induced isotropic tensor function $\Sigma_f$ on $\mathrm{Sym}(n)$. Our results generalize the Chandler Davis theorem for convex scalar isotropic functions and are obtained independently of Hill's original proof of this equivalence. We also discuss simple invertibility conditions for isotropic matrix functions. We conclude by showing that injectivity of the Cauchy stress $V\mapsto\sigma(V)$, continuous differentiability,  and positive definiteness of $\mathrm{sym}\,\mathrm D\sigma(\id)$ in the natural state imply the strong Baker-Ericksen inequalities.
\end{abstract}
\\[.65em]

\clearpage

{\parskip=-0.4mm\tableofcontents}
%
%
%
%
\section{Introduction}

Isotropic functions of symmetric matrices play a central role in nonlinear elasticity, matrix analysis, and invariant theory \cite{silhavy2013mechanics,mihai2017characterize,anssari2023large}. 
Their structure is governed by the spectral theorem: every symmetric matrix can be diagonalized by an orthogonal transformation, 
and isotropy implies that such functions depend only on the eigenvalues of their argument. 
This fundamental observation has far-reaching consequences and interesting connections to primary matrix functions \cite{martin2015some}.

A fundamental result in this context is the convexity theorem of Chandler Davis \cite{davis1957all}, which states that 
an isotropic scalar-valued function defined on the space $\Symn$ of symmetric matrices is convex if and only if 
its representation in terms of eigenvalues is convex as a symmetric function on $\R^n$. 
An independent proof and application were later given by Ball \cite{ball1976convexity} and the theorem has become a standard tool in nonlinear elasticity and convex matrix analysis.

In the differentiable case, convexity of a scalar-valued function is equivalent to monotonicity of its gradient. 
Thus, in the potential case, i.e.\ when a tensor-valued function arises as the derivative of a scalar-valued isotropic potential, the Davis theorem implies an equivalence between:
\begin{itemize}
	\item monotonicity of the corresponding isotropic tensor function $\D W$ on $\Sym(n)$ where $W$ is isotropic, and
	\item monotonicity of the gradient field $\grad\fpot$ in eigenvalue variables $\lambda\in\R^n$ where $\fpot$ denotes the eigenvalue representation of $W$.
\end{itemize}

However, many important constitutive relations in nonlinear elasticity are not gradients of scalar-valued potentials \cite{johnson1994large}, the most prominent example being the Cauchy stress tensor itself. 
In such non-potential cases, the equivalence between eigenvalue-based monotonicity and full tensorial monotonicity 
is far from obvious. Indeed, it is natural to suspect that this implication may fail outside the gradient framework.

In a remarkable paper, Rodney Hill \cite{hill1970constitutive} observed that the equivalence should nevertheless remain true for arbitrary isotropic tensor functions and presented what he described as an ``elementary but important theorem'' reducing matrix-monotonicity to vector-monotonicity. Hill's result may therefore be regarded as a first analogue of the Chandler Davis convexity theorem. His proof is very condensed and the ordering step is not made explicit. 
Also, in \cite[Appendix 1 Convex Functions]{ogden1997non} Ogden interprets Hill's proof for vector-monotonicity implies matrix-monotonicity for $n=3$ and concludes that it holds for the subset of ordered eigenvalues.
Nevertheless, the statement is correct (for any order) and Hill’s underlying rearrangement idea can be completed rigorously.

The main purpose of the present paper is to revisit Hill's theorem with a new, completely independent proof of the statement: for symmetric vector fields $f\col\R^n\to\R^n$, the following are equivalent:
\begin{itemize}
	\item vector-monotonicity in $\R^n$, and
	\item matrix-monotonicity of the induced isotropic tensor function $\Sigma_f$ on $\Sym(n)$.
\end{itemize}

Beyond its intrinsic interest in matrix analysis, this equivalence has direct implications in nonlinear elasticity. 
It shows that monotonicity conditions formulated in terms of principal stresses and principal strains are equivalent to tensorial monotonicity conditions for isotropic constitutive laws. 

In particular, it clarifies the relationship between various nonlinear generalizations of Drucker-type stability conditions 
and provides insight into both invertibility properties and the Baker-Ericksen inequalities \cite{baker1954inequalities}.

%
%
%
%
%
\section{Representations of isotropic functions and classical results}\label{sec:definitions}
We begin with some basic definitions concerning symmetric vector functions.
\begin{definition}[Symmetric set]
	We call $M\subset\R^n$ a \emph{symmetric set} if
	\begin{equation}
		x\in M\qquad\implies\qquad\left(x_{\pi(i)}\right)_i\in M
	\end{equation}
	for all permutations $\pi\col\{1,\dotsc,n\}\to\{1,\dotsc,n\}$.
\end{definition}
\begin{remark}
	If $I\subset\R$ is an interval, then $I^n\subset\R^n$ is a symmetric set.
\end{remark}
\begin{definition}
	Let $M\subset\R^n$ be a symmetric set. Then we denote by $\Sym(M)\subset\Sym(n)$ the set of all symmetric matrices with eigenvalues in $M$.
\end{definition}
Note that although the algebraic order of eigenvalues is not unique, for every matrix $S$ with eigenvalues $(s_1,\dotsc,s_n)\in M$, every permutation $(s_{\pi(1)},\dotsc,s_{\pi(n)})$ also belongs to $M$ by symmetry.

In the following, we denote by $\Sym(\Rp^n)=\Symp(n)$ the set of all positive definite matrices.
\begin{definition}[Symmetric function]\label{definition:Symmetric}
	A function $f=(f_1,\dotsc,f_n)\col M\subset\R^n\to\R^n$ on a symmetric set $M$ is called \emph{symmetric} if
	\begin{align}
		f_i(\lambda_{\pi(1)},\dotsc,\lambda_{\pi(n)}) = f_{\pi(i)}(\lambda_1,\dotsc,\lambda_n)
	\end{align}
	for any permutation $\pi\col\{1,\dotsc,n\}\to\{1,\dotsc,n\}$.
\end{definition}

%
%
\subsection{Three equivalent representations}\label{sec:3representations}
We recall the different representations of isotropic scalar- and tensor-valued functions in terms of matrices, (unordered) eigenvalues and ordered eigenvalues.
\begin{definition}[Isotropy]
	Let $M\subset\R^n$ be a symmetric set. A scalar-valued function $W \col \Sym(M)\to\R$ is called isotropic if
	\begin{equation}
		W(Q^TS\.Q)=W(S)
		\qquad\forall\; Q\in\OO(n)\,.
	\end{equation}
	A tensor function $\Sigma \col \Sym(M)\to\Sym(n)$ is called isotropic if
	\begin{equation}
		\Sigma(Q^TS\.Q)=Q^T\Sigma(S)\.Q
		\qquad\forall\; Q\in\OO(n)\,.\label{eq:isotropic}
	\end{equation}
\end{definition}
It is well known that $W$ can be represented in terms of the \emph{unordered} eigenvalues of the argument, more specifically: there exists a unique symmetric function $\fpot\col M\subset\R^n\to\R$ such that
\begin{equation}\label{eq:eigenvalueRepresentation}
	W(S) = \fpot(\lambda_1,\dotsc,\lambda_n)
\end{equation}
for all $S\in\Sym(M)\subset\Sym(n)$ with eigenvalues $\lambda_1(S),\dotsc,\lambda_n(S)$. Note that the symmetry of $\fpot$ ensures that \eqref{eq:eigenvalueRepresentation} holds independently of the ordering of the vector $\lambda\in M\subset\R^n$.

Similarly, an isotropic tensor function $\Sigma \col \Sym(M)\subset\Sym(n)\to\Sym(n)$  admits a unique representation
\begin{equation}
	\Sigma(S)=Q^T\diag\!\big(f(\lambda_1,\dots,\lambda_n)\big)\.Q\,,\qquad Q\in\OO(n)\,,
\end{equation}
where $f \col M\subset \R^n\to\R^n$ is a symmetric vector field and
\begin{equation}
	S=Q^T\diag(\lambda_1,\dots,\lambda_n)\.Q\,.
\end{equation}
\begin{definition}[Ordered vectors]
	Let
	\begin{equation}
		\Rsort \colonequals \{(\lambda_1,\dotsc,\lambda_n)\in\R^n \setvert \lambda_1\geq\dotsc\geq\lambda_n\},\qquad
		\Msort\colonequals \{(\lambda_1,\dotsc,\lambda_n)\in M \setvert \lambda_1\geq\dotsc\geq\lambda_n\}
	\end{equation}
	be the set of all ordered vectors in $\R^n$ and in a symmetric set $M\subset\R^n$, respectively.
\end{definition}
There is a similar, yet distinct representation in terms of the ordered eigenvalues: It is well known that there exists a unique function $\psi\col \Msort\subset\Rsort\to\R$ such that
\begin{equation}\label{eq:orderedEigenvalueRepresentation}
	W(S) =  \psi(\sort{\lambda}_1,\dotsc,\sort{\lambda}_n)
\end{equation}
for all $S\in\Sym(M)\subset\Sym(n)$ with (unsorted) eigenvalues $\lambda_1(S),\dotsc,\lambda_n(S)$, where $\sort{\lambda}\in\Msort$ denotes the sorting of a vector $\lambda\in M\subset\R^n$ in decreasing order.
		
Similarly, an isotropic tensor function $\Sigma \col \Sym(M)\subset\Sym(n)\to\Sym(n)$  admits a unique representation
\begin{equation}
	\Sigma(S)=Q^T\diag\!\big(\phi(\sort{\lambda}_1,\dotsc,\sort{\lambda}_n)\big)\.Q\,,
		\qquad	S=Q^T\diag(\sort{\lambda}_1,\dotsc,\sort{\lambda}_n)\.Q\,,
\end{equation}
where $\phi \col \Msort\subset \Rsort\to\R^n$ is a vector field satisfying
\begin{equation}
	\lambda_i=\lambda_j\qquad\implies\qquad\phi_i(\lambda)=\phi_j(\lambda)\,.\label{eq:orderPreservation}
\end{equation}
In particular, convexity or monotonicity of an isotropic function can also be characterized in the ordered formulation by restricting to $\Rsort$, provided the corresponding ordering and equality compatibility conditions \eqref{eq:orderPreservation} are imposed.

Tables \ref{table:overviewRealValued} and \ref{table:overviewTensorValued} provide an overview of the tensor, eigenvalue and ordered-eigenvalue representation of scalar-valued and tensor-valued functions on $\Symn$.  We continue with the equivalence between all three notions in Section \ref{sec:Preliminaries}.

\begin{table}[h!]
	\tabulinesep=.4em
	\begin{tabu}{|X[m,c]|X[m,c]|X[m,c]|X[m,c]|}
		\hline
		& $W\col\Symn\to\R$ & $\fpot\col\R^n\to\R$ & $\psi\col\Rsort\to\R$
		\\\hline
		representation: & $W(S)$ & $\fpot(\lambda_1,\dotsc,\lambda_n)$ & $\psi(\sort{\lambda}_1,\dotsc,\sort{\lambda}_n)$
		\\\hline
		well-definedness/\newline isotropy: & $W(Q^TSQ)=W(S)$ & $\fpot((\lambda_{\pi(i)})_i) = \fpot(\lambda)$ & ---
		\\\hline
		convexity: & $W$ convex on $\Symn$ & $\fpot$ convex on $\R^n$ & $\psi$ convex on $\Rsort$ and $\grad \psi(\lambda) \in \Rsort$
		\\\hline
	\end{tabu}
	\caption{\label{table:overviewRealValued} Different representations of isotropic scalar-valued functions on $\Symn$ with $\psi$ differentiable for the derivative-order characterization of convexity.}
\end{table}
\begin{table}[h!]
	\tabulinesep=.4em
	\begin{tabu}{|X[m,c]|X[m,c]|X[m,c]|X[m,c]|}
		\hline
		& $\Sigma\col\Symn\to\Symn$ & $f\col\R^n\to\R^n$ & $\phi\col\Rsort\to\R^n$
		\\\hline
		representation: & $\Sigma(S)$ & $f(\lambda_1,\dotsc,\lambda_n)$ & $\phi(\sort{\lambda}_1,\dotsc,\sort{\lambda}_n)$
		\\\hline
		well-definedness/\newline isotropy: & $\Sigma(Q^TSQ)=Q^T\Sigma(S)Q$ & $f_i\bigl((\lambda_{\pi(i)})_i\bigr) = f_{\pi(i)}\bigl((\lambda_i)_i\bigr)$ & \hspace{0.5cm}$\lambda_i = \lambda_j$\newline $\implies\, \phi_i(\lambda) = \phi_j(\lambda)$
		\\\hline
		monotonicity: & $\Sigma$ monotone on $\Symn$ & $f$ monotone on $\R^n$ & $\phi$ monotone on $\Rsort$ and $\phi(\lambda) \in \Rsort$
		\\\hline
	\end{tabu}
	\caption{\label{table:overviewTensorValued} Different representations of isotropic tensor-valued functions on $\Symn$.}
\end{table}

%
%
%
\subsection{Basic properties}
We continue with an overview of basic properties for convexity and monotonicity for the (unordered) case, i.e.\ with a symmetric set $M\subset\R^n$ and associated symmetric matrices $\Sym(M)\subset\Symn$ with (unordered) eigenvalues in $M$.
\begin{definition}[Convexity]
	An isotropic scalar-valued function $W \col \Sym(M)\to\R$ on a convex set $\Sym(M)\subset\Sym(n)$ is convex if
	\begin{equation}
		W((1-t)\.S+t\.T)\leq(1-t)\.W(S)+t\.W(T)
		\qquad\forall\; S,T\in\Sym(M)\subset\Sym(n)\,,\quad t\in[0,1]\,.\label{eq:convex1}
	\end{equation}
	A scalar-valued function $\fpot \col M\to\R$ on a convex set $M\subset \R^n$ is convex if
	\begin{equation}
		\fpot((1-t)\.x+t\.y)\leq(1-t)\.\fpot(x)+t\.\fpot(y)
		\qquad\forall\; x,y\in M\subset\R^n\,,\quad t\in[0,1]\,.\label{eq:convex2}
	\end{equation}
\end{definition}
We refer to \emph{strict convexity} when the inequalities \eqref{eq:convex1} and \eqref{eq:convex2} are strict for all $S\neq T$ and $x\neq y$ for $t\in(0,1)$.
\begin{definition}[Monotonicity]
	\label{def:monotonicity}
	An isotropic tensor function $\Sigma \col \Sym(M)\subset\Sym(n)\to\Sym(n)$ is \emph{monotone} if
	\begin{equation}
		\iprod{\Sigma(S)-\Sigma(T),S-T}\geq 0
		\qquad\forall\; S,T\in\Sym(M)\subset\Sym(n)\,,\label{eq:monotonicity1}
	\end{equation}
	where $\iprod{\cdot,\cdot}$ is the canonical inner product on $\Rnn$, i.e.\ $\iprod{A,B}=\tr(A^TB)$ for $A,B\in\Rnn$.
	
	A vector field $f \col M\subset \R^n\to\R^n$ is \emph{monotone} if
	\begin{equation}
		\iprod{f(x)-f(y),x-y}\geq 0
		\qquad\forall\; x,y\in M\subset\R^n,\label{eq:monotonicity2}
	\end{equation}
	where $\iprod{\cdot,\cdot}$ is the canonical inner product on $\R^n$.
\end{definition}
We refer to \emph{strict monotonicity} when the inequalities \eqref{eq:monotonicity1} and \eqref{eq:monotonicity2} are strict for all $S\neq T$ and $x\neq y$.

In the following section, we will now recall the Chandler Davis convexity theorem which connects the matrix and eigenvalues representation.

%
%
\subsection{The Chandler Davis convexity theorem}\label{sec:ChandlerDavis}

Let $W\col\Symn\to\R$ be isotropic, i.e.\ $W(Q^TS\.Q)=W(S)\text{ for all orthogonal }Q\in O(n)\,.$ The beautiful Chandler Davis convexity theorem \cite{davis1957all} for convex scalar-valued isotropic functions $W$ on real symmetric matrices $S\in\Symn$ states that $W\col\Symn\to\R$ is convex if and only if $W$ restricted to diagonal matrices 
\begin{align}
	W(S)=W(Q^T\diag(\lambda_1,\dotsc,\lambda_n)\.Q)=W\begin{pmatrix}\lambda_1&&\\&\ddots&\\&&\lambda_n\end{pmatrix}=\fpot(\lambda_1,\dotsc,\lambda_n)
\end{align}
is convex, in which $Q^T\diag(\lambda_1,\dotsc,\lambda_n)\.Q=S$ for some $Q\in\OO(n)$. Chandler Davis \cite{davis1957all} and John Ball \cite{ball1976convexity} have given independent proofs of this fact not requiring differentiability of $W$. Sir John Ball has extensively used the characterization of convexity expressed by $\fpot$ for the investigations of his polyconvexity \cite{ball1976convexity} and rank-one convexity in nonlinear elasticity. Ball's \cite{ball1976convexity} Theorem 5.1(i) states that if
\begin{align}
	\fpot\col\R^n_+\to\R\quad\text{is symmetric,}˝\qquad\text{then}\quad W\col U\in\Symp(n)\mapsto \fpot(\lambda_1(U),\dotsc,\lambda_n(U))\label{eq:BallTheorem5.1i}
\end{align}
is convex if and only if $\fpot$ is convex, with $\lambda_i(U)$ the eigenvalues of the positive definite tensor $U\in\Symp(n)\,.$ For further generalizations see also \citet{marques1982isotropie}. 
A well-known immediate application of \eqref{eq:BallTheorem5.1i} is the convexity of the mapping
\begin{align}
	W\col \Symp(n)\to\R\,,\qquad W(U)=-\log\det U\,.
\end{align}
Indeed
\begin{equation}
	W(U)=-\log(\lambda_1\dotsc\lambda_n)=-(\log\lambda_1+\dotsc+\log\lambda_n)=\fpot(\lambda_1,\dotsc,\lambda_n)
\end{equation}
and $\fpot$ is symmetric, where $\lambda_i=\lambda_i(U)$ are the eigenvalues of $U\,.$ Clearly, $\fpot\col\Rp^n\to\R$ is convex, and so is $W$.

Another surprising application of the Chandler Davis convexity theorem is an ingenious short proof by \citet{rivin2010golden} of a Golden-Thompson-type inequality \cite{golden1965lower,thompson1965inequality}, see also \citet[p.261]{bhatia2013matrix}, which states that for two Hermitian matrices $A$ and $B$, which need not commute, the following trace inequality holds:
\begin{align}
	\tr(\exp(A+B))\leq\tr(\exp(A))\.\tr(\exp(B))\,.\label{eq:GoldenThompson}
\end{align}
Rivin showed that the G-T-type inequality\footnote{The classical Golden-Thompson inequality uses $\tr(\exp(A)\.\exp(B))$ as its right-hand side. The product-of-trace bound from Rivin's statement \eqref{eq:GoldenThompson} is weaker, i.e.\ $\tr(\exp(A)\.\exp(B))\leq\tr(\exp(A))\.\tr(\exp(B))$ for positive definite matrices $\exp(A)$, $\exp(B)$.} \eqref{eq:GoldenThompson} is an easy consequence of the convexity of the mapping $X\in\Symn \mapsto\log\tr(\exp X)$, see Appendix \ref{appendix:RivinGoldenThompson}.

Further generalizations of the Chandler Davis theorem to groups acting non-linearly on convex subsets of arbitrary vector spaces are presented in \citet{grabovsky2005generalization} with interesting applications to the theory of composite materials.
Further proofs and interesting applications can be found in \citet[Another simple proof of a theorem of Chandler Davis]{rivin2002another}, \citet[Group invariance and convex matrix analysis]{lewis1996group}, \citet[Davis’ convexity theorem and extremal ellipsoids]{weber2010davis} and \citet[Hyperbolic polynomials and convex analysis]{lewisz1998hyperbolic}.

\citet{lewis1996derivatives} showed that $W$ is differentiable if and only if $\fpot$ is differentiable, see also \citet{silhavy2013mechanics} for higher differentiability.
In the differentiable case, Rodney \citet{hill1968constitutive} presented a simple argument for the Chandler Davis theorem (repeated in this paper as Remark \ref{remark:HillPotential}). Further developments and refinements of the convexity characterization for isotropic matrix functions were given in \cite{le1990fonctions}.

For Fr\'echet-differentiable functions $W\col\Symn\to\R\,,$ the convexity of $W$ translates into a monotonicity condition for the derivative $\D W(S)\in\Symn\,.$ The derivative of an isotropic scalar-valued function $W$ is an isotropic tensor function, i.e.
\begin{align}
	\D W(Q^TS\.Q)=Q^T\.\D W(S)\.Q\qquad\forall\;Q\in O(n)
\end{align}
and strict monotonicity (cf.\ Definition \ref{def:monotonicity}) in the canonical Hilbert-space inner product reads
\begin{align}
	\iprod{\D W(S)-\D W(T),\,S-T}>0\qquad\forall\; S\neq T\in\Symn\,.\label{eq:introductionMonotonicity}
\end{align}
Similarly, for the corresponding permutation-symmetric function $\fpot\col\R^n\to\R\,,$ strict convexity translates into strict monotonicity of the derivative:
\begin{align}
	\iprod{\grad\fpot(\lambda)-\grad\fpot(\mu),\,\lambda-\mu}>0\qquad\forall\;\lambda\neq\mu\in\R^n.\label{eq:introductionMonotonicity2}
\end{align}
We note that, due to isotropy, $\grad\fpot=(\partial_1\fpot,\dotsc,\partial_n\fpot)\col\R^n\to\R^n$ itself is symmetric.
\begin{proposition}\label{lemma:mono1}
	Let  $M\subset\R^n$ be a convex symmetric set with $\Sym(M)\subset\Sym(n)$ convex and $W \col \Sym(M)\to \R$ be isotropic, i.e.\ we can write $W(S)=\fpot(\lambda(S))$ for some $\fpot \col M \to \R$ symmetric and differentiable and $\lambda=(\lambda_1,\dotsc,\lambda_n)$ the eigenvalues of $S$.
	We define $\Sigma\colonequals \D W$ and $f\colonequals\grad\fpot$.
	Then
	\begin{equation}
		\Sigma\quad \text{is monotone on}\quad\Sym(M)
		\qquad\iff\qquad
		f\quad\text{is monotone on}\quad M\,.
	\end{equation}
\end{proposition}
\begin{proof}
	Since $\Sigma=\D W$ and $f=\grad\fpot$, we already discussed
	\begin{align*}
		\Sigma\quad\text{is monotone on}\quad\Sym(M)
		\qquad&\iff\qquad
		W\quad\text{is convex on}\quad\Sym(M)\,,\\
		f\quad\text{is monotone on}\quad M
		\qquad&\iff\qquad
		\fpot\quad\text{is convex on}\quad M\,.
	\end{align*}
	The theorem of Chandler Davis yields
	\begin{align*}
		W\quad\text{is convex on}\quad\Sym(M)
		\qquad\iff\qquad
		\fpot\quad\text{is convex on}\quad M\,.
	\end{align*}
	Combining these equivalences yields
	\begin{align*}
		\Sigma \text{ monotone}
		\quad&\iff\quad
		W \text{ convex}
		\quad\iff\quad
		\fpot \text{ convex}
		\quad\iff\quad
		f \text{ monotone}\,.
		\qedhere
	\end{align*}
\end{proof}
The proof of Proposition \ref{lemma:mono1} requires a potential in order to apply the theorem of Chandler Davis. A straightforward computation between monotonicity of $\Sigma$ and $f$ only holds for commuting matrices.\footnote{
	Let $S,T\in\Sym(n)$ with orthogonal matrices $Q,R\in O(n)$ such that $S = Q^T\diag(\lambda)\.Q$ and $T = R^T\diag(\mu)\.R$.
Monotonicity would directly follow from equivalence of
\begin{align*}
	\iprod{\Sigma(S)-\Sigma(T),S-T}
	&= \iprod{Q^T\diag(f(\lambda))\.Q - R^T\diag(f(\mu))\.R, Q^T\diag(\lambda)\.Q - R^T\diag(\mu)\.R}\\
	&\overset{*}{=} \iprod{\diag(f(\lambda)) - \diag(f(\mu)), \diag(\lambda) - \diag(\mu)}=\iprod{f(\lambda)-f(\mu),\lambda-\mu}\,.
\end{align*}
However, step $*$ is not true in general for $Q\neq R$. The equivalence holds if and only if $S$ and $T$ are simultaneously diagonalizable or, equivalently, commute with $S\.T=T\.S$  \cite{agn_thiel2019empirical}.}

Back to our example
\begin{equation}
	W(U)=-\log\det U\qquad\longrightarrow\qquad \D W(U)=-\frac{1}{\det(U)}\Cof(U)=-U^{-1}\,,
\end{equation}
we see that the isotropic tensor function $U\to-U^{-1}$ is strictly monotone:
\begin{equation}
	\iprod{-(U^{-1}-V^{-1}),\,U-V}>0\qquad\forall\;U\neq V\in\Symp(n)\,,
\end{equation}
because the corresponding vector field
\begin{equation}
	\fpot(\lambda_1,\dotsc,\lambda_n)=-(\log\lambda_1+\dotsc+\log\lambda_n)\qquad\longrightarrow\qquad \grad\fpot(\lambda_1,\dotsc,\lambda_n)=-\left(\frac{1}{\lambda_1},\dotsc,\frac{1}{\lambda_n}\right)^T
\end{equation}
is (strictly) monotone since
\begin{equation}
	\sum_{i=1}^n-\left(\frac{1}{\lambda_i}-\frac{1}{\mu_i}\right)(\lambda_i-\mu_i)>0\qquad\forall\;\lambda\neq\mu\in\R^n.
\end{equation}
 Hill's result \cite{hill1968constitutive} and the Davis theorem express the equivalence of these two monotonicity conditions \eqref{eq:introductionMonotonicity} and \eqref{eq:introductionMonotonicity2}, cf.\ Proposition \ref{lemma:mono1}. We refer to this as the \enquote{potential case}.
In this paper we will consider the \enquote{non-potential case} in the following sense:

The derivative $\D W$ naturally generalizes to the isotropic tensor function
\begin{align}
	\Sigma_f\col\Sym(M)\to\Symn\,,\quad \Sigma_f(\underbrace{Q^T\.\diag(\lambda_1,\dotsc,\lambda_n)\. Q}_{S\in\Sym(n)}) = \underbrace{Q^T\.\diag(f(\lambda_1,\dotsc,\lambda_n))\. Q}_{\Sigma_f(S)\in\Sym(n)} \quad\forall\;Q\in\On\label{eq:introductionMatrixFunction}
\end{align}
with a symmetric vector function $f=(f_1,\dotsc,f_n)\col M\subset\R^n\to\R^n$, where $\Sym(M)\subset\Sym(n)$ denotes the set of all symmetric $n\times n$ matrices with eigenvalues in $M$, $\On$ is the orthogonal group and $\diag(\lambda_1,\dotsc,\lambda_n)$ is the diagonal matrix with diagonal entries $\lambda_1,\dotsc,\lambda_n$. Conversely, the eigenvalue representation of an isotropic tensor function is permutation symmetric.

Going back to Definition \ref{def:monotonicity}, we know call the monotonicity of a vector field $f \col M\subset \R^n\to\R^n$ as \emph{vector-monotonicity} and write:
\begin{definition}[Matrix-monotonicity]
	A symmetric function $f\col M\subset\R^n\to\R^n$ is called \emph{matrix-monotone} if
	\begin{equation}
		\iprod{\Sigma_f(S)-\Sigma_f(T),\, S-T} \geq 0\qquad\forall\;S,T\in\Sym(M)\subset\Symn
	\end{equation}
	where $\iprod{\cdot,\cdot}$ denotes the canonical inner product on $\Rnn$. We call $f$ \emph{strictly matrix-monotone} if the inequality is strict for all $S\neq T$.
\end{definition}
 It is easy to see that matrix-monotonicity implies vector-monotonicity by inserting only diagonal matrices in \eqref{eq:introductionMatrixFunction}:
\begin{lemma}
	\label{lemma:matrixMonotonicityImpliesVectorMonotonicity}
	Let $f\col M\subset\R^n\to\R^n$ be a (strictly) matrix-monotone symmetric function. Then $f$ is (strictly) vector-monotone.
\end{lemma}
\begin{proof}
	Let $f$ be matrix-monotone. For $x,y\in M\subset\R^n$, choose $S=\diag(x)$ and $T=\diag(y)$. Then
	\begin{equation}
		\iprod{f(x)-f(y),\, x-y} = \iprod{\Sigma_f(S)-\Sigma_f(T),\, S-T} \geq 0\,.
	\end{equation}
	If $f$ is strictly matrix-monotone, then the inequality is strict for $x\neq y$.
\end{proof}
The reverse implication is known to be true in case that $f=\grad\fpot$ (the potential case as seen above). But what about the non-potential case, i.e.\ with $f$ not necessarily of the form $\grad\fpot$?

We have extensively checked with a number of colleagues from mathematics, mechanics and engineering as to whether it is known that vector-monotonicity implies matrix-monotonicity. We did not get any hint in this direction and the general feeling was that such an implication might not be true in the general non-potential case. Indeed, we started out trying to find a counterexample to the mentioned implication. Since this did not work out, we re-analyzed the situation and found a general proof independent of any mechanical context. Our Theorem \ref{theorem:mainResult}:
\begin{framed}
	\textit{A symmetric function $f\col M\subset\R^n\to\R^n$ is (strictly) vector-monotone if and only if it is (strictly) matrix-monotone.}
\end{framed}
is the main result of this contribution.

\textit{In a late stage of preparing this manuscript, we have found in the otherwise famous paper by \citet{hill1970constitutive} a somewhat hidden similar statement in the context of nonlinear elasticity. Upon first reading, we came up with a counterexample to Hill's claim based on a literal interpretation of his text. However, in light of the proof of our main Theorem \ref{theorem:mainResult} it is possible to clarify Hill's attempt. We will delineate both readings of Hill's analysis and present our genuine proof of the general result.}

Let us next turn to where and how the question of matrix-monotonicity versus vector-monotonicity appears in nonlinear elasticity.
	
%
%
%
\subsection{Motivation: ideal isotropic nonlinear elasticity -- stress increases with strain and the elastic Drucker postulate}\label{sec:motivation}

In isotropic linear elasticity, it is customary to say that \enquote{stress increases with strain} as the Cauchy stress is a strongly  monotone function of the infinitesimal symmetric strain tensor\footnote{Also sometimes written as $\langle \mathrm d\sigma, \mathrm d\varepsilon\rangle>0$ or $\langle\dot{\sigma},\dot{\varepsilon}\rangle>0$ and referred to as \textbf{Drucker-stability} condition.   } $\varepsilon=\sym \D u$, i.e.
\begin{align}\label{Css}
	\sigma:\Sym(3)\to \Sym(3)\,,\qquad 	\iprod{\sigma(\eps)-\sigma(\overline\eps), \eps-\overline\eps}\geq c^+\norm{\eps-\overline\eps}^2,\qquad \qquad \eps\neq \overline\eps
\end{align}	
for some $c^+>0$, provided the shear modulus $\mu>0$ and the bulk modulus $\kappa=\frac{2\mu+3\lambda}{3}>0$. This condition also implies stable material behaviour and uniqueness of the linear boundary value problem. Moreover, $\varepsilon\mapsto \sigma(\varepsilon)$ is invertible. In linear elasticity\footnote{In nonlinear elasticity the Cauchy stress is not in general a potential with respect to the chosen finite-strain variable.} the Cauchy stress is naturally the derivative of the quadratic energy with respect to infinitesimal strain, i.e.\ $\sigma=\D_{\varepsilon} W_{\rm lin}(\varepsilon)$.
Thus condition \eqref{Css} expresses nothing else but the strict convexity of $\varepsilon\mapsto W(\varepsilon)$. Since $W$ is a rotationally invariant function with $W_{\rm lin}(Q^T\.\varepsilon\.Q)=W_{\rm lin}(\varepsilon)$ for all $Q\in {\rm O}(3)$ due to isotropy, we may write $W(\varepsilon)=\fpot(\varepsilon_1, \varepsilon_2, \varepsilon_3)$ as a function of the principal strains (the eigenvalues of the strain tensor $\varepsilon$). According to the Chandler Davis theorem \cite{davis1957all,Lewis03,Lewis96,Lewis96b}, $\varepsilon\mapsto W(\varepsilon)$ is convex if and only if $(\varepsilon_1, \varepsilon_2, \varepsilon_3)\mapsto \fpot(\varepsilon_1, \varepsilon_2, \varepsilon_3)$ is convex. The principal stresses (eigenvalues of $\sigma$) are obtained as $\sigma_i=\partial_{\varepsilon_i}\fpot(\varepsilon_1, \varepsilon_2, \varepsilon_3)$. Therefore, the monotonicity  restricted to some symmetric stress measure in principal stresses versus principal strains is equivalent to the monotonicity of Cauchy stresses versus strains. Note that this argument needs the existence of a potential for the given stress tensor (the potential case).
	
Well known and often used ``potential'' pairings in nonlinear elasticity are the second Piola-Kirchhoff stress $S_2=2\.\D_C \. W(C)$, the Biot-stress tensor $T_{\rm Biot}=\D_U\. W(U)$, and the Kirchhoff stress tensor $\tau=\D_{\log V}W(\log V)$ in terms of the strain tensors $C, U, \log V,$ respectively, where we have the left and right polar-decomposition \cite{neff2014grioli,neff2014logarithmic} of $F\in\GLp(3)$ as
\begin{align}
F=V\. R\qquad\text{and}\qquad F=R\, U\,,\qquad V\,,\;U\in\Symp(3)\,,\quad R\in\SO(3)\\
\text{and}\quad B\colonequals FF^T=V\.R\.R^TV^T=V^2,\qquad C\colonequals F^TF=U^TR^TR\.U=U^2.\notag
\end{align}
Unfortunately, in nonlinear elasticity the situation is not exhausted by these possibilities. A fundamental ambiguity arises as to which stress tensor and which strain tensor should possibly be taken as a replacement of \eqref{Css}. As an example, consider Hill's inequality \cite{hill1968constitutive,hill1970constitutive,wang1973introduction,d2025constitutive} which requires monotonicity of the Kirchhoff-stress $\tau$ in the spatial logarithmic strain-tensor $\log V$, i.e.
\begin{align}\label{Kss}
	\langle \tau(\log V)-\tau(\log \overline V), \log V-\log \overline V\rangle>0\,,\qquad \qquad V\neq \overline V\in \Symp(3)\,.
\end{align}	
Clearly, for isotropic hyperelasticity, this condition is equivalent to the strict convexity of the mapping $\log V\to W(\log V)$ in terms of the strain-tensor\footnote{Ogden's energy \cite{ogden1972large} has been derived conforming with Hill's inequality: it satisfies Hill's inequality in a compact set.} $\log V$ due to
\begin{equation}
	\tau=\D_{\log V}W(\log V)
\end{equation}
and applying the Chandler Davis theorem. However, this condition alone has little or nothing to do with stability since e.g.\ the quadratic Hencky energy \cite{neff2016geometry,hencky1929superpositionsgesetz} satisfies \eqref{Kss}, but is unstable in large strains \cite{agn_neff2015exponentiatedI,jog2013conditions}.
It might be argued that for a useful stability condition, in the spirit of \textbf{Drucker's postulate} \eqref{Css} \cite{clayton2014analysis}, we should always use the Cauchy stress tensor since this stress tensor is most readily measured experimentally - it is also called the true-stress tensor since it appears naturally in the current configuration.
Based on this interpretation, in \cite{jog2013conditions,jog2015continuum,wollner2025search,neff2025hypo,ghiba2026polyconvexity,klein2026polyconvexity,wollner2026concurrent} the condition TSTS-$\rm M^+$ \cite{agn_neff2015exponentiatedI} has been put forward as a new stability condition.\footnote{\textbf{T}rue \textbf{S}tress-\textbf{T}rue \textbf{S}train (TSTS-M$^+$) means the monotonicity of the Cauchy stress tensor as a function of $\log B$ or $\log V,$ which is equivalent, since $\log B=2\log V$.} It reads
\begin{align}\label{TSTS}
	\text{TSTS-M}^+:\qquad\langle \sigma (\log V)-\sigma (\log \overline V), \log V-\log \overline V\rangle >0\,,\qquad \qquad V\neq \overline V\in\Symp(3)\,.
\end{align}
Up to the present date, no elastic energy has been found which would satisfy \eqref{TSTS} everywhere and be Legendre-Hadamard-elliptic in the whole deformation range. A simple example, satisfying \eqref{TSTS} everywhere but which is elliptic only in a large  compact set is the exponentiated Hencky-type energy \cite{agn_neff2015exponentiatedI}
\begin{align}
	F\mapsto \frac{1}{2} e^{\|\log V\|^2}, \qquad\qquad V\in\Symp(3)\,,\quad F\in\GLp(3)\,.
\end{align}
A precursor to the TSTS-M$^+$ condition is Leblond's requirement \cite{leblond1992constitutive} expressing the \enquote{same} monotonicity in terms of principal Cauchy stresses versus principal logarithmic stretches, TSTS$^*$-M$^+$, which reads
\begin{align}
	\text{TSTS*-M}^+\!:\;\.\sum\limits_{i=1}^3[\sigma_i(\log \lambda_1,\log \lambda_2,\log \lambda_3)-\sigma_i(\log \overline{\lambda}_1,\log \overline{\lambda}_2,\log \overline{\lambda}_3)]\,[\log \lambda_i-\log\overline{\lambda}_i]>0\,,\quad\lambda\neq\lambdabar\,.
\end{align}
The equivalence of TSTS$^*$-M$^+$ with TSTS-M$^+$ is not directly clear from the Chandler Davis theorem since $\sigma$ does not have a potential in general. However, the equivalence would simplify the calculation needed to test for TSTS$^*$-M$^+$ instead of TSTS-M$^+$, cf.\ Appendix \ref{appendix:examples}.
In other papers \cite{guo2006application} it is a nonlinear version of \textbf{Drucker's postulate}\footnote{The meaning of Drucker's postulate is ambiguous, since the corresponding stress-strain pair is not a priori clear. The commercial finite element program \textsc{Abaqus} checks Kirchhoff stress versus logarithmic strain. In \citet[eq.(28)]{boyce2000constitutive} we read, however: \enquote{The Drucker stability criterion requires that the tangential stiffness matrix (or Hessian) be positive definite. The Hessian may be written in terms of the strain energy density as $H_{ijkl}=\frac{\partial^2 W}{\partial\eps_{ij}\.\partial\eps_{kl}}$, where $\eps_{ij}$ is the strain.}
This interpretation seems to favour \eqref{Ci}.
In Baaser \cite{baaser2026hyperelastic} it has been clarified that the built–in \textsc{Abaqus} stability check is fully equivalent to Hill’s condition \eqref{Kss} for the incompressible case $\det V=\det\overline V=1$.} 
that is stipulated, which is there meant to mean, in terms of principal Cauchy stresses $\sigma$ versus principal stretches,
\begin{align}\label{ui1}
	\text{TSS*-M}^+:\qquad\sum\limits_{i=1}^3[\sigma_i( \lambda_1, \lambda_2, \lambda_3)-\sigma_i( \overline{\lambda}_1, \overline{\lambda}_2, \overline{\lambda}_3)]\,[ \lambda_i-\overline{\lambda}_i]>0\,,\qquad\qquad\lambda\neq\lambdabar\,.
\end{align}
The corresponding condition in terms of stretch tensors is
\begin{align}\label{TSS}
	 \text{TSS-M}^+:\qquad\langle \sigma (V)-\sigma (\overline V),\,  V- \overline V\rangle >0\,,\qquad \qquad V\neq \overline V\in\Symp(3)\,.
\end{align}
Again, while \eqref{TSS} implies \eqref{ui1} on taking $V,\overline V$ diagonal, the reverse implication is not immediately clear. Another condition in this spirit for $T_{\rm Biot}(U)=\D_U W(U)$ is (see also \cite{ogden1977inequalities})
\begin{align}\label{Ci}
	\langle T_{\rm Biot}(U)-T_{\rm Biot}(\overline U),\, U-\overline U \rangle>0\,,\qquad \qquad U\neq\overline U\in \Symp(3)\,,
\end{align}
which is connected to the {\it Coleman-Noll} inequality \cite{coleman1959thermostatics} in the sense that the Coleman-Noll inequality also implies that the map $U\mapsto W(U)$ is convex in $U\,.$ Another condition for $S_2(C)=2\.\D_C W(C)$ is
\begin{align}\label{Ciu}
	\langle S_2(C)-S_2(\overline C),\, C- \overline C \rangle>0\,,\qquad \qquad C\neq \overline C\in\Symp(3)\,,
\end{align}
which is the requirement that $C\mapsto W(C)$ be convex in $C$. This last requirement is not completely off topic \cite{lehmich2012convexity,spector2015note,silhavy2015convexity}
since there exist energies which satisfy \eqref{Ciu} and are polyconvex while satisfying $W(F)\to \infty$ for $\det F\to 0$ \cite{lehmich2012convexity}.
We also note that \eqref{Css}, \eqref{Kss}, \eqref{Ci} and \eqref{Ciu} specify \textit{work-conjugate} pairs \cite{ogden1997non}, while \eqref{TSTS} and \eqref{TSS} do not.
	
It seems reasonable to require that constitutive conditions in isotropic nonlinear elasticity should be independent of whether they are written in principal-stress/principal-strain formulation versus tensorial stress–strain formulation.
In this paper, we will show that monotonicity conditions and invertibility of stresses versus strains indeed satisfy this requirement.
	
%
%
%
%
\section{Monotonicity of isotropic tensor functions}

This paper is now structured as follows. In the following section, we characterize isotropic tensor functions in the light of vector- and matrix-monotonicity.
Then we illustrate Hill's clarified proof for the potential case.
Last, we present our own proof for the equivalence in the non-potential case.

In addition, we give further insight into Hill's interpretation of the potential case in a list of appendices:
\begin{itemize}
	\item Appendix \ref{appendix:OriginalProofHill} shows our first reading of Hill's proof limited to the set of ordered eigenvalues.
	\item Appendix \ref{appendix:ImprovedInterpretationHill} details an (improved) interpretation for non-ordered eigenvalues.
	\item Appendix \ref{appendix:HillOgdenProof} states Ogden's interpretation of Hill's proof again on the subset of ordered eigenvalues.
\end{itemize}

Overall, our verdict on Hill's work is that his original theorem is, of course, correct. However, its proof is presented in a remarkably condensed form that leaves several essential arguments to the reader. In particular, the ordering of the vectors is ambiguous.
We will now present a new, complete, and self-contained proof that we hope makes the relation between vector and matrix-monotonicity more transparent and accessible.
\begin{theorem}[Chandler Davis theorem \cite{davis1957all,ball1976convexity}]
	\label{theorem:DavisLewis}
	Let $W\col \Symn\to\R$ be isotropic, i.e.
	\begin{equation}
		W(Q^TSQ) = W(S) \qquad\forall\; S\in\Symn\,,\quad Q\in\On\,.
	\end{equation}
	Then $W$ is convex on $\Symn$ if and only if its representation in terms of the eigenvalues, i.e.\ the uniquely determined symmetric function $\fpot\col\R^n\to\R$ with
	\begin{equation}
		W(S) = \fpot(\lambda_1,\dotsc,\lambda_n)
	\end{equation}
	for all $S\in\Symn$ with eigenvalues $\lambda_1,\dotsc,\lambda_n$, is convex on $\R^n$.
\end{theorem}
A convexity criterion analogous to the Chandler Davis Theorem but for the ordered eigenvalues was given by Friedland \cite[Theorem 2.2]{friedland1981convex}.
\begin{theorem}\label{theorem:friedlandOrderedConvexity}
	Let $W\col \Symn\to\R$ be isotropic and let $\psi\col \Rsort\to\R$ be continuously differentiable with $W(S) =  \psi(\sort{\lambda}_1,\dotsc,\sort{\lambda}_n)$. Then $W$ is convex if and only if $\psi$ is convex and
	\begin{equation}
		\pdd{\psi}{x_1}(x) \geq \ldots \geq \pdd{\psi}{x_n}(x)\qquad\forall\;x\in\Rsort\,.
	\end{equation}
	Furthermore, $W$ is strictly convex if and only if $\psi$ is strictly convex and
	\begin{equation}
		\pdd{\psi}{x_i}(x) > \pdd{\psi}{x_j}(x) \qquad\text{if}\quad x_i>x_j\,.
	\end{equation}
\end{theorem}
Next, we continue on Section \ref{sec:definitions} with some preliminary statements that we need to prove our main result.

%
%
%
\subsection{Preliminaries}\label{sec:Preliminaries}

In Section \ref{sec:3representations}, we have introduced the three different representations of isotropic scalar- and tensor-valued functions on $\Sym(n)$ in terms of matrices, (unordered) eigenvalues and ordered eigenvalues. We will now state their equivalence:
\begin{lemma}\label{lemma:existenceAndUniquenessOfMatrixFunction}
	Let $f=(f_1,\dotsc,f_n)\col M\subset\R^n\to\R^n$ be symmetric. Then there exists a unique tensor function $\Sigma_f\col\Sym(M)\subset\Sym(n)\to\Symn$ with
	\begin{equation}
		\Sigma_f(\underbrace{Q^T\.\diag(\lambda_1,\dotsc,\lambda_n)\. Q}_{S\in\Sym(M)}) = \underbrace{Q^T\.\diag(f(\lambda_1,\dotsc,\lambda_n))\. Q}_{\Sigma_f(S)\in\Sym(n)} \qquad\forall\; Q\in\On\,.
	\end{equation}
\end{lemma}
\begin{proof}
	See \cite{silhavy2013mechanics}.
\end{proof}
In general, the reverse of Lemma \ref{lemma:existenceAndUniquenessOfMatrixFunction} is also true: for every isotropic function $\sigma\col\Symn\to\Symn$ (i.e.\ any function of the form \ref{eq:isotropic}), there exists a vector function $f\col\R^n\to\R^n$ (commonly denoted by $\sigmahat=f$) such that
$\sigma=\Sigma_f$. For $n=3$, this follows from the unique
characterization \cite[p.471(13.9)]{antman} 
\begin{align}
	\sigma(S)=\beta_0(I_1,I_2,I_3)\.\id+\beta_1(I_1,I_2,I_3)\.S+\beta_2(I_1,I_2,I_3)\.S^2\label{eq:PrincipalMatrixInvariants}
\end{align}
of $\sigma$ due to the Rivlin-Ericksen representation theorem \cite{rivlin1997stress,richter1948isotrope} with the principal matrix invariants
\begin{align}
	I_1&=\tr(S)\,, &I_2&=\frac{1}{2}[(\tr\,S)^2-\,\tr(S^2)]=\tr(\Cof S)\,, &I_3&=\det S\,,\\
	I_1&=s_1+s_2+s_3\,, &I_2&=s_1\.s_2+s_1\.s_3+s_2\.s_3\,, &I_3&=s_1\.s_2\.s_3\label{eq:InvariantsInEigenvalues}
\end{align}
and scalar-valued functions $\beta_0,\beta_1,\beta_2\col\R^3\to\R$ with $s_1,s_2,s_3$ as the eigenvalues of $S$. Taking into consideration \eqref{eq:InvariantsInEigenvalues}, we may express \eqref{eq:PrincipalMatrixInvariants} also as
\begin{align}
	\sigma(S)=\alpha_0(s_1,s_2,s_3)\.\id+\alpha_1(s_1,s_2,s_3)\.S+\alpha_2(s_1,s_2,s_3)\.S^2\label{eq:SigmaInPrincipalMatrixInvariants}
\end{align}
with symmetric scalar-valued functions $\alpha_0,\alpha_1,\alpha_2\col\R^3\to\R\,.$ For diagonal matrices $S=D=\diag(d_1,d_2,d_3)$ it follows
\begin{align}
	\sigma(D)=\alpha_0(d_1,d_2,d_3)\.\id+\alpha_1(d_1,d_2,d_3)\.\diag(d_1,d_2,d_3)+\alpha_2(d_1,d_2,d_3)\.\diag(d_1^2,d_2^2,d_3^2)
\end{align}
and we can define the vector function $\sigmahat\col\R^3\to\R^3$ as
\begin{align}
	\sigmahat_i(d_1,d_2,d_3)=\alpha_0(d_1,d_2,d_3)+\alpha_1(d_1,d_2,d_3)\.d_i+\alpha_2(d_1,d_2,d_3)\.d_i^2\,.\label{eq:definitionOfSigmaF}
\end{align}
The function $\sigmahat$ is symmetric in terms of Definition \ref{definition:Symmetric}  because for each permutation $\pi\col\{1,2,3\}\to\{1,2,3\}$
\begin{align}
	\sigmahat_i(d_{\pi(1)},d_{\pi(2)},d_{\pi(3)})&=\alpha_0(d_{\pi(1)},d_{\pi(2)},d_{\pi(3)})+\alpha_1(d_{\pi(1)},d_{\pi(2)},d_{\pi(3)})\.d_{\pi(i)}+\alpha_2(d_{\pi(1)},d_{\pi(2)},d_{\pi(3)})\.d_{\pi(i)}^2\notag\\
	&=\alpha_0(d_1,d_2,d_3)+\alpha_1(d_1,d_2,d_3)\.d_{\pi(i)}+\alpha_2(d_1,d_2,d_3)\.d_{\pi(i)}^2=\sigmahat_{\pi(i)}(d_1,d_2,d_3)\,.
\end{align}

The next two statements follow directly from Lemma \ref{lemma:existenceAndUniquenessOfMatrixFunction}.
\begin{lemma}
	Let $f=(f_1,\dotsc,f_n)\col M\subset\R^n\to\R^n$ be symmetric. Then $\Sigma_f\col\Sym(M)\subset\Sym(n)\to\Symn$ is \emph{isotropic}, i.e.
	\begin{equation}
		\Sigma_f(Q^T S\.Q) = Q^T \Sigma_f(S)\.Q
	\end{equation}
	for all $S\in\Sym(M)\subset\Sym(n)$ and all $Q\in\On$.
\end{lemma}
\begin{lemma}
	Let $\Sigma\col\Sym(M)\subset\Sym(n)\to\Symn$ be an isotropic tensor function. Then there exists a unique function $f\col M\subset\R^n\to\R^n$ such that
	\begin{equation}
		\Sigma(Q^T\diag(\lambda_1,\dotsc,\lambda_n)\.Q) = Q^T\diag(f(\lambda_1,\dotsc,\lambda_n))\.Q
	\end{equation}
	for all $Q\in\On$ and all $(\lambda_1,\dotsc,\lambda_n)\in M$. The function $f$ is symmetric due to the non-uniqueness of $\diag(\lambda_1,\dotsc,\lambda_n)$. Different $Q\in\On$ can cause a permutation in the order of the eigenvalues.
\end{lemma}
\begin{remark}\label{remark:identifySigmaWithF}
	As a result of the last two lemmas we can identify $f\longleftrightarrow\Sigma_f$.
\end{remark}
\begin{lemma}
	Let $f=(f_1,\dotsc,f_n)\col M\subset\R^n\to\R^n$ be symmetric. Then there exists a unique function $\phi\col \Msort\subset\Rsort\to\R^n$ with
	\begin{equation}
		\phi(\lambda_1,\dotsc,\lambda_n)=f(\lambda_1,\dotsc,\lambda_n)\qquad\forall\;\lambda\in \Msort\subset\Rsort\,.
	\end{equation}
	Furthermore, $\phi_i(x)=\phi_j(x)$ for all $x\in \Msort\subset\Rsort$ with $x_i=x_j$ due to the symmetry of $f$.
\end{lemma}
\begin{lemma}\label{lemma:existenceOfsymmetricExpansion}
	Let $\phi\col \Msort\subset\Rsort\to\R^n$ such that $\phi_i(x)=\phi_j(x)$ for all $x\in \Msort\subset\Rsort$ with $x_i=x_j$. Then there exists a unique symmetric function $f=(f_1,\dotsc,f_n)\col M\subset\R^n\to\R^n$ with
	\begin{equation}
		f(\lambda_1,\dotsc,\lambda_n)=\phi(\lambda_1,\dotsc,\lambda_n)\qquad\forall\;\lambda\in \Msort\subset\Rsort\,.
	\end{equation}
\end{lemma}
\begin{proof}
	We can extend $f\col M\subset\R^n\to\R^n$ as
	\begin{equation}
	f(\lambda_1,\dotsc,\lambda_n)=\left(\phi_{\pi^{-1}(i)}\left(\left(\lambda_{\pi(j)}\right)_j\right)\right)_i
	\end{equation}
	with $\pi\col\{1,\dotsc,n\}\to\{1,\dotsc,n\}$ a permutation such that $\left(\lambda_{\pi(j)}\right)_j\in \Msort\subset\Rsort.$ The permutation is unique up to points where two components of $\lambda$ are equal. Because of the compatibility condition \eqref{eq:orderPreservation}, $f$ is a symmetric function on $M\subset\R^n$.
\end{proof}
\begin{corollary}
	As a result of the last two lemmas we can identify $f\longleftrightarrow \phi$ and with Remark \ref{remark:identifySigmaWithF} also $\phi\longleftrightarrow f\longleftrightarrow\Sigma_f$. We will write $\Sigma_{f_\phi}$ or just $\Sigma^\phi$ for the isotropic tensor function defined by $\phi$.
\end{corollary}
\begin{lemma}
\label{lemma:orderedVectorFunctionInducesTensorFunction}
	Let $\phi\col\Rsort\to\R^n$ such that $\phi_i(x)=\phi_j(x)$ for all $x\in\Rsort$ with $x_i=x_j$. Then there exists a unique isotropic function $\Sigma^\phi \col\Symn\to\Symn$ such that
	\begin{equation}
		\Sigma^\phi (Q^T\diag(\lambda_1,\dotsc,\lambda_n)\. Q) = Q^T\diag(\phi(\lambda_1,\dotsc,\lambda_n))\. Q
	\end{equation}
	for all $Q\in\On$ and all $\lambda_1\geq\ldots\geq\lambda_n$.
\end{lemma}
\begin{proof}
Due to Lemma \ref{lemma:existenceAndUniquenessOfMatrixFunction}, it is sufficient to note that there exists a unique symmetric function $f\col\R^n\to\R^n$ such that $f(x)=\phi(x)$ for all $x\in\Rsort$ and we can choose $\Sigma^\phi =\Sigma_f$.
\end{proof}
Note that, compared to the representation by vector valued functions on $\R^n$, the requirement of symmetry is replaced by the simpler requirement $x_i=x_j\implies \phi_i(x)=\phi_j(x)$ for a function $\phi$ defined only on the set $\Rsort$ to induce an isotropic tensor-valued function on $\Symn$.

Based on our initial Definition \ref{def:monotonicity} of vector-monotonicity, we elaborate in the context of sorted values.
\begin{lemma}\label{lemma:vectorMonotonicityImpliesSameAlgebraicOrder}
	Let $f\col\R^n\to\R^n$ be symmetric and vector-monotone. Then, the vector $x\in\R^n$ and $f(x)\in\R^n$ are sorted in the same algebraic order:
	\begin{align}
		x_1\geq x_2\geq\dotsc\geq x_n\qquad\implies\qquad f_1(x)\geq f_2(x)\geq\dotsc\geq f_n(x)\,.
	\end{align} 
	If $f$ is strictly vector-monotone, then
	\begin{equation}
		x_i\neq x_j\qquad\implies\qquad(x_i-x_j)(f_i(x)-f_j(x))>0\qquad\forall\; x\in\R^n.
	\end{equation}
\end{lemma}
\begin{proof}
	Let $i,j\in\{1,\dotsc,n\}$ be arbitrary indices with $i\neq j$. For any $x\in\R^n$ we choose $y\in\R^n$ as
	\begin{align}
		y_k=x_k\qquad\forall\;k\neq i,j\,,\qquad y_i=x_j\quad\text{and}\quad y_j=x_i\,.
	\end{align}
	Vector-monotonicity leads to
	\begin{align}
		\iprod{f(x)-f(y),\, x-y}=\sum_{k=1}^n(f_k(x)-f_k(y))(x_k-y_k)=2\.(f_i(x)-f_j(x))(x_i-x_j)\geq 0\,.
	\end{align}
	If $x_i\geq x_j$ then $f_i(x)\geq f_j(x)$ and if $x_i\leq x_j$ then $f_i(x)\leq f_j(x)$, for all indices $i\neq j$.
	Therefore, the vector $f(x)\in\R^n$ is sorted in the same algebraic order as $x\in\R^n$. Last, for all $x\in\R^n$ where $x_i\neq x_j$ we have $x\neq y$ and it holds
	\begin{equation}
		(x_i-x_j)(f_i(x)-f_j(x))=\frac12\.\iprod{f(x)-f(y),\, x-y}>0
	\end{equation}
	if $f$ is strictly vector-monotone.
\end{proof}
\begin{remark}
	As seen in \eqref{eq:definitionOfSigmaF}, for applications in isotropic nonlinear elasticity the vector function $f=(f_1,\dotsc,f_n)\col M\subset\R^n\to\R^n$ will often represent the vector field $\sigmahat$ of principal Cauchy stresses as functions of the principal stretches
	\begin{equation}
		\sigmahat\col\R_+^3\to\R^3,\qquad\sigmahat\col(\lambda_1,\lambda_2,\lambda_3)\mapsto(\sigmahat_1,\sigmahat_2,\sigmahat_3)\,.
	\end{equation}
	The corresponding matrix function will be the Cauchy stress tensor field $\sigma\col\Symp(3)\to\Sym(3)$. Through the isotropy of $\sigma$ we can take 
	\begin{equation}
		\sigma(V)=\sigma(Q^T\diag(\lambda_1,\lambda_2,\lambda_3)\.Q)=Q^T\sigma(\diag(\lambda_1,\lambda_2,\lambda_3))\.Q=Q^T\diag(\sigmahat(\lambda_1,\lambda_2,\lambda_3))\.Q\,,
	\end{equation}
	with $\lambda_1,\lambda_2,\lambda_3$ the eigenvalues of $V=Q^T\diag(\lambda_1,\lambda_2,\lambda_3)\.Q\,.$ 
\end{remark}
We will introduce a short well-known Lemma first mentioned in 1934 by G.\ H.\ Hardy \cite{hardy1952inequalities} as well as later found in Theobald \cite{theobald1975inequality} and Richter \cite{richter1958abschatzung}.
\begin{lemma}\label{lemma:hardy}
	Let $x,y\in\R^n$, then
	\begin{equation}
		\colorbox{dortmundgreen!50}{$\displaystyle\sum_{i=1}^n\sort{x}_i\.\sort{y}_i$}
		\geq\sum_{i=1}^nx_i\.y_i\geq
		\colorbox{red!30}{$\displaystyle\sum_{i=1}^n\sort{x}_i\.\sort{y}_{n+1-i}$}
		\,.\label{eq:hardy}
	\end{equation}
	If $x$ and $y$ are similarly ordered and have identical equality blocks, with strict ordering between distinct blocks, equality in the upper rearrangement inequality \eqref{eq:hardy} holds exactly for permutations preserving those blocks.
\end{lemma}
G.\ H.\ Hardy also gave a very intuitive visualization:\footnote{Because of the associative law, it is sufficient to leave $x$ unchanged and only arrange the order of the vector $y$ so that it corresponds to the order of the vector $x$.}
\enquote{\textit{The theorem becomes 'intuitive' if we interpret the $x$ as distances along a rod to hooks and the $y$ as weights suspended from the hooks. To get the maximum statical moment with respect to an end of the rod, we hang the heaviest weights on the hooks farthest from that end.}} \cite[10.2]{hardy1952inequalities}, cf.\ Figure \ref{fig:Hardy}.
\begin{figure}[h!]
	\centering
	\begin{tikzpicture}[scale=1,
		weight/.style={draw, rounded corners=2pt, minimum width=0.65cm,
			minimum height=1.0cm, fill=white},
		every node/.style={font=\small}
		]
		\def\tilt{0.1}     
		\def\length{5.5}  
		\coordinate (P) at (0,0);
		
		\coordinate (L) at (-\length,{-\length*\tilt});
		\coordinate (R) at ( \length,{ \length*\tilt});
		\draw[line width=1.2pt] (L)--(R);
		\fill (P) circle (2.8pt);
		\node[above=3pt] at (P) {pivot};
		
		\draw[fill=gray!25]
		(-0.35,-2.2)--(0,-0.05)--(0.35,-2.2)--cycle;
		\draw[fill=gray!25]
		(-0.8,-2.2) rectangle (0.8,-2.5);
		
		\foreach \x/\xlabel/\ylabel/\kg/\farbe in
		{
			-5/x_1/y_1/3/dortmundgreen!50,
			-3.5/x_2/y_2/2/dortmundgreen!50,
			-2/x_3/y_3/1/dortmundgreen!50,
			2/x_3/y_1/3/red!30,
			3.5/x_2/y_2/2/red!30,
			5/x_1/y_3/1/red!30
		}{
			\pgfmathsetmacro\yy{\tilt*\x} 
			\fill (\x,\yy) circle (1.8pt); 
			\node[above=4pt] at (\x,\yy) {$\xlabel$}; 
			\draw (\x,\yy)--(\x,\yy-1.2); 
			\node[weight,align=center,fill=\farbe] (w) at (\x,\yy-1.5)
			{$\ylabel$\\$\kg\.$kg}; 
		}
		
		\node[align=center] at (-4,1.0)
		{
			{ \normalsize optimal ordering}\\[0.2em]
			$\matr{\text{heavier weights}\\\text{farther from pivot}}$
		};
		\node[align=center] at (4,1.6)
		{
			{ \normalsize sub-optimal ordering}\\[0.2em]
			$\matr{\text{heavier weights}\\\text{closer from pivot}}$
		};
		
		\draw[decorate,decoration={brace,mirror,amplitude=5pt}]
		(-5.4,-2.8)--(-1.4,-2.8);
		\node at (-3.4,-3.3) {larger moment};
		\draw[decorate,decoration={brace,mirror,amplitude=5pt}]
		(1.4,-2.8)--(5.4,-2.8);
		\node at (3.4,-3.3) {smaller moment};
		
		\draw[->,line width=1pt]
		(1.0,0.4) to[out=20,in=110] (-1.0,0.2);
		\node[align=center] at (0,1.1)
		{tips toward the side\\with larger moment};
		\node[align=center] at (0,-3)
		{$y_1\geq y_2\geq y_3$};
	\end{tikzpicture}
	\caption{Mechanical interpretation of Lemma \ref{lemma:hardy} with $x_i$
		as distances from the pivot and $y_i$ as suspended weights. Placing the largest weights at the largest distances maximizes the total moment $\sum_i x_i y_i$.}\label{fig:Hardy}
\end{figure}
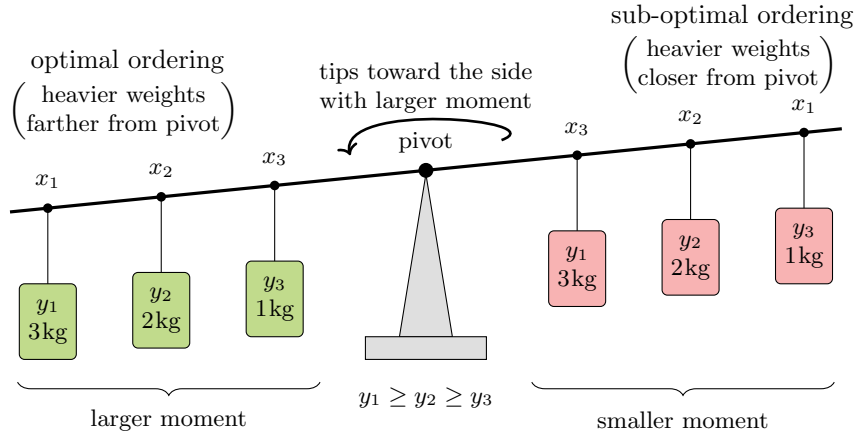

%
%
%
%
\subsection{Hill's proof in the potential case}
\begin{remark}
	The main goal of this work is to prove the reversal of Lemma \ref{lemma:matrixMonotonicityImpliesVectorMonotonicity}, i.e.\ that vector-monotonicity automatically implies matrix-monotonicity. The most natural approach to prove matrix-monotonicity would be to show the inequality chain
		\begin{align}
			\iprod{\Sigma_f(S)-\Sigma_f(T),\, S-T}\geq\iprod{f(x)-f(y),\,x-y}\geq 0\,.\label{eq:ineqChain}
		\end{align}
	which does not hold in general, see for more detail Appendix \ref{appendix:OriginalProofHill}.
\end{remark}
Nevertheless, the first observation that vector-monotonicity might imply matrix-monotonicity in nonlinear isotropic elasticity is due to \citet{hill1968constitutive}, independent of the earlier paper by \citet{davis1957all} on convexity.
\begin{remark}\label{remark:HillPotential}
	\textbf{[Hill's proof for the differentiable potential case]}\\
	We briefly review the correct proof from \citet[p.238 eq.(25)]{hill1968constitutive} for the potential case. We assume that $\sigma=\D_\eps W(\eps)$ with an arbitrary isotropic convex differentiable potential $W\col\Sym(n)\to\R\,.$  As seen we can represent any isotropic energy potential $W(\eps)=\fpot(\eps_1,\dotsc,\eps_n)=\fpot(\eps_i)\,.$ Then $\sigma_i(\eps_1,\dotsc,\eps_n)=\partial_{\eps_i}\fpot(\eps_1,\dotsc,\eps_n)\,.$
	
	 We sketch Hill's proof of the result that $\fpot\col\R^n\to\R$ being a convex symmetric function implies that $W\col\Sym(n)\to\R$ is convex, too. For convexity of a differentiable potential, we have to show that
	\begin{equation}
		W(\overline\eps)-W(\eps)\geq\iprod{\D W(\eps),\overline\eps-\eps}
	\end{equation}
	holds\footnote{Here, the tensors $\eps$ and $\overline\eps$ are representatives of some nonlinear strain tensors, i.e.\ $\eps=\log V$ or $\eps=V-\id\,.$} for all $\eps,\overline\eps\in\Sym(n)\,.$ We start by writing
	\begin{align}
		W(\overline\eps)-W(\eps)=\fpot(\overline\eps_i)-\fpot(\eps_i)=\fpot({\widetilde\eps}_i)-\fpot(\eps_i)\,,\label{eq:IsotropyEigenvaluesSort}
	\end{align}
	with ${\widetilde\eps}_i$ being the eigenvalues $\overline\eps_i\,,$ but sorted in the same algebraic order as $\eps_i$.
	For $\fpot(\overline\eps_i)=\fpot({\widetilde\eps}_i)$ in \eqref{eq:IsotropyEigenvaluesSort} we take advantage of the isotropy of $W$ which implies that $\fpot$ is permutation invariant (symmetric).\footnote{Hill actually starts by assuming $\eps$ and $\overline\eps$ being given in the same algebraic order, \citet[p.238, last sentence]{hill1968constitutive}. Therefore, his proof is, strictly speaking, only valid for that case.} The vector convexity of $\fpot(\eps_i)$ leads then to
	\begin{equation}
		\fpot({\widetilde\eps}_i)-\fpot(\eps_i)\geq\sum_{i=1}^n\partial_{\eps_i}\fpot(\eps_i)({\widetilde\eps}_i-\eps_i)
		=\sum_{i=1}^n\partial_{\eps_i}\fpot(\eps_i)\.{\widetilde\eps}_i-\sum_{i=1}^n\partial_{\eps_i}\fpot(\eps_i)\.\eps_i=\sum_{i=1}^n\partial_{\eps_i}\fpot(\eps_i)\.{\widetilde\eps}_i-\iprod{\D W(\eps),\eps}\,.
	\end{equation}
	Next, we use Lemma \ref{lemma:hardy} to estimate\footnote{With $\sigma=\D_\eps W(\eps)$ and $\sigma_i=\partial_{\eps_i}\fpot(\eps_i)$, we can write $\displaystyle\iprod{\sigma,\overline\eps}\leq\sum_{i=1}^n\sort\sigma_i\.\sort{\overline\eps}_i=\sum_{i=1}^n\sigma_i\.{\widetilde\eps}_i$, cf.\ Appendix \ref{appendix:ImprovedInterpretationHill}.} \cite{hardy1952inequalities, theobald1975inequality, richter1958abschatzung}
	\begin{align}
		\sum_{i=1}^n\partial_{\eps_i}\fpot(\eps_i)\.{\widetilde\eps}_i\geq\iprod{\D W(\eps),\overline\eps}\,,\label{eq:HillPotential1}
	\end{align}
	for which we have to show that $\partial_{\eps_i}\fpot(\eps_i)=:\sigma_i$ and ${\widetilde\eps}_i$ are sorted in the same algebraic order. Because ${\widetilde\eps}_i$ and $\eps_i$ are ordered in the same algebraic order due the definition of ${\widetilde\eps}_i$, it is sufficient to show that $\partial_{\eps_i}\fpot(\eps_i)=:\sigma_i$ and $\eps_i$ are sorted in the same algebraic order, see Lemma \ref{lemma:vectorMonotonicityImpliesSameAlgebraicOrder}. To this end, Hill uses the monotonicity of $\sigma_i(\eps_i)=\partial_{\eps_i}\fpot(\eps_i)$ due to the convexity of $\fpot$ which implies
	\begin{align}
		\sum_{i=1}^n(\sigma_i'-\sigma_i)(\eps_i'-\eps_i)=\sum_{i=1}^n(\partial_{\eps_i}\fpot(\eps_i')-\partial_{\eps_i}\fpot(\eps_i))(\eps_i'-\eps_i)\geq 0\label{eq:SigmaAndEpsSortedInTheSameOrder}
	\end{align}
	where $\eps_i'$ is an arbitrary permutation from $\eps_i$ by simply exchanging two specific indices. For instance, with $(\eps_1',\eps_2',\eps_3')=(\eps_2,\eps_1,\eps_3)$, we get
	\begin{align}
		2\.(\sigma_2-\sigma_1)(\eps_2-\eps_1)\geq0\,.\label{eq:SigmaAndEpsSortedInTheSameOrder2}
	\end{align}
	Therefore $\sigma_i$ and $\eps_i$ are ordered in the same way. By construction in \eqref{eq:IsotropyEigenvaluesSort}, $\eps_i$ and $\widetilde\eps_i$ are already sorted in the same algebraic order, therefore \eqref{eq:HillPotential1} holds. 
	Finally we get
	\begin{align}
		W(\overline\eps)-W(\eps)&=\fpot(\overline\eps_i)-\fpot(\eps_i)=\fpot({\widetilde\eps}_i)-\fpot(\eps_i)\notag\\
		&\geq\sum_{i=1}^n\partial_{\eps_i}\fpot(\eps_i)({\widetilde\eps}_i-\eps_i)=\sum_{i=1}^n\partial_{\eps_i}\fpot(\eps_i)\.{\widetilde\eps}_i-\iprod{\D W(\eps),\eps}\\
		&\geq\iprod{\D W(\eps),\overline\eps}-\iprod{\D W(\eps),\eps}=\iprod{\D W(\eps),\overline\eps-\eps}\notag
	\end{align}
	for arbitrary $\eps,\overline\eps\in\Sym(n).$ With regard to our monotonicity with $\sigma_i=\partial_{\eps_i}\fpot(\eps_i)\,,\ \sigma(\eps)=\D_\eps W(\eps)$ the result is equivalent to:
	\begin{framed}
	The potential case $\sigma_i=\partial_{\eps_i}\fpot(\eps_i)\,,\qquad\sigma(\eps)=\D_\eps W(\eps)$ for differentiable functions $W\col\Sym(n)\to\R$
		\begin{align*}
			&\text{vector-monotonicity:}&\sum_{i=1}^n(\sigma_i-\overline\sigma_i)(\eps_i-\overline\eps_i)=\sum_{i=1}^n(\partial_{\eps_i}\fpot(\eps_i)-\partial_{\eps_i}\fpot(\overline\eps_i))(\eps_i-\overline\eps_i)&\geq 0\\
			\iff\qquad&\text{matrix-monotonicity:}&\iprod{\sigma-\overline\sigma,\eps-\overline\eps}=\iprod{\D_\eps W(\eps)-\D_\eps W(\overline\eps),\,\eps-\overline\eps}&\geq 0\,.
		\end{align*}
	\end{framed}
	The alternative result for a nondifferentiable scalar-valued function $W$ can be found in the earlier paper by \citet[and \ref{theorem:DavisLewis} in this paper]{davis1957all} and later in \citet[p.365, Theorem 5.1]{ball1976convexity}.
\end{remark}
\begin{remark}\label{remark:OriginalProofHill}
	Hill \cite{hill1970constitutive} extends his considerations to so-called conjugate stress-strain pairs which, however, exclude the Cauchy stress.\footnote{\citet[p.464]{hill1970constitutive}: \enquote{Of course, stress tensors can be defined which are not conjugate to any strain measure in the present sense. One such is Cauchy stress for a compressible solid,...}} For the conjugate stress-strain case [Ogden \cite{ogden1997non}, p.158] Hill proposes two independent arguments.
	
	His first argument is based on his method of rotated Lagrangian axes. This method critically uses the concept of conjugate stress-strain pairs. Therein, Hill claims that it is sufficient to consider only coaxial stretches $\eps$ and $\overline\eps$, whose eigenvalues are given in the same algebraic order \cite[Hill, p.466, second sentence]{hill1970constitutive}. Since this argument is both restricted to the special case of conjugate stress-strain pairs, ordered eigenvalues and differentiability of the stresses and not truly transparent, we focus on his second, purely algebraic argument, which deserves detailed attention and which is treated in the Appendix \ref{appendix:OriginalProofHill}
\end{remark}

%
%
%
%
\subsection{Permutation matrices}
We want to give our initial proof for the equivalence of vector-monotonicity with matrix-monotonicity in the general non-potential case, independently of our improved interpretation of Hill's proof, cf.\ Appendix \ref{appendix:ImprovedInterpretationHill}.

First, we need to introduce some properties of permutation matrices.
\begin{lemma}
\label{lemma:vectorMonotonicityImpliesDiagonalMonotonicity}
	Let $f\col M\subset\R^n\to\R^n$ be a vector-monotone symmetric function. Then
	\begin{equation}
		\iprod{\Sigma_f(D_1)-\Sigma_f(D_2),\,D_1-D_2} \geq 0
	\end{equation}
	for all \emph{diagonal} matrices $D_1,D_2\in\Sym(M)\subset\Symn$. If $f$ is strictly vector-monotone, then the inequality is strict.
\end{lemma}
\begin{proof}
	Choose $x,y\in M\subset\R^n$ such that $D_1=\diag(x)$ and $D_2=\diag(y)$. Then
	\begin{equation}
		\iprod{\Sigma_f(D_1)-\Sigma_f(D_2),\,D_1-D_2} = \iprod{f(x)-f(y),\,x-y} \geq 0\,.
	\end{equation}
	If $f$ is strictly vector-monotone, then the inequality is strict for $x\neq y$, equivalently, for $D_1\neq D_2$.
\end{proof}
\begin{definition}[Doubly stochastic matrix]
	A matrix $P\in\Rnn$ is called \emph{doubly stochastic} if the following conditions are satisfied:
	\begin{itemize}
		\item[i)] $P_{ij}\geq0$ for all $i,j\in\{1,\dotsc,n\}$\,,
		\item[ii)] $\sum_{k=1}^n P_{kj}=1$ for all $j\in\{1,\dotsc,n\}$\,,
		\item[iii)] $\sum_{k=1}^n P_{ik}=1$ for all $i\in\{1,\dotsc,n\}$\,.
	\end{itemize}
\end{definition}
\begin{lemma}
\label{lemma:componentwiseSquaredOrthogonalMatrixIsDoublyStochastic}
	Let $Q\in\On$. Then the matrix $P\in\Rnn$ with
	\begin{equation}
		P_{ij} \coloneqq (Q_{ij})^2
	\end{equation}
	for all $i,j\in\{1,\dotsc,n\}$ is doubly stochastic.\footnote{Note carefully that $P_{ij}\neq (Q^2)_{ij}$ in general.}
\end{lemma}
\begin{example}
	For the rotation matrices for $n=2$, the resulting doubly stochastic matrices have the form
	\begin{equation}
		Q=\matr{\cos\alpha&-\sin\alpha\\\sin\alpha&\cos\alpha},\qquad P=\matr{(\cos\alpha)^2&(\sin\alpha)^2\\(\sin\alpha)^2&(\cos\alpha)^2}
	\end{equation}
	with an arbitrary rotation angle $\alpha\in\R$ and $(\cos\alpha)^2+(\sin\alpha)^2=1$.
\end{example}
\begin{definition}[Permutation matrix]
	A matrix $P\in\Rnn$ is called a \emph{permutation matrix} if $P_{ij}\in\{0,1\}$ for all $i,j\in\{1,\dotsc,n\}$ and $P$ is doubly stochastic (and thus $P\in\On$). Equivalently, $P$ has exactly one entry equal to 1 in each row and column.
	
	A matrix $P\in\Rnn$ is called a \emph{signed permutation matrix} if $P_{ij}\in\{-1,0,1\}$, for all $i,j\in\{1,\dotsc,n\}$ and it has exactly one nonzero entry $\pm1$ in each row and column. Hence, the corresponding \emph{permutation matrix} $\Phat\in\Rnn$ with $\Phat_{ij}\coloneq |P_{ij}|$ for all $i,j\in\{1,\dotsc,n\}$ is doubly stochastic.
\end{definition}
\begin{lemma}
\label{lemma:permutationMatrixAsSignedSquare}
	Each permutation matrix $P\in\Rnn$ can be written as $P_{ij}=(Q_{ij})^2$ for all $i,j\in\{1,\dotsc,n\}$ with a signed permutation matrix $Q\in\SOn$.
\end{lemma}
\begin{proof}
	Let $Q\in\Rnn$ with $Q_{ij}=0$ if $P_{ij}=0$ and $Q_{ij}\in\{-1,1\}$ if $P_{ij}=1\,.$ Then
	\begin{equation}
			\abs{Q_{ij}}=(Q_{ij})^2=P_{ij}\,.\label{eq:permutationMatrixAsSignedSquare}
	\end{equation}
	Thus $Q$ is a signed permutation matrix by construction and therefore belongs in $\On$. If $\det Q=-1$, changing the sign of one row (or column) yields another signed permutation matrix satisfying the same identity \eqref{eq:permutationMatrixAsSignedSquare} and having determinant $1$. Hence, we can always choose $Q\in\SOn$.
\end{proof}
\begin{lemma}
\label{lemma:permutationMatrixPreservesDiagonalForm}
	Let $Q\in\On$ be a signed permutation matrix. Then $Q^TDQ$ is diagonal for any diagonal matrix $D\in\Symn$.
\end{lemma}
\begin{proof}
	For a diagonal matrix $D=\diag(d_1,\dotsc,d_n)$, we compute
	\begin{align}
		(Q^TD Q)_{ij} = \sum_{k,l=1}^n (Q^T)_{ik}\. D_{kl}\. Q_{lj} = \sum_{k=1}^n (Q^T)_{ik}\. D_{kk}\. Q_{kj} = \sum_{k=1}^n d_k \. Q_{ki}\. Q_{kj}\,.
	\end{align}
	$Q$ is a signed permutation matrix and therefore $\abs Q$ is doubly stochastic:
	\begin{align}
		&\sum_{k=1}^n |Q_{kj}|=1\quad\text{and}\quad Q_{kj}\in\{-1,0,1\}\\
		\implies\qquad&\exists\,\jhat\in\{1,\dotsc,n\}:\qquad|Q_{k\jhat}|=1\qquad\text{and}\qquad Q_{kj}=0\quad\forall\; j\neq\jhat\,.\notag
	\end{align}
	In particular, for $i \neq j$, no row $k$ can contain non-zero entries simultaneously in columns $i$ and $j$; hence
	\begin{align}
		Q_{ki}\. Q_{kj}&=0\qquad\forall\; i\neq j\\
		\implies\qquad(Q^TD\. Q)_{ij}=\sum_{k=1}^n d_k \. Q_{ki}\. Q_{kj}&=0\qquad\forall\; i\neq j\,.\notag\qedhere
	\end{align}
\end{proof}
\begin{theorem}[Birkhoff–von-Neumann theorem]
\label{theorem:birkhoff-von-neumann}
	The set $B_n\subset\Rnn$ of all doubly stochastic $n\times n$-matrices is the convex hull of the set of all permutation matrices, and the permutation matrices are exactly the vertices of $B_n$.
\end{theorem}
\begin{proof}
	See \citet{birkhoff1946three} or von \citet{von1953certain}.
\end{proof}

In particular, the Birkhoff–von-Neumann Theorem implies that every doubly stochastic matrix $P\in B_n$ can be written by a convex combination 
\begin{equation}
	P=\sum_{\pi=1}^m c_\pi\.P_\pi\qquad\text{with}\quad\sum_{\pi=1}^m c_\pi=1\,,\label{eq:BirkhoffNeumannConvexCombination}
\end{equation}
where $P_\pi\in\Rnn$ are permutation matrices and $c_\pi>0$.

%
%
%
\subsection{Equivalence of vector-monotonicity and matrix-monotonicity}

We can now state our main result, which is independent of the framework of nonlinear elasticity:
\begin{theorem}\label{theorem:mainResult}
	A symmetric function $f\col M\subset\R^n\to\R^n$ is (strictly) vector-monotone if and only if it is (strictly) matrix-monotone.
\end{theorem}
The main part of the proof consists in showing that the following lemma holds.
\begin{lemma}
	\label{lemma:mainLemma}
	For each $k\in\{1,\dotsc,m\}$ let $A\kpow,B\kpow\in\Rnn$ be diagonal and let
	\begin{equation}
		\Psi\col\On\to\R\,,\qquad \Psi(Q) = \sum_{k=1}^m \iprod{A\kpow, Q^T B\kpow Q}\,.
	\end{equation}
	Then
	\begin{itemize}
		\item[i)] $\Psi$ attains its maximum at $\Qhat\in\SOn$ such that $\Qhat^T B\kpow \Qhat$ is diagonal for each $k\in\{1,\dotsc,m\}$.
		\item[ii)] If the diagonal entries at all $A\kpow$ and $B\kpow$ are ordered with identical equality blocks, then for all maximizer $\Qtilde\in\On$ of $\Psi$ it holds
		$\Qtilde^T B\kpow \Qtilde=B\kpow$ for each $k\in\{1,\dotsc,m\}$.
	\end{itemize}
\end{lemma}
A simple direct proof of this lemma for the 2D case and $k=2$ can be found in Appendix \ref{appendix:mainLemma2D}.
\begin{proof}[Proof of Theorem \ref{theorem:mainResult}]
	Due to Lemma \ref{lemma:matrixMonotonicityImpliesVectorMonotonicity}, we only need to show that (strict) vector-monotonicity implies (strict) matrix-monotonicity.
	
	Let $f$ be a vector-monotone function. For $S,T\in\Sym(M)$, we choose $Q_1,Q_2\in\On$ and diagonal matrices $D_1,D_2$ with ordered diagonal entries such that
	\begin{equation}
		S = Q_1^T D_1 \.Q_1 \qquad\text{and}\qquad T = Q_2^T D_2\. Q_2\,.
	\end{equation}
	With $\Qtilde\colonequals Q_2\.Q_1^T\in\On$, we find
	\begin{align}
		&\hspace*{-2.1em}\iprod{S-T,\, \Sigma_f(S)-\Sigma_f(T)}\notag\\
		&= \iprod{Q_1^T D_1\. Q_1 - Q_2^T D_2\. Q_2,\,  \Sigma_f(Q_1^TD_1\.Q_1)  -  \Sigma_f(Q_2^TD_2\.Q_2) }\notag\\
		&= \iprod{Q_1^T D_1\. Q_1 - Q_2^T D_2\. Q_2,\, Q_1^T \Sigma_f(D_1) Q_1 - Q_2^T \Sigma_f(D_2)\. Q_2}\notag\\
		&= \iprod{D_1 - Q_1\.Q_2^T D_2\. Q_2\.Q_1^T,\, \Sigma_f(D_1) - Q_1\.Q_2^T \Sigma_f(D_2) Q_2\.Q_1^T}\notag\\
		&= \iprod{D_1 - \Qtilde^T D_2 \.\Qtilde,\, \Sigma_f(D_1) - \Qtilde^T \Sigma_f(D_2)\. \Qtilde}\notag\\
		&= \iprod{D_1, \Sigma_f(D_1)} + \iprod{\Qtilde^T D_2\. \Qtilde,\, \Qtilde^T \Sigma_f(D_2)\. \Qtilde} - \iprod{D_1, \Qtilde^T \Sigma_f(D_2)\. \Qtilde} - \iprod{\Qtilde^T D_2\. \Qtilde,\, \Sigma_f(D_1)}\notag\\
		&= \iprod{D_1, \Sigma_f(D_1)} + \iprod{D_2, \Sigma_f(D_2)} - \iprod{D_1, \Qtilde^T \Sigma_f(D_2)\. \Qtilde} - \iprod{\Sigma_f(D_1), \Qtilde^T D_2\. \Qtilde}\\
		&\geq \iprod{D_1, \Sigma_f(D_1)} + \iprod{D_2, \Sigma_f(D_2)} - \max_{Q\in\On} \big[\iprod{D_1, Q^T \Sigma_f(D_2)\.Q} + \iprod{\Sigma_f(D_1), Q^T D_2\. Q}\big]\,.\notag
	\end{align}
	We can now apply Lemma \ref{lemma:mainLemma} i) to $A=D_1$, $B=\Sigma_f(D_2)$, $C=\Sigma_f(D_1)$ and $D=D_2$ to obtain a maximizer $\Qhat\in\SOn\subset\On$. Then
	\begin{align}
		&\hspace*{-2em}\iprod{S-T,\, \Sigma_f(S)-\Sigma_f(T)}\notag\\
		&\geq \iprod{D_1, \Sigma_f(D_1)} + \iprod{D_2, \Sigma_f(D_2)} - \iprod{D_1, \Qhat^T \Sigma_f(D_2)\. \Qhat} - \iprod{\Sigma_f(D_1), \Qhat^T D_2\. \Qhat}\label{eq:mainTheoremStrictness}\\
		&= \iprod{D_1, \Sigma_f(D_1)} + \iprod{\Qhat^T D_2\. \Qhat,\, \Qhat^T \Sigma_f(D_2)\. \Qhat} - \iprod{D_1, \Qhat^T \Sigma_f(D_2)\. \Qhat} - \iprod{\Qhat^T D_2 \.\Qhat, \Sigma_f(D_1)}\notag\\
		&= \iprod{D_1 - \Qhat^T D_2 \.\Qhat,\, \Sigma_f(D_1) - \Qhat^T \Sigma_f(D_2)\. \Qhat}\notag\\
		&= \iprod{D_1 - \Qhat^T D_2 \.\Qhat,\, \Sigma_f(D_1) - \Sigma_f(\Qhat^T D_2 \.\Qhat)}\notag\,.
	\end{align}
	Due to Lemma \ref{lemma:mainLemma} i), $D_3\colonequals\Qhat^T D_2 \.\Qhat\in\Sym(M)$ is diagonal, thus Lemma \ref{lemma:vectorMonotonicityImpliesDiagonalMonotonicity} yields
	\begin{equation}
		\iprod{S-T,\, \Sigma_f(S)-\Sigma_f(T)}
		\geq\iprod{D_1 - D_3,\, \Sigma_f(D_1) - \Sigma_f(D_3)} \geq 0\,,\label{eq:mainTheoremEnd}
	\end{equation}
	for all $S,T\in\Sym(M)$ and thus $f$ is matrix-monotone.
	\medskip
	
	Now, if $f$ is strictly vector-monotone, the last estimate \eqref{eq:mainTheoremEnd} is also strict for $D_1\neq D_3$. We consider the remaining case where $S\neq T$, but
	\begin{equation}
		D_1=D_3=\Qhat^T D_2\.\Qhat\,.
	\end{equation}
	Then, the diagonal entries of $D_1$ and $D_2$ are identical up to permutation and thus $D_1=D_2$, since both $d_1\colonequals\diag\inv(D_1)$ and $d_2\colonequals\diag\inv(D_2)$ are ordered.
	
	If $\Qtilde=Q_2\.Q_1^T$ is not a maximizer, the estimate \eqref{eq:mainTheoremStrictness} is strict and thus, we directly obtain strict matrix-monotonicity.
	Assume therefore that $\Qtilde\in\On$ is also a maximizer but does not necessarily equal $\Qhat\in\SOn$. 
	Because $f$ is strictly vector-monotone, Lemma \ref{lemma:vectorMonotonicityImpliesSameAlgebraicOrder} implies that $d_1$ and $f(d_1)$ (and thus also $d_2$ and $f(d_2)$) are ordered with identical equality blocks.
	Then according to Lemma \ref{lemma:mainLemma} ii), $\Qtilde^T D_2\.\Qtilde=D_2$ and thus
	\begin{align}
		S&=Q_1^T D_1 \.Q_1
		=Q_1^T D_2 \.Q_1
		=Q_1^T (\Qtilde^T D_2\.\Qtilde) \.Q_1
		=Q_1^T (Q_2\.Q_1^T)^T D_2\.(Q_2\.Q_1^T) \.Q_1\\
		&=Q_1^T Q_1\.Q_2^T D_2\.Q_2\.Q_1^T \.Q_1
		=Q_2^T D_2\. Q_2=T\,,\notag
	\end{align}
	contrary to the assumption $S\neq T$. Thus $\Qtilde$ cannot be a maximizer and inequality \eqref{eq:mainTheoremStrictness} is strict. Therefore, strict vector-monotonicity implies strict matrix-monotonicity.
\end{proof}
\begin{proof}[Proof of Lemma \ref{lemma:mainLemma}]
	i) The set $\On\subset\Rnn$ is compact and the mapping $\Psi$ is continuous, so there exists a (not necessarily unique) maximizer $\Qhat$. According to Lemma \ref{lemma:permutationMatrixPreservesDiagonalForm}, it suffices to show that at least one maximizer can be chosen as a signed permutation matrix.
	Let $a\kpow\colonequals\diag\inv(A\kpow)$ and $b\kpow\colonequals\diag\inv(B\kpow)$. We first compute
	\begin{equation}
		(Q^T B\kpow Q)_{ij} = \sum_{r,s=1}^n (Q^T)_{ir}\. B\kpow_{rs}\. Q_{sj}
		= \sum_{r=1}^n (Q^T)_{ir}\. B\kpow_{rr}\. Q_{rj}
		= \sum_{r=1}^n b\kpow_r \. Q_{ri}\. Q_{rj}\,.
	\end{equation}
	In particular,
	\begin{equation}
		(Q^TB\kpow Q)_{ii} = \sum_{j=1}^n b\kpow_j\. (Q_{ji})^2
	\end{equation}
	for all $i\in\{1,\dotsc,n\}$ and thus
	\begin{align}
		\iprod{A\kpow, Q^T B\kpow Q}=\sum_{i=1}^n A_{ii}\kpow\.(Q^T B\kpow Q)_{ii}
		=\sum_{i=1}^n a_i\kpow\.\sum_{j=1}^n b\kpow_j\. (Q_{ji})^2.
	\end{align}
	Therefore,
	\begin{align}
		\Psi(Q) &= \sum_{k=1}^m \iprod{A\kpow, Q^T B\kpow Q}
		=\sum_{k=1}^m \sum_{i,j=1}^n a_i\kpow\.b\kpow_j\. (Q_{ji})^2
		=\sum_{k=1}^m \sum_{i,j=1}^n a_i\kpow\.b\kpow_j\. P^Q_{ji}\\
		&=\sum_{k=1}^m \sum_{j=1}^n (P^Q\.a\kpow)_j\.b\kpow_j
		=\sum_{k=1}^m \iprod{P^Q\.a\kpow,b\kpow}\,,\notag
	\end{align}
	where $P^Q\in\Rnn$ is the doubly stochastic matrix with $P^Q_{ji}=(Q_{ji})^2$, cf.\ Lemma \ref{lemma:componentwiseSquaredOrthogonalMatrixIsDoublyStochastic}. Thus
	\begin{align}
		\max_{Q\in\On}\Psi(Q)&= \sum_{k=1}^m \iprod{P^Q\.a\kpow,b\kpow}
		\leq \max_{P\in B_n} \sum_{k=1}^m \iprod{P\.a\kpow,b\kpow}\,,\label{eq:BirkhoffEstimate}
	\end{align}
	where $B_n\subset\Rnn$ is the set of all doubly-stochastic matrices.
	The mapping
	\begin{equation}
		L\col\Rnn\to\R\,,\qquad L(P)=\sum_{k=1}^m \iprod{P\.a\kpow,b\kpow}
	\end{equation}
	is linear and thus convex. Therefore, it attains its maximum on the convex set $B_n$ at (at least) one of the vertices of $B_n$, which are exactly the permutation matrices, cf.\ Theorem \ref{theorem:birkhoff-von-neumann}. According to Lemma \ref{lemma:permutationMatrixAsSignedSquare} each permutation matrix $\Phat$ can be written as $\Phat_{ji}=(\Qhat_{ji})^2$ with a signed
	permutation matrix $\Qhat\in\SOn$, so the maximum $\displaystyle\max_{Q\in\On}\Psi(Q)$ is attained at a signed permutation matrix $\Qhat\in\SOn$ and due to Lemma \ref{lemma:permutationMatrixPreservesDiagonalForm}, $\Qhat^TD\.\Qhat$ is diagonal.
	\medskip
	
	ii) Under the additional assumption that $a\kpow$ and $b\kpow$ are ordered with identical equality blocks, it follows from Lemma \ref{lemma:hardy} that
	\begin{equation}
		\iprod{P\.a\kpow,b\kpow}\leq \iprod{a\kpow,b\kpow}
	\end{equation}
	for all $k\in\{1,\dotsc,m\}$, and that equality holds if and only if $P\.a\kpow=a\kpow$. In particular, $P=\id$ maximizes $L$, and for any $P\in\Rnn$ with $P\.a\kpow\neq a\kpow$ for some $k\in\{1,\dotsc,m\}$ (equivalently, $P\.b\kpow\neq b\kpow$ for some $k\in\{1,\dotsc,m\}$, since the equality blocks are identical), it must hold
	\begin{equation}
		L(P)<L(\id)\,.
	\end{equation}
	Now, for each maximizer $\Qtilde\in\On$ of $\Psi$ we consider $P^\Qtilde\in\Rnn$ with $P^\Qtilde_{ji}=(\Qtilde_{ji})^2$.
	The objective function $L(P)$ is linear in $P$ and with the Birkhoff decomposition \eqref{eq:BirkhoffNeumannConvexCombination}, we have
	\begin{equation}
		L(P^\Qtilde)=\sum_{\pi}^m c_\pi\.L(P_\pi)\,,\qquad L(P_\pi)\leq\Psi(\Qtilde)=L(P^\Qtilde)\,.\label{eq:convexCombination}
	\end{equation}
	Because $\sum_{\pi=1}^m c_\pi=1$ and $c_\pi>0$, each permutation matrix $P_\pi$ must already attain the upper bound individually. Then
	\begin{equation}
		P^\Qtilde=\sum_{\pi}c_\pi\.P_\pi\,,\qquad\text{with}\qquad P_\pi\.b\kpow=b\kpow
	\end{equation}
	for all $k\in\{1,\dotsc,m\}$ and thus
	\begin{equation}
		(\Qtilde\. B\kpow \Qtilde^T)_{ii} 
		=\sum_{j=1}^n (\Qtilde_{ij})^2\. b\kpow_j
		= (P^{\Qtilde}b\kpow)_i = \sum_\pi c_\pi\.(P_\pi\. b\kpow)_i = \sum_\pi c_\pi\.b\kpow_i = b\kpow_i = B\kpow_{ii}
	\end{equation}
	for all $i\in\{1,\dotsc,n\}$.
	Hence, $\Qtilde\. B\kpow \Qtilde^T$ and $B\kpow$ have identical diagonal entries.
	Due to the $\SO(n)$-invariance of the Frobenius norm, and since $B\kpow$ is diagonal,
	\begin{equation}
		\sum_{i,j=1}^n (\Qtilde\. B\kpow \Qtilde^T)_{ij}^2 = \norm{\Qtilde\. B\kpow \Qtilde^T}^2 = \norm{B\kpow}^2 = \sum_{i=1}^n (B\kpow_{ii})^2 = \sum_{i=1}^n (\Qtilde\. B\kpow \Qtilde^T)_{ii}^2\,,
	\end{equation}
	which immediately implies $\Qtilde\. B\kpow \Qtilde^T$ is also diagonal and thus
	\begin{equation}
		\Qtilde\. B\kpow \Qtilde^T = B\kpow
		\qquad\iff\qquad
		B\kpow=\Qtilde^T B\kpow \Qtilde 
	\end{equation}
	for all $k\in\{1,\dotsc,m\}$.
\end{proof}

%
%
\subsection{Generalization of Friedland's theorem}
Similar to the generalization of Chandler Davis Theorem \ref{theorem:DavisLewis}, Friedland's theorem \ref{theorem:friedlandOrderedConvexity} for ordered eigenvalues can be generalized to tensor-valued functions without a potential as well. 
\begin{theorem}
	Let $\phi\col\Rsort\to\R^n$ such that $\phi_i(x)=\phi_j(x)$ for all $x\in\Rsort$ with $x_i=x_j$. Then the function $\Sigma^\phi \col\Symn\to\Symn$ as in Lemma \ref{lemma:orderedVectorFunctionInducesTensorFunction} is monotone if and only if
	\begin{itemize}
		\item[i)]
			$\phi$ is monotone and
		\item[ii)]
			$\phi(x)\in\Rsort$ for all $x\in\Rsort$.
	\end{itemize}
	Furthermore, $\Sigma^\phi$ is strictly monotone if and only if
	\begin{itemize}
		\item[i*)]
			$\phi$ is strictly monotone and
		\item[ii*)]
			$\phi_i(x)>\phi_j(x)$ for all $x\in\Rsort$ with $x_i>x_j$.
	\end{itemize}
\end{theorem}
\begin{proof}
	If $\Sigma^\phi $ is monotone, then for all $x,y\in\Rsort$
	\begin{align}
		\iprod{\phi(x)-\phi(y),x-y} &= \iprod{\diag(\phi(x))- \diag(\phi(y)),\diag(x)-\diag(y)}\\
		&= \iprod{\Sigma^\phi (\diag(x))-\Sigma^\phi (\diag(y)),\diag(x)-\diag(y)} \geq 0\notag
	\end{align}
	and strict inequality holds if $\Sigma^\phi $ is strictly monotone. 
	For the proof of ii), we assume there exists a $x\in\Rsort$ with $\phi(x)\not\in\Rsort$. Thus there exist $i\neq j$ such that $x_i\geq x_j$ but $\phi_i(x)<\phi_j(x)$. We choose $y\in\R^n$ as
		\begin{align}
			y_k=x_k\qquad\forall\;k\neq i,j\,,\qquad y_i=x_j\quad\text{and}\quad y_j=x_i\,.
		\end{align}
	Lemma \ref{lemma:existenceOfsymmetricExpansion} states the existence of a symmetric extension $f_\phi\col\R^n\to\R^n$ of $\phi$ and we compute
	\begin{align}
		&\phantom{.}\hspace{-1cm}\iprod{\Sigma^\phi (\diag(x))-\Sigma^\phi (\diag(y)),\diag(x)-\diag(y)} \notag\\
		&=
		\left(\Sigma^\phi (\diag(x))_{ii}-\Sigma^\phi (\diag(y))_{ii}\right)(x_i-y_i)+\left(\Sigma^\phi (\diag(x))_{jj}-\Sigma^\phi (\diag(y))_{jj}\right)(x_j-y_j)\notag\\
		&=(x_i-x_j)\left(\Sigma^\phi (\diag(x))_{ii}+\Sigma_{f_\phi}(\diag(y))_{jj}-\Sigma^\phi (\diag(x))_{jj}-\Sigma_{f_\phi}(\diag(y))_{ii}\right)\notag\\
		&=(x_i-x_j)\left(\phi_i(x)+{f_\phi}_j(y)-\phi_j(x)-{f_\phi}_i(y)\right)\notag\\
		&=2\.(x_i-x_j)\left(\phi_i(x)-\phi_j(x)\right)<0\,.\label{eq:strictMonotonicityPairing}
	\end{align}
	This contradicts the monotonicity of $\Sigma^\phi $ and thus ii) must hold. In the case of strict monotonicity to prove ii*), suppose that $x_i>x_j$ but $\phi_i(x)=\phi_j(x)$. Swapping the two coordinates gives $y\neq x$ but zero in the pairing \eqref{eq:strictMonotonicityPairing} contradicting strict monotonicity of $\Sigma^\phi$. Thus ii*) must hold.

	\medskip
	Now let i), ii) hold. Then there exists a unique function $f\col\R^n\to\R^n$ with $f(\sort{x})=\phi(\sort{x})$ for all $\sort{x}\in\Rsort$, and $\Sigma^\phi =\Sigma_f$, cf.\ Lemma \ref{lemma:orderedVectorFunctionInducesTensorFunction}. For $x\in\R^n$, let $\sort{x}\in\Rsort$ denote the vector with the components of $x$ sorted in descending order. Then ii) implies $f(\sort{x}) = \sort{f(x)}$ for all $x\in\R^n$ and thus
	\begin{align}
		\iprod{f(x)-f(y), x-y} &= \iprod{f(x),x} + \iprod{f(y),y} - \iprod{f(x),y} - \iprod{f(y),x}\nnl
		&= \iprod{\sort{f(x)}, \sort{x}} + \iprod{\sort{f(y)}, \sort{y}} - \iprod{f(x),y} - \iprod{f(y),x}\nnl
		&\geq \iprod{\sort{f(x)}, \sort{x}} + \iprod{\sort{f(y)}, \sort{y}} - \iprod{\sort{f(x)}, \sort{y}} - \iprod{\sort{f(y)}, \sort{x}} \label{eq:orderedVectorsHillInequality}\\
		&= \iprod{f(\sort{x}), \sort{x}} + \iprod{f(\sort{y}), \sort{y}} - \iprod{f(\sort{x}), \sort{y}} - \iprod{f(\sort{y}), \sort{x}}\nnl
		&= \iprod{\phi(\sort{x}), \sort{x}} + \iprod{\phi(\sort{y}), \sort{y}} - \iprod{\phi(\sort{x}), \sort{y}} - \iprod{\phi(\sort{y}), \sort{x}}\nnl
		&= \iprod{\phi(\sort x)-\phi(\sort y), \sort x-\sort y}
		\;\geq\; 0 \label{eq:orderedVectorsOrderedMonotonicity}
	\end{align}
	for all $x,y\in\R^n$, where \eqref{eq:orderedVectorsHillInequality} is due to Lemma \ref{lemma:hardy}.
	Furthermore, if $x\neq y$ and $\phi$ is strictly monotone, then either $\sort{x}\neq\sort{y}$ and hence strict inequality holds in \eqref{eq:orderedVectorsOrderedMonotonicity}; or $\sort{x}=\sort{y}$ which means, that $x$ and $y$ are not sorted in the same algebraic order. 
	We call two vectors $a,b\in\R^n$ \emph{ordered compatibly} if there exists a permutation  $\pi\col\{1,\dotsc,n\}\to\{1,\dotsc,n\}$  such that simultaneously $\sort{a}_i=a_{\pi(i)}$ and $\sort{b}_i=b_{\pi(i)}$ for all $i=1,\dotsc,n$.
	Condition ii*) ensures that $x$ and $f(x)$ as well as $y$ and $f(y)$ are ordered compatibly and with $x\neq y$ at least one pair $\iprod{f(x),y}$ or $\iprod{f(y),x}$ is not ordered compatibly.
	Thus when ordering, step \eqref{eq:orderedVectorsHillInequality} must involve an inversion of two indices with distinct values in the corresponding vectors and it is therefore strict \cite{hardy1952inequalities}.
	
	Thus the function $f$ is (strictly) monotone on $\R^n$ and we can apply Theorem \ref{theorem:mainResult} to find that $\Sigma^\phi =\Sigma_f$ is (strictly) monotone.
\end{proof}
\begin{remark}
	Note that if $\Sigma^\phi $ represents the Cauchy stress tensor $\sigma\in\Sym(3)$, then condition ii*) corresponds to the strong Baker-Ericksen inequalities
	\begin{equation}
		\text{BE$^+$:}\qquad\lambda_i>\lambda_j\qquad\implies\qquad\sigma_i>\sigma_j\qquad\forall\;i,j\in\{1,\dotsc,n\}\label{eq:BE}
	\end{equation}
	with the principal stresses $\sigma_i$ corresponding to the principal stretches $\lambda_i$.
\end{remark}

%
%
%
%
%
\section{Invertibility of isotropic tensor functions}
A similar equivalence criterion holds for the \emph{invertibility} of isotropic tensor functions.
\begin{theorem}\label{theorem:invertibility}
	Let $f\col\R^n\to\R^n$ be symmetric.
	\begin{itemize}
		\item[i)] The function $\Sigma_f\col\Symn\to\Symn$ is injective if and only if $f$ is injective.
		\item[ii)] The function $\Sigma_f\col\Symn\to\Symn$ is surjective if and only if $f$ is surjective.
	\end{itemize}
	In particular, $\Sigma_f$ is invertible if and only if $f$ is invertible.
\end{theorem}
\begin{proof}
	i)\quad
	If $\Sigma_f$ is injective, then we observe that for all $x,y\in\R^n$ with $f(x)=f(y)$,
	\begin{equation}
		\Sigma_f(\diag(x)) = \diag(f(x)) = \diag(f(y)) = \Sigma_f(\diag(y))\,,
	\end{equation}
	thus the injectivity of $\Sigma_f$ implies $\diag(x)=\diag(y)$ or, equivalently, $x=y$.
	
	Now let $f$ be injective. Then there exists a \emph{left inverse} of $f$, i.e.\ a function $h\col\range(f)\to\R^n$ such that $h(f(x))=x$ for all $x\in\R^n$, where $\range(f) = \{f(x) \setvert x\in\R^n\}$ denotes the range of $f$. We observe that $\range(f)$ is a symmetric set, and since, for all $y=f(x)\in\range(f)$ and all permutations $\pi\col\{1,\dotsc,n\}\to\{1,\dotsc,n\}$,
	\begin{align}
		h(y_{\pi(1)},\dotsc, y_{\pi(n)}) &= h(f_{\pi(1)}(x),\dotsc f_{\pi(n)}(x))\\
		&= h(f(x_{\pi(1)},\dotsc, x_{\pi(n)}))
		= (x_{\pi(1)},\dotsc, x_{\pi(n)})
		= (h_{\pi(1)}(y),\dotsc, h_{\pi(n)}(y))\,,\notag
	\end{align}
	the function $h$ is a symmetric vector function.
		
	Let $\Sigma_h\col \Sym(\range(f))\to\Symn$ denote the isotropic tensor function corresponding to the symmetric vector function $h$. Then for all $X = Q^T\diag(x)\.Q\in\Symn$,
	\begin{equation}
		\Sigma_h(\Sigma_f(X)) = \Sigma_h(Q^T\diag(f(x))\.Q) = Q^T\diag(h(f(x)))\.Q = Q^T\.\diag(x)\.Q = X\,.
	\end{equation}
	Therefore, $\Sigma_h$ is a left inverse of $\Sigma_f$, which implies that $\Sigma_f$ is injective.
	
	\medskip
	 ii)\quad
	If $\Sigma_f$ is surjective, then for any $y\in\R^n$, there exists $S\in\Symn$ such that $\Sigma_f(S) = \diag(y)$. Let $x_1,\dotsc,x_n$ denote the eigenvalues of $S$.
	Then there exists a rearrangement $x_\pi = x_{\pi(1)},\dotsc,x_{\pi(n)}$ of $x$ such that $y=f(x_\pi)$, since $f$ maps the eigenvalues of $S$ to those of $\Sigma_f(S)=\diag(y)$. Therefore, $f$ is surjective as well.
	
	Now let $f$ be surjective. Then for any $S = Q^T\diag(y)\.Q\in\Symn$ with $y\in\R^n$ and $Q\in\On$, there exists $x\in\R^n$ such that $f(x)=y$ and thus
	\begin{equation}
		\Sigma_f(Q^T\diag(x)\.Q) = Q^T\diag(f(x))\.Q = Q^T\diag(y)\.Q = S\,,
	\end{equation}
	hence $\Sigma_f$ is surjective.
\end{proof}
This result fundamentally improves a previous statement in \citet{truesdell1975inequalities} in which the Baker-Ericksen inequalities \eqref{eq:BE} were invoked to first assert the so-called \textbf{semi-invertibility} of $B\to\sigma(B)$, i.e.\ the possibility to write
\begin{align}
	B=\psi_0(I_1(B),I_2(B),I_3(B))\.\id+\psi_1(I_1(B),I_2(B),I_3(B))\.\sigma(B)+\psi_2(I_1(B),I_2(B),I_3(B))\.(\sigma(B))^2\,.
\end{align}
Observe that
\begin{equation}
I_1(B)=\lambda_1^2+\lambda_2^2+\lambda_3^2=\norm{F}^2\,,\quad I_2(B)=\lambda_1^2\.\lambda_2^2+\lambda_1^2\.\lambda_3^2+\lambda_2^2\.\lambda_3^2=\norm{\Cof F}^2\,,\quad I_3(B)=\lambda_1^2\.\lambda_2^2\.\lambda_3^2=(\det F)^2
\end{equation}
and thus that the semi-invertibility can also be written as
\begin{align}
	B=\alpha_0(\lambda_1,\lambda_2,\lambda_3)\.\id+\alpha_1(\lambda_1,\lambda_2,\lambda_3)\.\sigma(B)+\alpha_2(\lambda_1,\lambda_2,\lambda_3)\.(\sigma(B))^2
\end{align}
with symmetric functions $\alpha_0,\alpha_1,\alpha_2\col\R_+^3\to\R\,.$ \citet{truesdell1975inequalities} add \enquote{... proper numbers [$(\lambda_1,\lambda_2,\lambda_3)$] of $[B]$ be (single valued) functions of the proper numbers of $[\sigma]$.}

Theorem \ref{theorem:invertibility} allows to dramatically reduce the effort to check for invertibility, since we only need to check for invertibility of principal Cauchy stresses $\widehat\sigma\col\R_+^3\to\R^3$ versus principal stretches.
Another, for us surprising, connection is given in our last Theorem.
\begin{theorem}
	Assume that the map $V\to\sigma(V)$ is injective and continuously differentiable. Moreover, assume that the tangent stiffness in the natural state $\id$ is positive definite, i.e.
	\begin{equation}
		\iprod{\sym\D\sigma(\id).H,H}>0\qquad\forall\;H\in\Sym(3)\backslash\{0\}.\label{eq:postiveDefiniteDSigma}
	\end{equation}
	Then the strong Baker-Ericksen inequalities \eqref{eq:BE} are satisfied throughout.
\end{theorem}
\begin{proof}
	By continuity of the tangent operator $V\mapsto \D\sigma(V)$ and the positive definiteness at $V=\id$, there exists a convex neighbourhood of $\id$ where $\sym \D\sigma(V)$ remains positive definite and thus $\sigma$ is strictly monotone in this neighbourhood.
	Now, choose any deformation state $\Vbar\in\Symp(3)$ with three distinct stretches $\lambda_1>\lambda_2>\lambda_3$ therein. 
	Since $\sigma$ is strictly monotone in this neighbourhood, the ordering argument of Lemma \ref{lemma:vectorMonotonicityImpliesSameAlgebraicOrder} implies that the corresponding stresses are also strictly ordered, $\sigma_1>\sigma_2>\sigma_3$, i.e.\ the strong Baker-Ericksen inequalities hold locally.\footnote{Alternatively, one can use the classical spectral derivative formula \cite{chadwick1971theorem} for the Fr\'echet derivative of \eqref{eq:postiveDefiniteDSigma} in the direction $H=e_i\otimes e_j+e_j\otimes e_i$ (only changing the eigenvectors and not the eigenvalues) to arrive at $(\sigma_i-\sigma_j)(\lambda_i-\lambda_j)>0$.}
	Finally, Proposition \ref{prop:injectivityBE} propagates them globally under continuity and injectivity of $V\mapsto\sigma(V)$.
\end{proof}
\begin{proposition}
\label{prop:injectivityBE}
	If the Cauchy stress response function $V\mapsto\sigma(V)$ is continuous and injective and satisfies the strong Baker-Ericksen inequalities at one $V\in\Symp(n)$ with only simple eigenvalues, then the strong BE-inequalities are satisfied everywhere.
\end{proposition}
\begin{proof}
	Let
	\begin{align}
		\widehat\sigma\col\Rp^n\to\R^n\,,\qquad \lambda = (\lambda_1,\dotsc,\lambda_n)\mapsto (\sigma_1(\lambda),\dotsc,\sigma_n(\lambda))
	\end{align}
	denote the representation of $\sigma$ in terms of the principal stretches. Then $\widehat\sigma$ is injective as well and thus
	\begin{equation}\label{eq:semiInvertibility}
		\lambda_i \neq \lambda_j \qquad\implies\qquad \sigma_i(\lambda) \neq \sigma_j(\lambda)\,,
	\end{equation}
	since otherwise
	\begin{align}
		\widehat\sigma(\lambda) = (\sigma_1(\lambda),\dotsc,\sigma_n(\lambda)) = (\sigma_1(\lambda),\dotsc,\sigma_n(\lambda))^{ij} = \widehat\sigma(\lambda^{ij})\,,
	\end{align}
	where, for $x\in\Rp^n$, $x^{ij}$ denotes the vector obtained by interchanging the $i$-th and $j$-th component values of $x$ here and henceforth.
	
	Now let $\lambda\in\Rp^n$ be a vector of principal stretches corresponding to some $V\in\Symp(n)$ with only simple eigenvalues such that the strong BE-inequalities are satisfied at $V$. Then
	\begin{align}
		(\lambda_i-\lambda_j)\. (\sigma_i(\lambda)-\sigma_j(\lambda)) > 0\qquad\forall\; i\neq j\,.
	\end{align}
	After possible reordering, we can assume without loss of generality that $\lambda_1>\lambda_2$.
	
	Now assume that $\sigma$ does not satisfy the strong BE-inequalities at some $\widetilde V\in\Symp(n)$ with eigenvalues $\widetilde\lambda_1,\dotsc,\widetilde\lambda_n$. Then for some pair $i\neq j$ using injectivity of $\sigmahat$, it holds
	\begin{align}
		(\widetilde\lambda_i-\widetilde\lambda_j)\. (\sigma_i(\widetilde\lambda)-\sigma_j(\widetilde\lambda)) < 0\,.
	\end{align}
	After possible reordering, we can assume without loss of generality that $i=1$, $j=2$ and that $\widetilde\lambda_1>\widetilde\lambda_2$.
	
	Now, consider the mappings
	\begin{equation}
		\varphi\col\Rp^n\to\R\,, \quad\varphi(x) = \sigma_1(x)-\sigma_2(x)
		\qquad\text{and}\qquad
		\gamma\col[0,1]\to\Rp^n\,,\quad\gamma(t)=t\.\widetilde\lambda+(1-t)\.\lambda\,.
	\end{equation}
	Then
	\begin{align}
		\varphi(\gamma(0)) = \varphi(\lambda) > 0
		\qquad\text{and}\qquad
		\varphi(\gamma(1)) = \varphi(\widetilde\lambda) < 0\,,
	\end{align}
	and since $\varphi\circ\gamma$ is continuous, there must be $\widehat\lambda\in\Rp^n$ such that $\varphi(\widehat\lambda)=0$. Since the set
	\begin{align}
		\{x\in\Rp^n \setvert x_1 > x_2\}
	\end{align}
	is convex, the straight line $\gamma\col[0,1]\to\Rp^n$ connecting $\lambda$ and $\widetilde\lambda$ conserves the strict ordering of the first two components. Thus $\widehat\lambda_1>\widehat\lambda_2$ and therefore
	\begin{align}
		\widehat\lambda_1 \neq \widehat\lambda_2
		\qquad\text{and}\qquad
		\sigma_1(\widehat\lambda) = \sigma_2(\widehat\lambda)\,,
	\end{align}
	in contradiction to \eqref{eq:semiInvertibility}.
\end{proof}
\begin{remark}\label{remark:correctLinearResponse}
	The requirement of satisfying the strong BE-inequalities at a single point with simple eigenvalues is always satisfied if the stress response is Legendre-Hadamard elliptic at $V=\id$ because strict ellipticity at the identity implies that the principal stresses $\sigma$ are ordered compatibly to the principal stretches $\lambda$ in a neighborhood of the reference configuration. This is nothing else but the correct linear elastic response in the natural state, cf.\ Appendix \ref{appendix:linearElasticity}.
	Indeed for linear elasticity $\mu>0$ is already sufficient for (BE$^+$), while ellipticity also requires $2\.\mu+\lambda>0$ and monotonicity of the stress-strain relation is equivalent to $\mu>0$ and $2\.\mu+3\.\lambda>0$, cf.\ Appendix~\ref{appendix:linearElasticity}.
\end{remark}
\begin{remark}
	Condition \eqref{eq:semiInvertibility} is also implied by (and is, in fact, equivalent to) the \emph{semi-invertibility} of the Cauchy stress response \cite{dunn1984elastic, agn_thiel2019empirical}. The requirement of invertibility in Proposition \ref{prop:injectivityBE} can therefore be weakened to semi-invertibility.
\end{remark}
\begin{example}\label{example:counterInvertibility}
	It is interesting (and for us surprising) to note that the invertibility (i.e.\ non-singularity) of $\D f(\lambda)\col\R^n\to\R^n$ does not immediately imply the invertibility of the derivative $\D\Sigma_f\col\Sym(n)\to\Sym(n)$ at $D =\diag(\lambda)$. For example, we consider the symmetric function
	\begin{equation}
		f(\lambda_1,\lambda_2)=\matr{\lambda_1^2+\lambda_2\\\lambda_1+\lambda_2^2},\qquad \D f(\lambda_1,\lambda_2)=\matr{2\.\lambda_1&1\\1&2\.\lambda_2}.
	\end{equation}
	Then for 
	\begin{equation}
		\lambda=\matr{1\\0}\qquad\implies\qquad \D f(1,0)=\matr{2&1\\1&0},\qquad \left[\D f(1,0)\right]\inv=\matr{0&1\\1&-2},
	\end{equation}
	i.e.\ $\D f(1,0)$ is invertible. We show that $\D\Sigma_f(B)$ is not invertible at $B=\diag(1,0)$ because
	\begin{equation}
		\D\Sigma_f\left(\diag(1,0)\right).\matr{0&-1\\-1&0}=\matr{0&0\\0&0}.\label{eq:invertibilityCounterExample}
	\end{equation}

	For $t\in\R$, we consider the rotation matrix
	\begin{equation}
		Q_t=\matr{\cos t&-\sin t\\\sin t&\cos t}\qquad\text{with}\qquad Q_0=\id\,.
	\end{equation}

	With Lemma \ref{lemma:existenceAndUniquenessOfMatrixFunction}, we compute
	\begin{equation}
		\Sigma_f(Q_t^T\.D\.Q_t)=Q_t^T\.\Sigma_f(\diag(1,0))\.Q_t=Q_t^T\.\diag(f(1,0))\.Q_t=Q_t^T\.\id\.Q_t=\id\,.
	\end{equation}
	Thus the expression is constant and it holds
	\begin{align}
		0&=\ddt\left.\Sigma_f(Q_t^T\.D\.Q_t)\right|_{t=0}
		=\D\Sigma_f\left(Q_t^T\.D\.Q_t\bigr|_{t=0}\right).\left(\ddt\left[Q_t^T\.D\.Q_t\right]_{t=0}\right)\\
		&=\D\Sigma_f(\diag(1,0)).\left(\ddt\left[Q_t^T\.\diag(1,0)\.Q_t\right]_{t=0}\right).\notag
	\end{align}
	Last, we show that
	\begin{align}
		\ddt\left[Q_t^T\.\diag(1,0)\.Q_t\right]_{t=0}&=
		\ddt\left[\matr{\cos t&\sin t\\-\sin t&\cos t}\matr{1&0\\0&0}\matr{\cos t&-\sin t\\\sin t&\cos t}\right]_{t=0}\notag\\
		&=\ddt\left[\matr{\cos t&\sin t\\-\sin t&\cos t}\matr{\cos t&-\sin t\\0&0}\right]_{t=0}\notag\\
		&=\ddt\left[\matr{(\cos t)^2&-\sin t\.\cos t\\-\sin t\.\cos t&(\sin t)^2}\right]_{t=0}\notag\\
		&=\left.\matr{-2\.\cos t\.\sin t&-(\cos t)^2+(\sin t)^2\\-(\cos t)^2+(\sin t)^2&2\.\sin t\.\cos t}\right|_{t=0}
		=\matr{0&-1\\-1&0}
	\end{align}
	which concludes \eqref{eq:invertibilityCounterExample}. In addition, we note that this example also counts for the potential case with $\fpot\col\R^2\to\R$ and
	\begin{equation}
		\fpot(\lambda_1,\lambda_2)=\frac13\.\lambda_1^3+\frac13\.\lambda_2^3+\lambda_1\.\lambda_2\qquad\implies\qquad f=\nabla\fpot\,.
	\end{equation}
\end{example}
It should be observed that in general, this loss of invertibility of $\D\Sigma_f$ necessarily emerges in cases where $f(\lambda_1,\lambda_2)=(\sigma_1,\sigma_2)$ with $\lambda_1\neq\lambda_2$ but $\sigma_1=\sigma_2$, which would be excluded by the strong Baker-Ericksen inequalities. Note that in the above example, the eigenvalues of $Q_t^TD\.Q_t$  remain constant for $D=\diag(\lambda_1,\lambda_2)$, while the eigenvector basis is rotated according to $Q_t$. Due to the double eigenvalue $\sigma_1=\sigma_2$ of $\Sigma_f(D)$, rotating the eigenvector basis according to $Q_t^T\Sigma_f(D)\.Q_t=\Sigma_f(Q_t^TD\.Q_t)$ has no effect, thus a \enquote{linearized eigenvector rotation} changes $D$, but not $\Sigma_f(D)$, which leads to the zero directional derivative in \eqref{eq:invertibilityCounterExample}. The derivative $\D f(\lambda_1,\lambda_2)$, on the other hand, linearizes only the change of eigen\textit{values} and might (as is the case here) be invertible. Note, however, that $f\colon\R^2\to\R^2$ itself is not a bijective function in the general setting described above, i.e.\ where $\lambda_1\neq\lambda_2$ but $\sigma_1=\sigma_2$, since $f(\lambda_1,\lambda_2)=(\sigma_1,\sigma_2)=(\sigma_2,\sigma_1)=f(\lambda_2,\lambda_1)$; in this case, $f(1,0)=f(0,1)=(1,1)$.

%
%
%
%
%
\section{Conclusion}
An overview of the most important statements between matrix- and vector-representation is shown in Table \ref{fig:conclusion1} for the potential case as well as in Table \ref{fig:conclusion2} for the non-potential case.

\begin{table}[h!]
		\begin{tabular}{C{0.05\textwidth} L{0.28\textwidth}|C{0.15\textwidth}|L{0.42\textwidth}}
		&matrix-representation &  & vector-representation\\
		\hline
		\rdelim\{{13}{0cm}&$V\to W(V)$ (strictly) convex	& $\begin{matrix}\text{Davis-Ball \cite{davis1957all,ball1976convexity}}\\\iff\end{matrix}$		& $(\lambda_1,			\lambda_2,\lambda_3)\to \fpot(\lambda_1,\lambda_2,\lambda_3)$ (strictly) convex\\
		\cline{2-4}
		\rdelim.{13}{1cm}[\rotatebox{90}{potential case}]&$V\to \D_VW(V)$ \mbox{(strictly) monotone}	& $\begin{matrix}\text{Hill \cite{hill1968constitutive}}\\\iff\end{matrix}$		& $(\lambda_1,\lambda_2,\lambda_3)\to\nabla\fpot(\lambda_1,\lambda_2,\lambda_3)\in\R^3$ \mbox{(strictly) monotone}\\
		\cline{2-4}
		&$W\col\Sym(3)\cong\R^6\to\R$ differentiable	& $\begin{matrix}\text{Lewis \cite{lewis1996derivatives}}\\\iff\end{matrix}$		& $\fpot\col\R^3\to\R$\newline differentiable\\
		\cline{2-4}
		&$\D^2_VW\in\Symp(6)$ everywhere & $\begin{matrix}\text{Hill \cite{hill1968constitutive}}\\\iff\end{matrix}$		& $\D\.\grad\fpot\in\Symp(3)$ everywhere\\
		\cline{2-4}
		&$\D W\col\Sym(3)\cong\R^6\to\R^6$ invertible	& $\begin{matrix}\text{Theorem \ref{theorem:invertibility}}\\\iff\end{matrix}$		& $\nabla \fpot\col\R^3\to\R^3$ invertible\\
		\cline{2-4}
		&$\det_{\R^{6\times 6}}\D^2W\neq 0$	& $\begin{matrix}\text{Example \ref{example:counterInvertibility}}\\\centernot\iff\end{matrix}$		& $\det_{\R^{3\times 3}}\.\D\.\nabla \fpot\neq 0$
	\end{tabular}
	\caption{Comparison between matrix- and vector-representation of various conditions discussed in this paper for the potential case with $W\col\Sym(3)\cong\R^6\to\R$ and $\fpot\col\R^3\to\R$ and $W(V)=\fpot(\lambda_1,\lambda_2,\lambda_3)$.}\label{fig:conclusion1}
\end{table}
\begin{table}[h!]
	\begin{tabular}{C{0.05\textwidth} L{0.28\textwidth}|C{0.15\textwidth}|L{0.42\textwidth}}
		&matrix-representation &  & vector-representation\\
		\hline
		\rdelim\{{10}{0cm}&$V\to\Sigma_{f}(V)$ \mbox{(strictly) monotone}	& $\begin{matrix}\text{Theorem \ref{theorem:mainResult}}\\\iff\\\text{and Hill \cite{hill1970constitutive}}\end{matrix}$		& $(\lambda_1,\lambda_2,\lambda_3)\to f(\lambda_1,\lambda_2,\lambda_3)\in\R^3$ \newline (strictly) monotone\\
		\cline{2-4}
		\rdelim.{10}{1cm}[\rotatebox{90}{non-potential case}]&$\Sigma_{f}\in C^k(\Sym(3))$	& $\begin{matrix}\text{\v{S}ilhav{\`y} \cite{silhavy2013mechanics}}\\\iff\end{matrix}$	& $f\in C^k(\R^3)$\\ 
		\cline{2-4}
		&$\Sigma_{f}\col\Sym(3)\cong\R^6\to\R^6$ invertible	& $\begin{matrix}\text{Theorem \ref{theorem:invertibility}}\\\iff\end{matrix}$		& $f\col\R^3\to\R^3$ invertible\\
		\cline{2-4}
		&$\sym \D\Sigma_{f}\in\Symp(6)$	everywhere & $\begin{matrix}\text{future work}\\\iff\end{matrix}$		& $\sym \D f\in\Symp(3)$ everywhere\\
		\cline{2-4}
		&$\det_{\R^{6\times 6}}\D\Sigma_{f}\neq 0$	& $\begin{matrix}\text{Example \ref{example:counterInvertibility}}\\\centernot\iff\end{matrix}$	& $\det_{\R^{3\times 3}}\D f\neq 0$
	\end{tabular}
	\caption{Comparison between matrix- and vector-representation of various conditions discussed in this paper for the non-potential case with $\Sigma_f\col\Sym(3)\cong\R^6\to\Sym(3)\cong\R^6$ and $f\col\R^3\to\R^3$ with $	\Sigma(Q^T\diag(\lambda_1,\dotsc,\lambda_n)\.Q) = Q^T\diag(f(\lambda_1,\dotsc,\lambda_n))\.Q$.}\label{fig:conclusion2}
\end{table}

%
%
%
%
\section*{Acknowledgement}
This paper has been on our desk for a long time.
We thank Chandrashekhar S.\ Jog (Indian Institute of Science) for continued discussions on the meaning and interpretation of the TSTS-M$^+$ inequality which stimulated part of this research, Ray Ogden for help in clarifying Rodney Hill's development and all our friends and colleagues whom we contacted over the years as to whether our main theorem is already known.

\footnotesize
\section{References}
\printbibliography[heading=none]

@article{agn_neff2015exponentiatedI,
	title={The exponentiated {H}encky-logarithmic strain energy. {P}art {I}: {C}onstitutive issues and rank-one convexity},
	author={Neff, Patrizio and Ghiba, Ionel-Dumitrel and Lankeit, Johannes},
	journal={Journal of Elasticity},
	volume={121},
	number={2},
	pages={143--234},
	year={2015},
	publisher={Springer},
	doi={10.1007/s10659-015-9524-7},
	class =	{mathematics}
}

@article{agn_thiel2019empirical,
  title =	{Do we need {T}ruesdell's empirical inequalities? {O}n the coaxiality of stress and stretch},
  author =	{Thiel, Christian and Voss, Jendrik and Martin, Robert J. and Neff, Patrizio},
  journal =	{International Journal of Non-Linear Mechanics},
  year =	{2019},
  publisher={Elsevier},
  note =	{\availableatarxiv{1812.03053}},
  doi=		{10.1016/j.ijnonlinmec.2019.02.004},
  class =	{mathematics}
}

@article{ball1976convexity,
  title =		{Convexity conditions and existence theorems in nonlinear elasticity},
  author =		{Ball, J. M.},
  journal =		{Archive for Rational Mechanics and Analysis},
  volume =		{63},
  number =		{4},
  pages =		{337--403},
  year =		{1976},
  publisher =	{Springer}
}

@article{truesdell1975inequalities,
  title =		{Inequalities sufficient to ensure semi-invertibility of isotropic functions},
  author =		{Truesdell, C. and Moon, H.},
  journal =		{Journal of Elasticity},
  volume =		{5},
  number =		{34},
  year =		{1975},
  publisher =	{Springer}
}

@article{dunn1984elastic,
  title =		{Elastic materials of coaxial type and inequalities sufficient to ensure strong ellipticity},
  author =		{Dunn, J. E.},
  journal =		{International Journal of Solids and Structures},
  volume =		{20},
  number =		{5},
  pages =		{417--427},
  year =		{1984},
  publisher =	{Elsevier}
}

@article{friedland1981convex,
  title =		{Convex spectral functions},
  author =		{Friedland, Shmuel},
  journal =		{Linear and Multilinear Algebra},
  volume =		{9},
  number =		{4},
  pages =		{299--316},
  year =		{1981},
  publisher =	{Taylor \& Francis}
}

@article{ogden1972large,
  title =		{Large deformation isotropic elasticity-on the correlation of theory and experiment for incompressible rubberlike solids},
  author =		{Ogden, R. W.},
  journal =		{Proceedings of the Royal Society of London. A. Mathematical and Physical Science},
  volume =		{326},
  number =		{1567},
  pages =		{565--584},
  year =		{1972}
}

@article{coleman1959thermostatics,
  title={On the thermostatics of continuous media},
  author={Coleman, Bernard D and Noll, Walter},
  journal={Archive for Rational Mechanics and Analysis},
  volume={4},
  number={1},
  pages={97--128},
  year={1959},
  publisher={Springer}
}

@article{davis1957all,%Original von Davis Chandler
  title={All convex invariant functions of hermitian matrices},
  author={Davis, Chandler},
  journal={Archiv der Mathematik},
  volume={8},
  number={4},
  pages={276--278},
  year={1957},
  publisher={Springer}
}

@article{rivin2002another,%Alternativer Beweis
  title={Another simple proof of a theorem of Chandler Davis},
  author={Rivin, Igor},
  journal={\arxivjournal{math/0208223}},
  year={2002}
}

@article{grabovsky2005generalization,%Verallgemeinerung auf Gruppen
  title={A generalization of the Chandler Davis convexity theorem},
  author={Grabovsky, Yury and Hijab, Omar},
  journal={Advances in Applied Mathematics},
  volume={34},
  number={1},
  pages={192--212},
  year={2005},
  publisher={Elsevier}
}

@article{lewis1996group,%Verweis auf Davis und Ball bei Theorem 8.1
  title={Group invariance and convex matrix analysis},
  author={Lewis, Adrian Stephen},
  journal={SIAM Journal on Matrix Analysis and Applications},
  volume={17},
  number={4},
  pages={927--949},
  year={1996},
  publisher={SIAM}
}

@article{hill1968constitutive,
	title={On constitutive inequalities for simple materials—I},
	author={Hill, Rodney},
	journal={Journal of the Mechanics and Physics of Solids},
	volume={16},
	number={4},
	pages={229--242},
	year={1968},
	publisher={Elsevier}
}

@inproceedings{hill1970constitutive,%Hill nicht Potentzial ab S.466
  title={Constitutive inequalities for isotropic elastic solids under finite strain},
  author={Hill, R},
  booktitle={Proceedings of the Royal Society of London A: Mathematical, Physical and Engineering Sciences},
  volume={314},
  number={1519},
  pages={457--472},
  year={1970},
  organization={The Royal Society}
}

@inproceedings{theobald1975inequality,%Abschätzung durch geordnete Eigenwerte
  title={An inequality for the trace of the product of two symmetric matrices},
  author={Theobald, C Mo},
  booktitle={Mathematical Proceedings of the Cambridge Philosophical Society},
  volume={77},
  number={02},
  pages={265--267},
  year={1975},
}

@BOOK{antman,%Nachschlagewerk p.407,470
  AUTHOR       = {Antman, Stuart S.},
  TITLE        = {Nonlinear Problems of Elasticity},
  PUBLISHER    = {Springer},
  YEAR         = {2005},
  ADDRESS      = {New York},
  Edition	=	{2}
}

@book{ogden1997non,
  title={Non-Linear Elastic Deformations},
  author={Ogden, Raymond W},
  year={1997},
  publisher={Courier Corporation}
}

@article{leblond1992constitutive,
  title={A constitutive inequality for hyperelastic materials in finite strain},
  author={Leblond, JB},
  journal={European Journal of Mechanics. A. Solids},
  volume={11},
  number={4},
  pages={447--466},
  year={1992},
  publisher={Elsevier}
}

@article{birkhoff1946three,%Birkhoff-von-Neumann theorem
  title={Three observations on linear algebra},
  author={Birkhoff, Garrett},
  journal={Universidad Nacional de Tucum{\'a}n Revista Serie A },
  volume={5},
  pages={147--151},
  year={1946}
}

@article{von1953certain,%Birkhoff-von-Neumann theorem
  title={A certain zero-sum two-person game equivalent to the optimal assignment problem},
  author={von Neumann, John},
  journal={Contributions to the Theory of Games},
  volume={2},
  pages={5--12},
  year={1953}
}

@article{lewis1996derivatives,%W und Phi differentable
  title={Derivatives of spectral functions},
  author={Lewis, Adrian Stephen},
  journal={Mathematics of Operations Research},
  volume={21},
  number={3},
  pages={576--588},
  year={1996},
  publisher={INFORMS}
}

@article{weber2010davis,
  title={Davis’ convexity theorem and extremal ellipsoids},
  author={Weber, Matthias J and Schr{\"o}cker, Hans-Peter},
  journal={Contributions to Algebra and Geometry},
  volume={51},
  number={1},
  pages={263--274},
  year={2010}
}

@article{lewisz1998hyperbolic,
	title={Hyperbolic polynomials and convex analysis},
	author={Bauschke, Heinz H and G{\"u}ler, Osman and Lewis, Adrian S and Sendov, Hristo S},
	journal={Canadian Journal of Mathematics},
	volume={53},
	number={3},
	pages={470--488},
	year={2001},
	publisher={Cambridge University Press}
}

@article{silhavy2015convexity,
  title={The convexity of $C$ to $h(\det C)$},
  author={Silhavy, M},
  journal={Technische Mechanik},
  volume={35},
  number={1},
  pages={60--61},
  year={2015}
}

@article{richter1958abschatzung,
  title={Zur Absch{\"a}tzung von Matrizennormen},
  author={Richter, Hans},
  journal={Mathematische Nachrichten},
  volume={18},
  number={1-6},
  pages={178--187},
  year={1958},
  publisher={Wiley Online Library}
}

@inproceedings{ogden1977inequalities,
  title={Inequalities associated with the inversion of elastic stress-deformation relations and their implications},
  author={Ogden, Raymond W},
  booktitle={Mathematical Proceedings of the Cambridge Philosophical Society},
  volume={81},
  number={02},
  pages={313--324},
  year={1977},
  organization={Cambridge University Press}
}

@article{mirsky1959trace,
  title={On the trace of matrix products},
  author={Mirsky, Leon},
  journal={Mathematische Nachrichten},
  volume={20},
  number={3-6},
  pages={171--174},
  year={1959},
  publisher={Wiley Online Library}
}

@article{neumann1937some,
  title={Some matrix inequalities and metrization of matrix space},
  author={von Neumann, J},
  journal={Tomsk University Review},
  volume={1},
  number={11},
  pages={286--300},
  year={1937}
}

@article{rivin2010golden,%Short proof
  title={Golden--Thompson from Davis},
  author={Rivin, Igor},
  journal={\arxivjournal{1010.2193}},
  year={2010},
}

@article{marques1982isotropie,
  title={Isotropie et convexit{\'e} dans l'espace des tenseurs sym{\^e}triques},
  author={Marques, MDPM and Moreau, JJ},
  journal={Travaux du Seminaire d'Analyse Convexe, Montpellier},
  volume={12},
  number={6},
  year={1982}
}

@book{silhavy2013mechanics,
  title={The Mechanics and Thermodynamics of Continuous Media},
  author={Silhavy, Miroslav},
  year={2013},
  publisher={Springer Science \& Business Media}
}

@article{neff2016geometry,
  title={Geometry of logarithmic strain measures in solid mechanics},
  author={Neff, Patrizio and Eidel, Bernhard and Martin, Robert J},
  journal={Archive for Rational Mechanics and Analysis},
  volume={222},
  number={2},
  pages={507--572},
  year={2016},
  publisher={Springer}
}

@article{hencky1929superpositionsgesetz,
  title={Das Superpositionsgesetz eines endlich deformierten relaxationsf{\"a}higen elastischen Kontinuums und seine Bedeutung f{\"u}r eine exakte Ableitung der Gleichungen f{\"u}r die z{\"a}he Fl{\"u}ssigkeit in der Eulerschen Form},
  author={Hencky, Heinrich},
  journal={Annalen der Physik},
  volume={394},
  number={6},
  pages={617--630},
  year={1929},
  publisher={Wiley Online Library}
}

@article{neff2014grioli,
  title={On Grioli’s minimum property and its relation to Cauchy’s polar decomposition},
  author={Neff, Patrizio and Lankeit, Johannes and Madeo, Angela},
  journal={International Journal of Engineering Science},
  volume={80},
  pages={209--217},
  year={2014},
  publisher={Elsevier}
}

@article{golden1965lower,
  title={Lower bounds for the Helmholtz function},
  author={Golden, Sidney},
  journal={Physical Review},
  volume={137},
  number={4B},
  pages={B1127},
  year={1965},
  publisher={APS}
}

@article{thompson1965inequality,
  title={Inequality with applications in statistical mechanics},
  author={Thompson, Colin J},
  journal={Journal of Mathematical Physics},
  volume={6},
  number={11},
  pages={1812--1813},
  year={1965},
  publisher={AIP}
}

@article{baker1954inequalities,
  title={Inequalities restricting the form of the stress-deformation relations for isotropic elastic solids and Reiner-Rivlin fluids},
  author={Baker, M and Ericksen, JL},
  journal={Journal of the Washington Academy of Sciences},
  volume={44},
  number={2},
  pages={33--35},
  year={1954},
  publisher={JSTOR}
}

@article{neff2014logarithmic,
  title={A logarithmic minimization property of the unitary polar factor in the spectral and Frobenius norms},
  author={Neff, Patrizio and Nakatsukasa, Yuji and Fischle, Andreas},
  journal={SIAM Journal on Matrix Analysis and Applications},
  volume={35},
  number={3},
  pages={1132--1154},
  year={2014},
  publisher={SIAM}
}

@article{martin2015some,
  title={Some remarks on the monotonicity of primary matrix functions on the set of symmetric matrices},
  author={Martin, Robert J and Neff, Patrizio},
  journal={Archive of Applied Mechanics},
  volume={85},
  number={12},
  pages={1761--1778},
  year={2015},
  publisher={Springer}
}

@article{spector2015note,
  title={A note on the convexity of C $\mapsto$ h(detC)},
  author={Spector, Scott J},
  journal={Journal of Elasticity},
  volume={118},
  number={2},
  pages={251--256},
  year={2015},
  publisher={Springer}
}

@article{johnson1994large,
  title={Large strain viscoelastic constitutive models for rubber, part I: Formulations},
  author={Johnson, AR and Quigley, CJ and Mead, JL},
  journal={Rubber Chemistry and Technology},
  volume={67},
  number={5},
  pages={904--917},
  year={1994}
}

@article{le1990fonctions,
  title={Sur les fonctions de matrices convexes et isotropes},
  author={Le Dret, H},
  journal={Comptes Rendus de l'Acad{\^e}mie des Sciences},
  volume={310},
  pages={617--620},
  year={1990}
}

@book{bhatia2013matrix,
  title={Matrix Analysis},
  author={Bhatia, Rajendra},
  volume={169},
  year={2013},
  publisher={Springer Science \& Business Media}
}

@incollection{rivlin1997stress,
  title={Stress-deformation relations for isotropic materials},
  author={Rivlin, Ronald Samuel and Ericksen, JL},
  booktitle={Collected Papers of RS Rivlin},
  pages={911--1013},
  year={1997},
  publisher={Springer}
}

@article{boyce2000constitutive,
  title={Constitutive models of rubber elasticity: a review},
  author={Boyce, Mary C and Arruda, Ellen M},
  journal={Rubber Chemistry and Technology},
  volume={73},
  number={3},
  pages={504--523},
  year={2000}
}

@article{richter1948isotrope,
  title={Das isotrope Elastizit{\"a}tsgesetz},
  author={Richter, H},
  journal={ZAMM-Journal of Applied Mathematics and Mechanics/Zeitschrift f{\"u}r Angewandte Mathematik und Mechanik},
  volume={28},
  number={7-8},
  pages={205--209},
  year={1948},
  publisher={Wiley Online Library}
}

@article{gerschgorin1931sa,
  title={Über die Abgrenzung der Eigenwerte einer Matrix},
  author={Gershgorin, Semyon Aranovich},
  journal={Bulletin de l'Acad\'emie des Sciences de l'URSS. Classe des Sciences Math\'ematiques },
  number={6},
  pages={749--754},
  year={1931},
}

@Article{Lewis03,
    Author = {Lewis, Adrian S},
    Title = {The mathematics of eigenvalue optimization},
    Journal = {Mathematical Programming},
    Volume = {97},
    Number = {1-2 (B)},
    Pages = {155-176},
    Year = {2003}
}

@InCollection{Lewis96,
    Author = {Lewis, Adrian S and Overton, Michael L},
    Title = {Eigenvalue optimization},
    BookTitle = {Acta Numerica Vol. 5},
    Pages = {149-190},
    Year = {1996},
    Publisher = {Cambridge University Press},
}

@Article{Lewis96b,
    Author = {Lewis, Adrian S},
    Title = {Convex analysis on the Hermitian matrices},
    Journal = {SIAM Journal on Optimization},
    Volume = {6},
    Number = {1},
    Pages = {164-177},
    Year = {1996}
}

@book{wang1973introduction,
	title={Introduction to Rational Elasticity},
	author={Wang, Chao-Cheng and Truesdell, Clifford},
	volume={1},
	year={1973},
	publisher={Springer Science \& Business Media}
}

@article{jog2013conditions,
   author={Jog, Chandrashekhar S. and Patil, Kunal},
  title={Conditions for the onset of elastic and material instabilities in hyperelastic materials},
  journal={Archive of Applied Mechanics},
    volume={83},
  pages={1-24},
  year={2013}
}

@book{jog2015continuum,
	title={Continuum Mechanics},
	author={Jog, Chandrashekhar S},
	volume={1},
	year={2015},
	publisher={Cambridge University Press}
}

@article{guo2006application,
	title={Application of a new constitutive model for the description of rubber-like materials under monotonic loading},
	author={Z. Guo and L.J. Sluys},
	journal={International Journal of Solids and Structures},
	volume={43},
	number={9},
	pages={2799--2819},
	year={2006}
}

@article{clayton2014analysis,
	title={Analysis of intrinsic stability criteria for isotropic third-order {G}reen elastic and compressible {N}eo-{H}ookean solids},
	author={J.D. Clayton and K.M. Bliss},
	journal={Mechanics of Materials},
	volume={68},
	pages={104--119},
	year={2014}
}

@article{lehmich2012convexity,
  title={On the convexity of the function {$C\rightarrow f (\text{det} C)$} on positive definite matrices},
  author={S. Lehmich and P. Neff and J. Lankeit},
  journal={ Mathematics and Mechanics of Solids},
  volume={19},
  pages={369-375},
  year={2014}
}

@book{hardy1952inequalities,
	title={Inequalities},
	author={Hardy, Godfrey Harold and Littlewood, John Edensor and P{\'o}lya, George},
	year={1952},
	publisher={Cambridge University Press}
}

@article{baaser2026hyperelastic,
	title={Hyperelastic stability landscape: A check for Hill stability of isotropic, incompressible hyperelasticity depending on material parameters},
	author={Baaser, Herbert},
	journal={Journal of Elasticity},
	volume={158},
	number={1},
	pages={8},
	year={2026},
	publisher={Springer}
}

@article{wollner2025search,
	title={In search of constitutive conditions in isotropic hyperelasticity: polyconvexity versus true-stress-true-strain monotonicity},
	author={Wollner, Maximilian P and Holzapfel, Gerhard A and Neff, Patrizio},
	journal={Journal of the Mechanics and Physics of Solids},
	pages={106465},
	year={2025},
	publisher={Elsevier}
}

@article{neff2025hypo,
	title={Hypo-elasticity, Cauchy-elasticity, corotational stability and monotonicity in the logarithmic strain},
	author={Neff, Patrizio and Holthausen, Sebastian and d’Agostino, Marco Valerio and Bernardini, Davide and Sky, Adam and Ghiba, Ionel-Dumitrel and Martin, Robert J},
	journal={Journal of the Mechanics and Physics of Solids},
	volume={202},
	pages={106074},
	year={2025},
	publisher={Elsevier}
}

@article{d2025constitutive,
	title={A constitutive condition for idealized isotropic Cauchy elasticity involving the logarithmic strain},
	author={d’Agostino, Marco Valerio and Holthausen, Sebastian and Bernardini, Davide and Sky, Adam and Ghiba, Ionel-Dumitrel and Martin, Robert J and Neff, Patrizio},
	journal={Journal of Elasticity},
	volume={157},
	number={1},
	pages={23},
	year={2025},
	publisher={Springer}
}

@article{ghiba2026polyconvexity,
	title={Polyconvexity implies Hill’s inequality in SL (2)},
	author={Ghiba, Ionel-Dumitrel and Wollner, Maximilian P and Neff, Patrizio},
	journal={European Journal of Mechanics-A/Solids},
	pages={106296},
	year={2026},
	publisher={Elsevier}
}

@article{klein2026polyconvexity,
	title={Polyconvexity does not imply true-stress-true-strain monotonicity in the incompressible three-dimensional case},
	author={Klein, Dominik K and Wollner, Maximilian P and Neff, Patrizio},
	journal={\arxivjournal{2607.06568}},
	year={2026}
}

@article{wollner2026concurrent,
	title={Concurrent enforcement of polyconvexity and true-stress-true-strain monotonicity in incompressible isotropic hyperelasticity: application to neural network constitutive models},
	author={Wollner, Maximilian P and Klein, Dominik K and Baaser, Herbert and Holzapfel, Gerhard A and Neff, Patrizio},
	journal={to appear in Journal of the Mechanics and Physics of Solids},
	note =	{\availableatarxiv{2605.20031}},
	year={2026}
}

@article{chadwick1971theorem,
	title={A theorem of tensor calculus and its application to isotropic elasticity},
	author={Chadwick, P and Ogden, RW334656},
	journal={Archive for Rational Mechanics and Analysis},
	volume={44},
	number={1},
	pages={54--68},
	year={1971},
	publisher={Springer}
}

@article{mihai2017characterize,
	title={How to characterize a nonlinear elastic material? A review on nonlinear constitutive parameters in isotropic finite elasticity},
	author={Mihai, L Angela and Goriely, Alain},
	journal={Proceedings of the Royal Society A: Mathematical, Physical and Engineering Sciences},
	volume={473},
	number={2207},
	pages={20170607},
	year={2017},
	publisher={The Royal Society Publishing}
}

@article{anssari2023large,
	title={Large isotropic elastic deformations: on a comprehensive model to correlate the theory and experiments for incompressible rubber-like materials},
	author={Anssari-Benam, Afshin},
	journal={Journal of Elasticity},
	volume={153},
	number={2},
	pages={219--244},
	year={2023},
	publisher={Springer}
}
\small
\begin{appendix}
\section{Appendix}
%
%
\subsection{Hill's algebraic proof for ordered sets of eigenvalues}\label{appendix:OriginalProofHill}
	At the first reading Hill \cite{hill1970constitutive} seems to prove the simple inequality chain \eqref{eq:ineqChain}
	\begin{align}
		\iprod{\Sigma_f(S)-\Sigma_f(T), S-T}\geq\iprod{f(x)-f(y),x-y}\geq 0\,.\tag{\ref{eq:ineqChain}}
	\end{align}
	 This interpretation is motivated based on Hill's index notation. In the following, we will show that this statement does not hold for the general case, but only for matrices with consistently sorted eigenvalue lists, i.e.\ decreasing order. However, it might be surmised that Hill did not intend to prove \eqref{eq:ineqChain}, but rather the similar proof to the potential case, given in Remark \ref{remark:HillPotential}.
	 
	Let $\sigma,\overline\sigma,\eps,\overline\eps\in\Sym(n)$ be given as arbitrary matrices\footnote{Here, the tensors $\eps$ and $\overline\eps$ are representatives of some nonlinear strain tensors, i.e.\ $\eps=\log V$ or $\eps=V-\id\,.$}. Hill calculates
	\begin{align}
		\iprod{\sigma-\overline\sigma,\eps-\overline\eps}&=\iprod{\sigma,\eps}-\iprod{\sigma,\overline\eps}-\iprod{\overline\sigma,\eps}+\iprod{\overline\sigma,\overline\eps}=\sum_i^n(\overline\sigma_i\.\overline\eps_i+\sigma_i\.\eps_i)-(\iprod{\overline\sigma,\eps}+\iprod{\sigma,\overline\eps})\,,\label{eq:Hill}
	\end{align}
	with $\sigma_i,\eps_i$ the eigenvalues of $\sigma,\eps\,.$ The step 
	\begin{align}
		\iprod{\sigma,\eps}=\sum_i^n\sigma_i\.\eps_i\,,\qquad\iprod{\overline\sigma,\overline\eps}=\sum_i^n\overline\sigma_i\.\overline\eps_i
	\end{align}
	only holds if $\sigma$ and $\eps$ as well as $\overline\sigma$ and $\overline\eps$ are simultaneously diagonalizable.
	However, every isotropic tensor function $\sigma(\eps)$ has this coaxiality property. Next, Hill uses the following statement
	\begin{align}
		\iprod{\overline\sigma,\eps}\leq\sum_i^n\overline\sigma_i\.\eps_i\qquad\text{and}\qquad\iprod{\sigma,\overline\eps}\leq\sum_i^n\sigma_i\.\overline\eps_i\,.\label{eq:HillsIncorrectInequality}
	\end{align}
	This statement only holds if all the eigenvalues $\sigma_i,\overline\sigma_i,\eps_i,\overline\eps_i$ are sorted in the same algebraic order, see Hardy \cite{hardy1952inequalities} as well \citet{theobald1975inequality} or \citet{richter1958abschatzung}\footnote{Hill does not refer to anyone but just indicates that these inequalities can be shown by direct calculations.}, but does not hold for the general case. We can compute a simple two-dimensional counterexample with unordered eigenvalues to disprove this statement
	\begin{align}
		&\overline\sigma=\matr{2&1\\1&2}=Q^T\diag(3,1)\.Q\,,\qquad\eps=\matr{1&0\\0&4}\label{eq:HillCounterExample2D}\\
		\implies\qquad&\iprod{\overline\sigma,\eps}=2\cdot 1+1\cdot 0+1\cdot 0+2\cdot 4=10\,,\qquad\sum_i^2\overline\sigma_i\.\eps_i=3\cdot 1+1\cdot 4=7\,.\notag
	\end{align}
	Obviously \eqref{eq:HillsIncorrectInequality} does not hold. If we restrict  $\sigma,\overline\sigma,\eps,\overline\eps\in\Sym(n)$ as matrices with eigenvalues sorted as
	\begin{align}
	\begin{matrix}
	\sigma_1&\geq&\sigma_2&\geq\dots\geq&\sigma_n\,,\\
	\eps_1&\geq&\eps_2&\geq\dots\geq&\eps_n\,,
	\end{matrix}\qquad\qquad\begin{matrix}
	\overline\sigma_1&\geq&\overline\sigma_2&\geq\dots\geq&\overline\sigma_n\,,\\
	\overline\eps_1&\geq&\overline\eps_2&\geq\dots\geq&\overline\eps_n\,,
	\end{matrix}\label{eq:OrderedHill}
	\end{align}
	Hill's statement and also his remaining proof remains correct. He calculates
	\begin{align}
		\iprod{\sigma-\overline\sigma,\eps-\overline\eps}&=\sum_i^n(\overline\sigma_i\.\overline\eps_i+\sigma_i\.\eps_i)-(\iprod{\overline\sigma,\eps}+\iprod{\sigma,\overline\eps})\notag\\
		&\geq\sum_i^n(\overline\sigma_i\.\overline\eps_i+\sigma_i\.\eps_i-\overline\sigma_i\.\eps_i-\sigma_i\.\overline\eps_i)=\iprod{\left(\begin{matrix}\sigma_1\\\vdots\\\sigma_n\end{matrix}\right)
		- \left(\begin{matrix}\overline\sigma_1\\\vdots\\\overline\sigma_n\end{matrix}\right),\,
\left(\begin{matrix}\eps_1\\\vdots\\\eps_n\end{matrix}\right)
		- \left(\begin{matrix}\overline\eps_1\\\vdots\\\overline\eps_n\end{matrix}\right)}\,.\label{eq:Hill2}
	\end{align}
	He argues that the right-hand side is non-negative due to vector-monotonicity, so the left-hand side must be non-negative, too. For  $\sigma,\overline\sigma,\eps,\overline\eps\in\Sym(n)$ as matrices with eigenvalues sorted as in \eqref{eq:OrderedHill} his proof is correct.
	
	In order to see that \eqref{eq:HillsIncorrectInequality} does not hold in the general case with unsorted eigenvalues, we can generate a simple two-dimensional counterexample to the literal interpretation. Consider
	\begin{align}
		\sigma(\eps)=\eps,\quad\matr{\eps_1\\\eps_2}=\matr{1\\0},\quad\matr{\overline\eps_1\\\overline\eps_2}=\matr{0\\1}\label{eq:Counterexample}
	\end{align}
	\[
		\text{and}\quad\eps=Q^T\diag(\eps_1,\eps_2)\.Q,\quad\overline\eps=\diag(\overline\eps_1,\overline\eps_2),\quad\sigma=Q^T\diag(\sigma_1,\sigma_2)\.Q,\quad\overline\sigma=\diag(\overline\sigma_1,\overline\sigma_2)\,.
	\]
	Take now $Q=\matr{\cos(\alpha)&-\sin(\alpha)\\\sin(\alpha)&\cos(\alpha)}=\matr{\frac{1}{\sqrt{2}} & -\frac{1}{\sqrt{2}}\\ \frac{1}{\sqrt{2}}& \frac{1}{\sqrt{2}}}$ as a rotation with $\alpha=\frac{\pi}{4}$, so that
	\begin{align}
		\eps=Q^T\diag(\eps_1,\eps_2)\.Q=\matr{\frac{\eps_1+\eps_2}{2}&\frac{-\eps_1+\eps_2}{2}\\\frac{-\eps_1+\eps_2}{2}&\frac{\eps_1+\eps_2}{2}}=\matr{\frac{1}{2}&-\frac{1}{2}\\-\frac{1}{2}&\frac{1}{2}}\,.
	\end{align}
	Then we compute
	\begin{align*}
		&\iprod{\sigma-\overline\sigma,\eps-\overline\eps}=\iprod{\matr{\frac{1}{2}&-\frac{1}{2}\\-\frac{1}{2}&\frac{1}{2}}-\matr{0&0\\0&1},\,\matr{\frac{1}{2}&-\frac{1}{2}\\-\frac{1}{2}&\frac{1}{2}}-\matr{0&0\\0&1}}=\left\|\matr{\frac{1}{2}&-\frac{1}{2}\\-\frac{1}{2}&-\frac{1}{2}}\right\|^2=1\,,\\
		&\text{while}\qquad\sum_{i=1}^2(\sigma_i-\overline\sigma_i)(\eps_i-\overline\eps_i)=(1-0)(1-0)+(0-1)(0-1)=2\,.
	\end{align*}
	Obviously, Hill's suggested inequality \eqref{eq:Hill2} is violated and
	\begin{align}
		\iprod{\Sigma_f(S)-\Sigma_f(T),\, S-T}\geq\iprod{f(x)-f(y),\,x-y}
	\end{align}
	cannot hold in general. Since our counterexample is potential, this shows that the estimate \eqref{eq:ineqChain} is already violated in the potential case.
	
	We have shown before that this simple inequality chain \eqref{eq:ineqChain} or  \eqref{eq:Hill2} does not hold in general, but it does so for the subset of matrices $\sigma,\overline\sigma,\eps,\overline\eps\in\Sym(n)$ with consistently sorted eigenvalue lists as in \eqref{eq:OrderedHill}. The literal unsorted reading of his proof fails
	because \eqref{eq:HillsIncorrectInequality} does not hold. That leads to the conclusion that Hill intends to consider only matrices with consistently sorted eigenvalue lists.
	In \cite[Appendix 1 Convex Functions]{ogden1997non} Ogden interprets Hill's proof for vector-monotonicity implies matrix-monotonicity and concludes that it holds for the subset of ordered eigenvalues, cf.\ Appendix \ref{appendix:HillOgdenProof}. Another possibility is, that Hill silently renames his indices, so that his eigenvalues are sorted in the right algebraic order, see the following Appendix \ref{appendix:ImprovedInterpretationHill}.
	
%
%
\subsection{Our interpretation for non-ordered eigenvalues}\label{appendix:ImprovedInterpretationHill}
	It might be surmised that Hill intends to sort his eigenvalues in the right algebraic order. If so, the index in \eqref{eq:HillsIncorrectInequality} should be renamed properly, so that we can use the correct nontrivial estimate \cite{hardy1952inequalities,richter1958abschatzung,theobald1975inequality}, see also \citet[p.171-174]{mirsky1959trace} and von \citet{neumann1937some}
	\begin{align}
		\iprod{\sigma,\overline\eps}\leq{\sum_{i=1}^n}{'}\sigma_i\.\overline\eps_i\,,\qquad\qquad\iprod{\overline\sigma,\eps}\leq{\sum_{i=1}^n}{'}\overline\sigma_i\.\eps_i\,,\label{eq:Theobald}
	\end{align}
	where in $\Sigma'$ it is understood that the eigenvalues on the right-hand side $\sigma_i,\overline\sigma_i,\eps_i,\overline\eps_i$ are sorted as in \eqref{eq:OrderedHill}.
	
	We name the repositioned sorted eigenvalues $\sort\sigma_i,\sort{\overline\sigma}_i,\sort\eps_i,{\widetilde\eps}_i,\;i\in\{1,\dotsc,n\}$ and rewrite \eqref{eq:Theobald} as
	\begin{align}
		\iprod{\overline\sigma,\eps}\leq\sum_{i=1}^n\sort{\overline\sigma}_i\.\sort\eps_i\qquad\text{and}\qquad\iprod{\sigma,\overline\eps}\leq\sum_{i=1}^n\sort\sigma_i\.{\widetilde\eps}_i\,.\label{eq:Theobald2}
	\end{align}
	As the next step we use that $\sigma=\sigma(\eps)$ with $\eps$ and also $\overline\sigma=\overline\sigma(\eps)$ with $\overline\eps$ are sorted in the same algebraic order as shown in Lemma \ref{lemma:vectorMonotonicityImpliesSameAlgebraicOrder}. Therefore, with the associative law we get
	\begin{align}
		\sum_{i=1}^n\sigma_i\.\eps_i=\sum_{i=1}^n\sort\sigma_i\.\sort\eps_i\qquad\text{and}\qquad\sum_{i=1}^n\overline\sigma_i\.\overline\eps_i=\sum_{i=1}^n\sort{\overline\sigma}_i\.{\widetilde\eps}_i\,.
	\end{align}
	We can now modify Hill's calculation in Appendix \ref{appendix:OriginalProofHill} to yield
	\begin{align}
		\iprod{\sigma-\overline\sigma,\eps-\overline\eps}&=\sum_{i=1}^n(\overline\sigma_i\.\overline\eps_i+\sigma_i\.\eps_i)-(\iprod{\overline\sigma,\eps}+\iprod{\sigma,\overline\eps})\notag\\
		&\geq\sum_{i=1}^n(\sort{\overline\sigma}_i\.{\widetilde\eps}_i+\sort\sigma_i\.\sort\eps_i-\sort{\overline\sigma}_i\.\sort\eps_i-\sort\sigma_i\.{\widetilde\eps}_i)=\iprod{\left(\begin{matrix}\sort\sigma_1\\\vdots\\\sort\sigma_n\end{matrix}\right)
		- \left(\begin{matrix}\sort{\overline\sigma}_1\\\vdots\\\sort{\overline\sigma}_n\end{matrix}\right),\,
\left(\begin{matrix}\sort\eps_1\\\vdots\\\sort\eps_n\end{matrix}\right)
		- \left(\begin{matrix}{\widetilde\eps}_1\\\vdots\\{\widetilde\eps}_n\end{matrix}\right)}\,.\label{eq:HillsCalculationNewArranged}
	\end{align}
	The right-hand side is non-negative due to vector-monotonicity, so the left-hand side must be non-negative, too.  
	Note again that the right-hand side of \eqref{eq:HillsCalculationNewArranged} is not \mbox{$\sum(\sigma_i-\overline\sigma_i)(\eps_i-\overline\eps_i)$}, i.e.\ with unordered vectors. It might be reasonably assumed that Hill intends the second interpretation of his proof, but forgot to rename his indices accurately. If not, Hill's proof is correct for $\sigma,\overline\sigma,\eps,\overline\eps\in\Sym(n)$ with eigenvalues $\sigma_i,\overline\sigma_i,\eps_i,\overline\eps_i$ sorted as in \eqref{eq:OrderedHill}.
	Thus, our improved interpretation of Hill's proof can be understood as an extension to arbitrary tensors without algebraic sorted eigenvalues.
	
	For our prior counterexample \eqref{eq:Counterexample} to the literal interpretation we have to reposition the eigenvalues
	\begin{align}
		\sort\sigma=\sort\eps=\eps=\matr{1\\0},\qquad\sort{\overline\sigma}={\widetilde\eps}=\matr{1\\0}.
	\end{align}
	Then we compute
	\begin{align}
		&\iprod{\sigma-\overline\sigma,\eps-\overline\eps}=\iprod{\matr{\frac{1}{2}&-\frac{1}{2}\\-\frac{1}{2}&\frac{1}{2}}-\matr{0&0\\0&1},\,\matr{\frac{1}{2}&-\frac{1}{2}\\-\frac{1}{2}&\frac{1}{2}}-\matr{0&0\\0&1}}=\left\|\matr{\frac{1}{2}&-\frac{1}{2}\\-\frac{1}{2}&-\frac{1}{2}}\right\|^2=1\\
		&\text{while}\qquad\sum_{i=1}^2(\sort\sigma_i-\sort{\overline\sigma}_i)(\sort\eps_i-{\widetilde\eps}_i)=(1-1)(1-1)+(0-0)(0-0)=0\,.\notag
	\end{align}
	Obviously, it holds the inequality
	\begin{align}
		\iprod{\sigma-\overline\sigma,\eps-\overline\eps}\geq\sum_{i=1}^n(\sort\sigma_i-\sort{\overline\sigma}_i)(\sort\eps_i-{\widetilde\eps}_i)\,.
	\end{align}

%
%
\subsection{Hill-Ogden proof}\label{appendix:HillOgdenProof}
	In \cite[Appendix 1 Convex Functions]{ogden1997non} Ogden interprets Hill's proof for vector-monotonicity implies matrix-monotonicity for $n=3$ and concludes that it holds for the subset of ordered eigenvalues similar to our first reading, cf.\ Appendix \ref{appendix:OriginalProofHill}.
	
	We consider the isotropic constitutive law
	\begin{align}
		\sigma=\sigma(S)\,,\qquad\sigma\col\Sym(3)\to\Sym(3)
	\end{align}
	in tensor form and the same law in principal stress-stretch form
	\begin{align}
		\widehat\sigma_i=\widehat\sigma_i(\lambda_1,\lambda_2,\lambda_3)\,,\qquad\widehat\sigma\col\R^3\to\R^3\,,\qquad\text{where}\quad\lambda_i=\lambda_i(S)\,.\label{appendix:ogdenbook2}
	\end{align}
	We assume that $\widehat\sigma$ is strictly vector-monotone
	\begin{align}
		\iprod{\left(\widehat\sigma(\lambda_1,\lambda_2,\lambda_3)-\widehat\sigma(\overline\lambda_1,\overline\lambda_2,\overline\lambda_3)\right),\;\matr{\lambda_1\\\lambda_2\\\lambda_3}-\matr{\overline\lambda_1\\\overline\lambda_2\\\overline\lambda_3}}&>0\notag\\
		\iff\qquad\sum_{i=1}^3\left(\widehat\sigma_i(\lambda_1,\lambda_2,\lambda_3)-\widehat\sigma_i(\overline\lambda_1,\overline\lambda_2,\overline\lambda_3)\right)(\lambda_i-\overline\lambda_i)&>0\label{appendix:ogdenbook3}
	\end{align}
	for arbitrary $\lambda\neq\overline\lambda\in\R^3$. We want to show that then
	\[
		\iprod{\sigma(S)-\sigma(T),\, S-T}>0\qquad\forall\;S,T\in\Sym(3)\,.
	\]
	\begin{proof}
		Let $S\neq T\in\Sym(3)$ be arbitrary given. We denote the in general unsorted real eigenvalues of $S$ and $T$ by
		\begin{equation}
			\lambdamax,\lambdamed,\lambdamin,\qquad\text{respectively}\qquad\olambdamax,\olambdamed,\olambdamin\,.
		\end{equation}
		Now we calculate
		\begin{align}
			\iprod{\sigma(S)-\sigma(T),\, S-T}=\iprod{\sigma(S),S}+\iprod{\sigma(T),T}-\iprod{\sigma(S),T}-\iprod{\sigma(T),S}\,.\label{appendix:ogdenbook4}
		\end{align}
		The first two terms can be immediately written as 
		\begin{equation}
			\iprod{\sigma(S),S}+\iprod{\sigma(T),T}=\iprod{\widehat\sigma(\lambdamax,\lambdamed,\lambdamin),\;\matr{\lambdamax\\\lambdamed\\\lambdamin}}+\iprod{\widehat\sigma(\olambdamax,\olambdamed,\olambdamin),\;\matr{\olambdamax\\\olambdamed\\\olambdamin}}\,,
		\end{equation}
		since $\sigma(S)$ is coaxial with $S$ and $\sigma(T)$ is coaxial with $T$, by isotropy. By strict vector-monotonicity \eqref{appendix:ogdenbook3} $\widehat\sigma$ is ordered as $\lambda=(\lambda_1,\lambda_2,\lambda_3)$ so the latter can be written as
		\begin{align}
			\iprod{\matr{\widehat{\sigma}_{\rm max}(\lambdamax,\lambdamed,\lambdamin)\\
				\widehat{\sigma}_{\rm med}(\lambdamax,\lambdamed,\lambdamin)\\
				\widehat{\sigma}_{\rm min}(\lambdamax,\lambdamed,\lambdamin)},\;\matr{\lambdamax\\\lambdamed\\\lambdamin}}+\iprod{\matr{\widehat{\sigma}_{\rm max}(\olambdamax,\olambdamed,\olambdamin)\\
				\widehat{\sigma}_{\rm med}(\olambdamax,\olambdamed,\olambdamin)\\
				\widehat{\sigma}_{\rm min}(\olambdamax,\olambdamed,\olambdamin)},\;\matr{\olambdamax\\\olambdamed\\\olambdamin}}\,.\label{appendix:ogdenbook5}
		\end{align}
		The two mixed terms in \eqref{appendix:ogdenbook4} can be estimated by Hardy \cite{hardy1952inequalities} as well as \citet{theobald1975inequality} or \citet{richter1958abschatzung} as 
		\begin{align}
		\begin{matrix}
			\iprod{\sigma(S), T}&\leq&\widehat{\sigma}_{\rm max}(\lambdamax,\lambdamed,\lambdamin)\olambdamax+\widehat{\sigma}_{\rm med}(\lambdamax,\lambdamed,\lambdamin)\olambdamed+\widehat{\sigma}_{\rm min}(\lambdamax,\lambdamed,\lambdamin)\olambdamin\\
			\iprod{\sigma(T), S}&\leq&\widehat{\sigma}_{\rm max}(\olambdamax,\olambdamed,\olambdamin)\lambdamax+\widehat{\sigma}_{\rm med}(\olambdamax,\olambdamed,\olambdamin)\lambdamed+\widehat{\sigma}_{\rm min}(\olambdamax,\olambdamed,\olambdamin)\lambdamin
		\end{matrix}\,.\label{appendix:ogdenbook6}
		\end{align}
		Putting \eqref{appendix:ogdenbook5} and \eqref{appendix:ogdenbook6} in \eqref{appendix:ogdenbook4} and using inequality \eqref{appendix:ogdenbook3} we obtain
		\begin{align}
			&\phantom.\hspace{-1cm}\iprod{\sigma(S)-\sigma(T),\, S-T}\notag\\
			\geq&\ \iprod{\matr{\widehat{\sigma}_{\rm max}(\lambdamax,\lambdamed,\lambdamin)\\
				\widehat{\sigma}_{\rm med}(\lambdamax,\lambdamed,\lambdamin)\\
				\widehat{\sigma}_{\rm min}(\lambdamax,\lambdamed,\lambdamin)},\;\matr{\lambdamax\\\lambdamed\\\lambdamin}}+\iprod{\matr{\widehat{\sigma}_{\rm max}(\olambdamax,\olambdamed,\olambdamin)\\
				\widehat{\sigma}_{\rm med}(\olambdamax,\olambdamed,\olambdamin)\\
				\widehat{\sigma}_{\rm min}(\olambdamax,\olambdamed,\olambdamin)},\;\matr{\olambdamax\\\olambdamed\\\olambdamin}}\notag\\
				&-\widehat{\sigma}_{\rm max}(\lambdamax,\lambdamed,\lambdamin)\olambdamax+\widehat{\sigma}_{\rm med}(\lambdamax,\lambdamed,\lambdamin)\olambdamed+\widehat{\sigma}_{\rm min}(\lambdamax,\lambdamed,\lambdamin)\olambdamin\notag\\
				&-\widehat{\sigma}_{\rm max}(\olambdamax,\olambdamed,\olambdamin)\lambdamax+\widehat{\sigma}_{\rm med}(\olambdamax,\olambdamed,\olambdamin)\lambdamed+\widehat{\sigma}_{\rm min}(\olambdamax,\olambdamed,\olambdamin)\lambdamin\notag\\
				=&\ \iprod{\matr{\widehat{\sigma}_{\rm max}(\lambdamax,\lambdamed,\lambdamin)\\
				\widehat{\sigma}_{\rm med}(\lambdamax,\lambdamed,\lambdamin)\\
				\widehat{\sigma}_{\rm min}(\lambdamax,\lambdamed,\lambdamin)}
				-\matr{\widehat{\sigma}_{\rm max}(\olambdamax,\olambdamed,\olambdamin)\\
				\widehat{\sigma}_{\rm med}(\olambdamax,\olambdamed,\olambdamin)\\
				\widehat{\sigma}_{\rm min}(\olambdamax,\olambdamed,\olambdamin)},\;\matr{\lambdamax\\\lambdamed\\\lambdamin}-\matr{\olambdamax\\\olambdamed\\\olambdamin}}\notag\\
				=&\ \iprod{\left(\widehat\sigma(\lambdamax,\lambdamed,\lambdamin)-\widehat\sigma(\olambdamax,\olambdamed,\olambdamin)\right),\;\matr{\lambdamax\\\lambdamed\\\lambdamin}-\matr{\olambdamax\\\olambdamed\\\olambdamin}}\geq0\,.\label{eq:HillOgdenProofStrict}
		\end{align}
		Note that $S\neq T$ does not automatically imply that the ordered eigenvalue vectors are distinct, i.e.\ the proof presented here only gives non-strict matrix-monotonicity in the last step \eqref{eq:HillOgdenProofStrict} and a separate equality-case argument is required, similar to the proof of Theorem \ref{theorem:mainResult}.
	\end{proof}
	
%
%
\subsection{A direct proof of our main lemma for the two-dimensional case}\label{appendix:mainLemma2D}
	For $n=k=2$, Lemma \ref{lemma:mainLemma}  i) can be shown by direct computation. 
	Consider the mapping
	\begin{align}
		\Psi\col\R\to\R\,,\quad \alpha\mapsto \Psi(Q(\alpha)) = \iprod{A,\, Q(\alpha)^T B\. Q(\alpha)} + \iprod{C,\, Q(\alpha)^T D\. Q(\alpha)}
	\end{align}
	with diagonal matrices $A,B,C,D\in\R^{2\times 2}$.
	The set $\SO(2)$ can be parameterized by
	\begin{align}
		Q\col \R\to\SO(2)\,,\qquad Q(\alpha) = \matr{\cos(\alpha)&-\sin(\alpha)\\ \sin(\alpha)&\cos(\alpha)}\,.
	\end{align}
	It therefore suffices to show that every maximizing critical point $\alpha_0\in\R$ satisfies $Q(\alpha_0)^T D\. Q(\alpha_0)$ is diagonal.
	Now, for any diagonal matrix $\Dhat = \diag(\dhat_1,\dhat_2)$, we find
	\begin{align}
		Q(\alpha)^T \Dhat\. Q(\alpha) &= \matr{\cos(\alpha)&\sin(\alpha)\\ -\sin(\alpha)&\cos(\alpha)} \matr{\dhat_1&0\\0&\dhat_2} \matr{\cos(\alpha)&-\sin(\alpha)\\ \sin(\alpha)&\cos(\alpha)}\notag\\
		&= \matr{\cos(\alpha)&-\sin(\alpha)\\ \sin(\alpha)&\cos(\alpha)} \matr{\dhat_1\.\cos(\alpha)&-\dhat_1\.\sin(\alpha)\\ \dhat_2\.\sin(\alpha)&\dhat_2\.\cos(\alpha)}\\
		&=
		\matr{
			\dhat_1\.\cos^2(\alpha) + \dhat_2\.\sin^2(\alpha) & (\dhat_2-\dhat_1)\.\sin(\alpha)\.\cos(\alpha)\\
			(\dhat_2-\dhat_1)\.\sin(\alpha)\.\cos(\alpha) & \dhat_1\.\sin^2(\alpha) + \dhat_2\.\cos^2(\alpha)\notag
		}\,.
	\end{align}
	Thus for $\Dtilde=\diag(\dtilde_1,\dtilde_2)$,
	\begin{align}
		\iprod{\Dtilde, Q(\alpha)^T \Dhat\. Q(\alpha)}
		&=
		\bigg\langle \matr{\dtilde_1&0\\0&\dtilde_2},\;\; 
		\matr{
			\dhat_1\.\cos^2(\alpha) + \dhat_2\.\sin^2(\alpha) & (\dhat_2-\dhat_1)\.\sin(\alpha)\.\cos(\alpha)\\
			(\dhat_2-\dhat_1)\.\sin(\alpha)\.\cos(\alpha) & \dhat_1\.\sin^2(\alpha) + \dhat_2\.\cos^2(\alpha)
		} \bigg\rangle\notag\\[.49em]
		&= \dtilde_1\left(\dhat_1\.\cos^2(\alpha) + \dhat_2\.\sin^2(\alpha)) + \dtilde_2\cdot (\dhat_1\.\sin^2(\alpha) + \dhat_2\.\cos^2(\alpha)\right)\\
		&= (\dtilde_1\.\dhat_1+\dtilde_2\.\dhat_2)\.\cos^2(\alpha) + (\dtilde_1\.\dhat_2+\dtilde_2\.\dhat_1)\.\sin^2(\alpha)\notag
	\end{align}
	and
	\begin{align}
		\iprod{A, Q(\alpha)^T B\. Q(\alpha)} + \iprod{C, Q(\alpha)^T D\. Q(\alpha)}
		&= (A_{11}\.B_{11}+A_{22}\.B_{22})\.\cos^2(\alpha) + (A_{11}\.B_{22}+A_{22}\.B_{11})\.\sin^2(\alpha)\\
		&\quad + (C_{11}\.D_{11}+C_{22}\.D_{22})\.\cos^2(\alpha) + (C_{11}\.D_{22}+C_{22}\.D_{11})\.\sin^2(\alpha)\,.\notag
	\end{align}
	Since, therefore,
	\begin{align}
		&\dd{\alpha}\; \iprod{A, Q(\alpha)^T B\. Q(\alpha)} + \iprod{C, Q(\alpha)^T D\. Q(\alpha)}\\
		&= \underbrace{2\.\sin(\alpha)\.\cos(\alpha)}_{=\sin(2\.\alpha)} \cdot \big(
		-A_{11}\.B_{11}-A_{22}\.B_{22} + A_{11}\.B_{22}+A_{22}\.B_{11} - C_{11}\.D_{11}-C_{22}\.D_{22} + C_{11}\.D_{22}+C_{22}\.D_{11}
		\big)\,,\notag
	\end{align}
	the critical points of the mapping $\alpha\mapsto \iprod{A, Q(\alpha)^T B\. Q(\alpha)} + \iprod{C, Q(\alpha)^T D\. Q(\alpha)}$ only occur at multiples of $\frac\pi2$ (unless the mapping is constant, in which case we can simply choose $\Qhat=\id$ as a maximizer). It is easy to see that for $\alpha_0=\frac{k\.\pi}{2}$ with an integer $k$, the matrix $Q(\alpha)$ is a signed permutation matrix, and thus $Q(\alpha_0)^T D\. Q(\alpha_0)$ is diagonal according to Lemma \ref{lemma:permutationMatrixPreservesDiagonalForm}.
	
%
%
\subsection{Linear elasticity as a simple example for the connection between $\sigma$ and $\sigmahat$}\label{appendix:linearElasticity}
	We want to point out the correlation between vector and matrix functions
	\begin{align}
		\sigma\col\Sym(3)\to\Sym(3)\,,\qquad\qquad\sigmahat\col\R^3\to\R^3
	\end{align}
	for linear elasticity
	\begin{align}
		W_{\rm lin}(\eps)=\mu\.\norm{\eps}^2+\frac{\lambda}{2}\tr(\eps)^2\qquad\text{with}\quad\eps=\sym(F-\id)=\sym\D u\,.
	\end{align}
	In linear elasticity, the Cauchy stress tensor $\sigma$ is defined as
	\begin{align}
		\sigma(\eps)=\D_\eps W_{\rm lin}(\eps)=2\mu\.\eps+\lambda\.\tr(\eps)\.\id\,.
	\end{align}
	The symmetric matrix $\eps\in\Sym(3)$ is diagonalizable as $\eps=Q^T\diag(\eps_1,\eps_2,\eps_3)\.Q$ for $Q\in \SO(3)$ and its eigenvalues are the principal strains $\eps_1,\eps_2,\eps_3\in\R$.
	Since by isotropy
	\begin{align}
		\sigma(Q^T\diag(\eps_1,\eps_2,\eps_3)\.Q)=Q^T\sigma(\diag(\eps_1,\eps_2,\eps_3))\.Q\,,
	\end{align}
	we can define the vector function
	\begin{align}
		\sigmahat(\eps_1,\eps_2,\eps_3)=\diag^{-1}(\sigma(\diag(\eps_1,\eps_2,\eps_3)))
		\qquad\text{with}\qquad
		\diag^{-1}\matr{\eps_1&0&0\\0&\eps_2&0\\0&0&\eps_3}=\matr{\eps_1\\\eps_2\\\eps_3}.
	\end{align}
	Thus we get
	\begin{align}
		\sigma(\diag(\eps_1,\eps_2,\eps_3))&=2\mu\.\matr{\eps_1&0&0\\0&\eps_2&0\\0&0&\eps_3}+\lambda\.(\eps_1+\eps_2+\eps_3)\matr{1&0&0\\0&1&0\\0&0&1},\notag\\
		\sigmahat(\eps_1,\eps_2,\eps_3)&=2\mu\.\matr{\eps_1\\\eps_2\\\eps_3}+\lambda\matr{\eps_1+\eps_2+\eps_3\\\eps_1+\eps_2+\eps_3\\\eps_1+\eps_2+\eps_3}=\matr{2\mu+\lambda&\lambda&\lambda\\\lambda&2\mu+\lambda&\lambda\\\lambda&\lambda&2\mu+\lambda}\matr{\eps_1\\\eps_2\\\eps_3}.
	\end{align}
	The monotonicity conditions for $\sigma$ and $\sigmahat$ are
	\begin{align}
			\iprod{\sigma(\eps)-\sigma(\overline{\eps}),\, \eps-\overline{\eps}}\geq 0
			\qquad\text{and}\qquad
		\iprod{\sigmahat(\eps_1,\eps_2,\eps_3)-\sigmahat(\overline{\eps}_1,\overline{\eps}_2,\overline{\eps}_3),\, \matr{\eps_1\\\eps_2\\\eps_3}-\matr{\overline{\eps}_1\\\overline{\eps}_2\\\overline{\eps}_3}} \geq 0
	\end{align}
	and they are equivalent because of Theorem \ref{theorem:mainResult}. Let us motivate this theorem by computing both monotonicity conditions directly. Since $\sigma$ is linear, vector-monotonicity is equivalent to
	\begin{align}
		\iprod{\sigmahat(\eps_1,\eps_2,\eps_3)-\sigmahat(\overline{\eps}_1,\overline{\eps}_2,\overline{\eps}_3),\, \matr{\eps_1\\\eps_2\\\eps_3}-\matr{\overline{\eps}_1\\\overline{\eps}_2\\\overline{\eps}_3}} &=  \iprod{\D\sigmahat\left[\matr{\eps_1\\\eps_2\\\eps_3}-\matr{\overline{\eps}_1\\\overline{\eps}_2\\\overline{\eps}_3}\right],\, \matr{\eps_1\\\eps_2\\\eps_3}-\matr{\overline{\eps}_1\\\overline{\eps}_2\\\overline{\eps}_3}}\notag\\
		&=\iprod{\sym \D\sigmahat\matr{\eps_1-\overline{\eps}_1\\\eps_2-\overline{\eps}_2\\\eps_3-\overline{\eps}_3},\, \matr{\eps_1-\overline{\eps}_1\\\eps_2-\overline{\eps}_2\\\eps_3-\overline{\eps}_3}}\geq 0\,.
	\end{align}		
	We get
	\begin{align}
		\matr{2\mu+\lambda&\lambda&\lambda\\\lambda&2\mu+\lambda&\lambda\\\lambda&\lambda&2\mu+\lambda}=\D\sigmahat=\sym \D\sigmahat\,,
	\end{align}
	with eigenvalues $2\mu$, $2\mu$ and $2\mu+3\lambda$. We want to ensure (semi-) positive definiteness of $\D\sigmahat$. Therefore, $\sigmahat$ is vector-monotone if and only if $\mu\geq 0$ and $2\mu+3\lambda\geq 0$. On the other hand, matrix-monotonicity is equivalent to
	\begin{align}
		\iprod{\sigma(\eps)-\sigma(\overline{\eps}),\, \eps-\overline{\eps}}
		&=\iprod{2\mu\.\eps+\lambda\tr(\eps)\.\id-(2\mu\.\overline{\eps}+\lambda\tr(\overline{\eps})\.\id),\, \eps-\overline{\eps}}\notag\\
		&=\iprod{2\mu\.(\eps-\overline{\eps})+\lambda\tr(\eps-\overline{\eps})\.\id,\, \eps-\overline{\eps}}\notag\\
		&=2\mu\.\norm{\eps-\overline{\eps}}^2+\lambda\.[\tr(\eps-\overline{\eps})]^2\geq 0\,.
	\end{align}
	We introduce 
	\begin{align}
		\norm{X}^2=\norm{\dev(X)+\frac{1}{3}\tr(X)\id}^2=\norm{\dev(X)}^2+\frac{1}{3}\.[\tr(X)]^2
	\end{align}
	and transform matrix-monotonicity further to
	\begin{align}
		\iprod{\sigma(\eps)-\sigma(\overline{\eps}),\, \eps-\overline{\eps}}
		&=2\mu\left(\norm{\dev(\eps-\overline{\eps})}^2+\frac{1}{3}\.[\tr(\eps-\overline{\eps})]^2\right)+\lambda\.[\tr(\eps-\overline{\eps})]^2\notag\\
		&=2\mu\.\norm{\dev(\eps-\overline{\eps})}^2+\frac{2\mu+3\lambda}{3}\.[\tr(\eps-\overline{\eps})]^2\geq 0\,.
	\end{align}
	Therefore, $\sigma$ is matrix-monotone if and only if $\mu\geq 0$ and $2\mu+3\lambda\geq 0$. We obtain the same condition for vector- and matrix-monotonicity, as expected. Strict vector- and matrix-monotonicity are equivalent to $\mu>0$ and $2\mu+3\lambda>0$.
	
%
%
\subsection{Matrix-monotonicity of $\sigma$ in $\log V$ through $\text{TSTS}^*\text{-M}^+$}\label{appendix:examples}
	In the context of Section \ref{sec:motivation}, we want to check for the important stability inequality for the Cauchy stress tensor $\sigma$ as a function of $\log V$ ($\text{TSTS-M}^+$), i.e.\ strict matrix-monotonicity in the strain tensor $\log V$
	\begin{align}
		\iprod{\sigma(\log(V))-\sigma(\log(\overline V)),\,\log(V)-\log(\overline V)}>0\,,\qquad\forall\;V,\overline V\in\Symp(3),\quad V\neq \overline V\,.
	\end{align}
	With our main Theorem \ref{theorem:mainResult}, we are able to give an alternate proof based on showing strict vector-monotonicity $\text{TSTS}^*\text{-M}^+$ inequality regarding $\sigmahat$
	\begin{align}
		\iprod{\sigmahat(\log\lambda_1,\log\lambda_2,\log\lambda_3)-\sigmahat(\log\overline{\lambda}_1,\log\overline{\lambda}_2,\log\overline{\lambda}_3),\, \left(\begin{array}{c}\log\lambda_1\\\log\lambda_2\\\log\lambda_3\end{array}\right)-\left(\begin{array}{c}\log\overline{\lambda}_1\\\log\overline{\lambda}_2\\\log\overline{\lambda}_3\end{array}\right)} >0\,,\quad\forall\matr{\lambda_1\\\lambda_2\\\lambda_3}\neq\matr{\overline\lambda_1\\\overline\lambda_2\\\overline\lambda_3}.
	\end{align}
	For this, we will first compute the Kirchhoff stress tensor
	\begin{align}
		\D_{\log V}W(\log V)=\tau=(\det V)\.\sigma
		\qquad\iff\qquad
		\sigma(\log V)=\left(\det V\right)^{-1}\D_{\log V}W(\log V)\label{eq:KirchhoffStressTensor}
	\end{align}
	to identify $\sigma(\log V)$.  Due to isotropy\footnote{We apply $\log(Q^T\diag(\lambda_1,\lambda_2,\lambda_3)Q)=Q^T\log(\diag(\lambda_1,\lambda_2,\lambda_3))Q=Q^T\diag(\log\lambda_1,\log\lambda_2,\log\lambda_3)\.Q\,.$} and $V=Q^T\diag(\lambda_1,\lambda_2,\lambda_3)Q$ it holds
	\begin{align}
		\sigma\left(\log(Q^T\diag(\lambda_1,\lambda_2,\lambda_3)\.Q)\right)=\sigma\left(Q^T\diag(\log\lambda_1,\log\lambda_2,\log\lambda_3)\.Q\right)=Q^T\sigma(\diag(\log\lambda_1,&\log\lambda_2,\log\lambda_3))\.Q\,,\label{eq:isotropyLogV}
	\end{align}
	which allows us to identify the corresponding vector function $\sigmahat(\log\lambda_1,\log\lambda_2,\log\lambda_3)$.
	
\subsubsection{An exp-Hencky-type energy}\label{appendix:example1}
	It was shown in \citet[Section 4]{agn_neff2015exponentiatedI}, that the energy function of exp-Hencky-type
	\begin{align}
		W(F)=\frac{1}{2} e^{\norm{\log V}^2}=\frac{1}{2} e^{\frac{1}{4}\norm{\log B}^2}\,,\qquad F=V\. R\,,\quad R\in\SO(3)\,,\quad V\in\Symp(3)\label{eq:ExpHencky1}
	\end{align}
	satisfies the $\text{TSTS-M}^+$ inequality. We use \eqref{eq:KirchhoffStressTensor} to identify
	\begin{align}
		\sigma(\log V)&=\left(e^{\log(\det V)}\right)^{-1}\D_{\log V}\left[\frac{1}{2}e^{\norm{\log V}^2}\right]=e^{-\tr(\log V)}\log V\.e^{\norm{\log V}^2}=\log V\.e^{\norm{\log V}^2-\tr(\log V)}.
	\end{align}
	We have \eqref{eq:isotropyLogV} to identify
	\begin{align}
		\sigmahat(\log\lambda_1,\log\lambda_2,\log\lambda_3)=\matr{\log\lambda_1\\\log\lambda_2\\\log\lambda_3}e^{\log^2(\lambda_1)+\log^2(\lambda_2)+\log^2(\lambda_3)-\log\lambda_1-\log\lambda_2-\log\lambda_3}.
	\end{align}
	The principal Cauchy stresses $\sigmahat$ satisfy $\text{TSTS}^*\text{-M}^+$, if $\sym\D\sigmahat\in\Symp(3)\,.$ The calculated eigenvalues of $\sym \D\sigmahat(x_1,x_2,x_3)$ are complicated terms
	\begin{align}	
		\matr{e^{(x_1-1) x_1+(x_2-1) x_2+(x_3-1) x_3}\\\frac{1}{2} e^{(x_1-1)
   x_1+(x_2-1) x_2+(x_3-1) x_3} \left(x_1 (2 x_1-1)+x_2 (2 x_2-1)+x_3 (2
   x_3-1)+2-\xi(x_1,x_2,x_3)\right)\\\frac{1}{2} e^{(x_1-1) x_1+(x_2-1) x_2+(x_3-1) x_3}
   \left(x_1 (2 x_1-1)+x_2 (2 x_2-1)+x_3 (2
   x_3-1)+2+\xi(x_1,x_2,x_3)\right)}\\
   		\text{with}\qquad\xi(x_1,x_2,x_3)=\sqrt{(4 (x_1-1) x_1+4 (x_2-1) x_2+4
   (x_3-1) x_3+3) \left(x_1^2+x_2^2+x_3^2\right)}\,.\notag
   	\end{align}
   	We introduce $s\colonequals x_1+x_2+x_3$ and $u\colonequals x_1^2+x_2^2+x_3^2$ and  check for the positivity of
   	\begin{equation}
		2\.u-s+2\pm\xi\,,\qquad\xi=\sqrt{(4\.u-4\.s+3)\.u}\,.
   	\end{equation}
   We compute
  	\begin{align}
  		(2\.u-s+2+\xi)(2\.u-s+2-\xi)&=(2\.u-s+2)^2-\xi^2=4\.u^2+s^2-4\.s\.u+8\.u-4\.s+4-(4\.u^2-4\.s\.u+3\.u)\notag\\
  		&=s^2+5\.u-4\.s+4=(s-2)^2+5\.u>0\,.
  	\end{align}
	Therefore, the product of both eigenvalues is always positive and it remains to show that
	\begin{equation}
		2\.u-s+2\geq2\.u-\sqrt{3\.u}+2\geq2\.u-2\.\sqrt{u}+2=u+1+(\sqrt{u}-1)^2>0\,,
	\end{equation}
	which holds because $u\geq0$ and the Cauchy-Schwarz inequality
	\begin{equation}
		s^2\leq\norm{(1,1,1)}^2\norm{(x_1,x_2,x_3)}^2=3\.u\,.
	\end{equation}
   	Therefore, $\sigmahat$ satisfies $\text{TSTS}^*\text{-M}^+$ and due to our main Theorem \ref{theorem:mainResult} the Cauchy stress tensor $\sigma$ satisfies the $\text{TSTS-M}^+$ condition everywhere.
	 
\subsubsection{A compressible Neo-Hooke-type energy}\label{appendix:example2}
	 A compressible Neo-Hooke energy with volumetric-isochoric split is given by
	 \begin{align}
	 	W(I_1,I_2,I_3)=\mu\.(I_1I_3^{-\frac{1}{3}}-3)+\frac{\kappa}{2}e^{\frac14\.[\log(I_3)]^2}=\mu\.\frac{\norm{F}^2}{(\det F)^\frac23}+\frac{\kappa}{2}e^{[\log(\det V)]^2}-3\.\mu
	 \end{align}
	 with $I_1,I_2,I_3$ the principal matrix invariants from $B=V^2$ and positive constants $\mu,\kappa\in\R^+$ (the shear and bulk modulus, respectively). In terms of eigenvalues of $V$ we get
	 \begin{align}
	 	W(\lambda_1,\lambda_2,\lambda_3)&=\mu\left(\frac{\lambda_1^2+\lambda_2^2+\lambda_3^2}{\sqrt[3]{\lambda_1^2\lambda_2^2\lambda_3^2}}-3\right)+\frac{\kappa}{2}e^{\log^2(\lambda_1\lambda_2\lambda_3)}\notag\\
	 	&=\mu\left(\left(e^{2\.\log(\lambda_1)}+e^{2\.\log(\lambda_2)}+e^{2\.\log(\lambda_3)}\right)\left(e^{\log(\lambda_1\lambda_2\lambda_3)}\right)^{-\frac{2}{3}}-3\right)+\frac{\kappa}{2}e^{(\log(\lambda_1)+\log(\lambda_2)+\log(\lambda_3))^2}\notag\\
	 	&=\mu\left(\tr(\exp(2\log V))e^{-\frac{2}{3}\tr(\log V)}-3\right)+\frac{\kappa}{2}e^{[\tr(\log V)]^2}=\widetilde W(\log V)\,.
	 \end{align}	 
	 We utilize \eqref{eq:KirchhoffStressTensor} to identify
	\begin{align}
	\sigma(\log V)=\ &\left(e^{\log(\det V)}\right)^{-1}\D_{\log V}\left[\mu\left(\tr(\exp(2\log V))e^{-\frac{2}{3}\tr(\log V)}-3\right)+\frac{\kappa}{2}e^{\tr(\log V)^2}\right]\notag\\
		=\ &e^{-\tr(\log V)}\left[\mu\left(2\exp(2\log V)e^{-\frac{2}{3}\tr(\log V)}+\tr(\exp(2\log V))\left(-\frac{2}{3}\id\right)e^{-\frac{2}{3}\tr(\log V)}\right)\right.\notag\\&+\left.\kappa\tr(\log V)\.\id\.e^{\tr(\log V)^2}\right]\notag\\
		=\ &e^{-\tr(\log V)}\left[2\mu\left(\exp(2\log V)-\frac{1}{3}\tr(\exp(2\log V))\id\right)e^{-\frac{2}{3}\tr(\log V)}+\kappa\tr(\log V)\.\id\.e^{\tr(\log V)^2}\right]\notag\\
		=\ &2\mu\dev(\exp(2\log V))e^{-\frac{5}{3}\tr(\log V)}+\kappa\tr(\log V)\.\id\.e^{\tr(\log V)^2-\tr(\log V)}.
	\end{align}
	We have \eqref{eq:isotropyLogV} to identify
	\begin{align}
		\sigmahat(\log\lambda_1,\log\lambda_2,\log\lambda_3)=\ &2\mu\matr{e^{2\log\lambda_1}-\frac{1}{3}\left(e^{2\log\lambda_1}+e^{2\log\lambda_2}+e^{2\log\lambda_3}\right)\\e^{2\log\lambda_2}-\frac{1}{3}\left(e^{2\log\lambda_1}+e^{2\log\lambda_2}+e^{2\log\lambda_3}\right)\\e^{2\log\lambda_3}-\frac{1}{3}\left(e^{2\log\lambda_1}+e^{2\log\lambda_2}+e^{2\log\lambda_3}\right)}e^{-\frac{5}{3}(\log\lambda_1+\log\lambda_2+\log\lambda_3)}\notag\\
		&+\matr{\kappa\\\kappa\\\kappa}(\log\lambda_1+\log\lambda_2+\log\lambda_3)\.e^{(\log\lambda_1+\log\lambda_2+\log\lambda_3)^2-\log\lambda_1-\log\lambda_2-\log\lambda_3}.
	\end{align}
	The principal Cauchy stresses $\sigmahat$ satisfy $\text{TSTS}^*\text{-M}^+$, if $\sym\D\sigmahat\in\Symp(3)\,.$ The material parameters $\kappa,\mu\in\R^+$ decide the domain of $\lambda\in\R^3$ where all three eigenvalues of $\sym\D\sigmahat(x_1,x_2,x_3)$ are positive. 
	 
 \subsubsection{A compressible Blatz-Ko energy}\label{appendix:example3} 
	The compressible Blatz-Ko energy is another generalization of the incompressible Neo-Hooke energy without volumetric-isochoric split. It is given by
	 \begin{align}
	 	W(I_1,I_2,I_3)=\frac{\mu}{2}\left(I_1+\frac{2}{\sqrt{I_3}}-5\right).
	 \end{align}
	 with $I_1,I_2,I_3$ the principal matrix invariants from $B=V^2$ and a positive material constant $\mu\in\R^+$. In terms of eigenvalues of $V$ we get
	 \begin{align}
	 	W(\lambda_1,\lambda_2,\lambda_3)&=\frac{\mu}{2}\left(\lambda_1^2+\lambda_2^2+\lambda_3^2+\frac{2}{\lambda_1\lambda_2\lambda_3}-5\right)\notag\\
	 	&=\frac{\mu}{2}\left(\left(e^{2\.\log(\lambda_1)}+e^{2\.\log(\lambda_2)}+e^{2\.\log(\lambda_3)}\right)+2\left(e^{\log(\lambda_1\lambda_2\lambda_3)}\right)^{-1}-5\right)\notag\\
	 	&=\frac{\mu}{2}\left(\tr(\exp(2\log V))+2\.e^{-\tr(\log V)}-5\right)=\widetilde W(\log V)\,.
	 \end{align}	 
	 We utilize \eqref{eq:KirchhoffStressTensor} to identify
	\begin{align}
	\sigma(\log V)&=\left(e^{\log(\det V)}\right)^{-1}\D_{\log V}\left[\frac{\mu}{2}\left(\tr(\exp(2\log V))+2\.e^{-\tr(\log V)}-5\right)\right]\notag\\
		&=e^{-\tr(\log V)}\left[\frac{\mu}{2}\left(2\exp(2\log V)-2\.\id\.e^{-\tr(\log V)}\right)\right]\notag\\
		&=\mu\exp(2\log V)e^{-\tr(\log V)}-\mu\.\id\.e^{-2\tr(\log V)}.
	\end{align}
	We have \eqref{eq:isotropyLogV} to identify
	\begin{align}
		\sigmahat(\log\lambda_1,\log\lambda_2,\log\lambda_3)=\mu\matr{e^{2\log\lambda_1}\\e^{2\log\lambda_2}\\e^{2\log\lambda_3}}e^{-(\log\lambda_1+\log\lambda_2+\log\lambda_3)}-\mu\.\matr{1\\1\\1}e^{-2(\log\lambda_1+\log\lambda_2+\log\lambda_3)}.
	\end{align}
	The principal Cauchy stresses $\sigmahat$ satisfy $\text{TSTS}^*\text{-M}^+$, if $\sym\D\sigmahat\in\Symp(3)\,.$ However with analytic calculations this condition is never satisfied, i.e.\ the second principal minor of $\sym\D\sigmahat$
	\begin{align}
		-\frac{\mu^2}{4}e^{-3(x_1+x_2+x_3)}\left(e^{x_1+x_2+x_3}\left(e^{2x_1}-e^{2x_2}\right)^2+16\left(e^{2x_1}+e^{2x_2}\right)\right)
	\end{align}
	is always smaller than zero for $\mu>0$. Thus $\sigmahat$ does not satisfy $\text{TSTS}^*\text{-M}^+$ and due to our main Theorem \ref{theorem:mainResult} the tensor $\sigma$ does not satisfy the $\text{TSTS-M}^+$ condition.
	 
 \subsubsection{The exp-Hencky energy with volumetric-isochoric split}\label{appendix:example4} 
	We want to test the exp-Hencky energy from \citet[Section 4]{agn_neff2015exponentiatedI}
	\begin{align}
		W(F)=\mu\.e^{\norm{\dev\log V}^2}+\frac{\kappa}{2}\.e^{[\tr\log V]^2}\,,\qquad F=V\. R\,,\quad R\in\SO(3)\,,\quad V\in\Symp(3)\label{eq:ExpHencky}
	\end{align}
	for the $\text{TSTS-M}^+$ inequality for the Cauchy stress tensor $\sigma$ as a function of $\log V\,.$
	We utilize \eqref{eq:KirchhoffStressTensor} to identify\footnote{Because of the chain rule and $\dev X$ being linear, we have $\D\left[\norm{\dev X}^2\right]=2\dev X$.}
	\begin{align}
		\sigma(\log V)&=\left(e^{\log(\det V)}\right)^{-1}\D_{\log V}\left[\mu\.e^{\norm{\dev\log V}^2}+\frac{\kappa}{2}\.e^{[\tr\log V]^2}\right]\notag\\
		&=e^{-\tr\log V}\left(2\mu\dev\log V\.e^{\norm{\dev\log V}^2}+\kappa\tr(\log V)\.\id\.e^{[\tr\log V]^2}\right)\notag\\
		&=2\mu\dev\log V\.e^{\norm{\dev\log V}^2-\tr\log V}+\kappa\tr(\log V)\.\id\.e^{[\tr\log V]^2-\tr\log V}.
	\end{align}
	We have \eqref{eq:isotropyLogV} to identify
	\begin{align}
		\sigmahat(\log\lambda_1,\log\lambda_2,\log\lambda_3)=\matr{\kappa\\\kappa\\\kappa}(\log\lambda_1&+\log\lambda_2+\log\lambda_3)\.e^{(\log\lambda_1+\log\lambda_2+\log\lambda_3)^2-\log\lambda_1-\log\lambda_2-\log\lambda_3}\\
		+2\mu\matr{\log\lambda_1-\frac{1}{3}\left(\log\lambda_1+\log\lambda_2+\log\lambda_3\right)\\\log\lambda_2-\frac{1}{3}\left(\log\lambda_1+\log\lambda_2+\log\lambda_3\right)\\\log\lambda_3-\frac{1}{3}\left(\log\lambda_1+\log\lambda_2+\log\lambda_3\right)}&e^{\frac{2}{3}\left(\log^2(\lambda_1)+\log^2(\lambda_2)+\log^2(\lambda_3)-\log\lambda_1\log\lambda_2-\log\lambda_2\log\lambda_3-\log\lambda_3\log\lambda_1\right)-\log\lambda_1-\log\lambda_2-\log\lambda_3}.\notag
	\end{align} 
	The principal Cauchy stresses $\sigmahat$ satisfy $\text{TSTS}^*\text{-M}^+$, if $\sym\D\sigmahat\in\Symp(3)\,.$
	The material parameters $\kappa,\mu\in\R^+$ decide the domain of $\lambda\in\R^3$ where all three eigenvalues of $\sym\D\sigmahat(x_1,x_2,x_3)$ are positive. This region includes a large neighbourhood of the identity but is not everywhere satisfied.
		
\subsection{Rivin's short proof of a Golden-Thompson-type inequality}\label{appendix:RivinGoldenThompson}
	The convexity of $f\col\Sym(n)\to\R,\quad f(X)=\log\tr(\exp(X))$ implies 
	\begin{align}
		&&f(\lambda\.X+(1-\lambda)\.Y)&\leq\lambda\.f(X)+(1-\lambda)\.f(Y)\notag\\
		&\iff&\qquad \log\tr(\exp(\lambda\.X+(1-\lambda)\.Y))&\leq\lambda\log\tr(\exp X)+(1-\lambda)\log\tr(\exp Y)\notag\\
		&\iff&\qquad \tr(\exp(\lambda\.X+(1-\lambda)\.Y))&\leq[\tr(\exp X)]^\lambda\.[\tr(\exp Y)]^{(1-\lambda)}\overset{\star}{\leq}\tr\left([\exp X]^\lambda\right)\.\tr\left([\exp Y]^{(1-\lambda)}\right)\notag\\
		&\;\,\Longrightarrow&\qquad \tr(\exp(\lambda\.X+(1-\lambda)\.Y))&\leq\tr(\exp(\lambda X))\.\tr(\exp((1-\lambda)Y))\notag\\
		&\iff&\qquad \tr(\exp(A+B))&\leq\tr(\exp(A))\.\tr(\exp(B))\,.
	\end{align}
	To see inequality $\star$ we set
	\begin{align}
		\lambda X=A\,,\quad(1-\lambda)\.Y=B\qquad\iff\qquad X=\frac{1}{\lambda}\.A\,,\quad Y=\frac{1}{1-\lambda}\.B\,.
	\end{align}
	Now we set $\lambda=\frac{1}{2}$ with $X=2\.A$ and $Y=2\.B$, i.e.\ we only consider
	\begin{align}
		[\tr(\exp X)]^\lambda\.[\tr(\exp Y)]^{(1-\lambda)}&=[\tr(\exp \frac{1}{\lambda}A)]^\lambda\.[\tr(\exp \frac{1}{1-\lambda}B)]^{(1-\lambda)}=\sqrt{\tr(\exp 2A)}\.\sqrt{\tr(\exp 2B)}\,,\\
		\tr\left([\exp X]^\lambda\right)\.\tr\left([\exp Y]^{(1-\lambda)}\right)&=\tr\left([\exp \lambda\frac{1}{\lambda}A)]\right)\.\tr\left([\exp (1-\lambda)\frac{1}{1-\lambda}B)]\right)=\tr(\exp A)\.\tr(\exp B)\,.\notag
	\end{align}
	Therefore, we have to prove
	\begin{align}
		\sqrt{\tr(\exp 2A)}\.\sqrt{\tr(\exp 2B)}&\leq\tr(\exp A)\.\tr(\exp B)\notag\\
		\sqrt{\left(e^{a_1}\right)^2+\dotsc+\left(e^{a_n}\right)^2}\,\sqrt{\left(e^{b_1}\right)^2+\dotsc+\left(e^{b_n}\right)^2}&\leq\left(e^{a_1}+\dotsc+e^{a_n}\right)\left(e^{b_1}+\dotsc+e^{b_n}\right)\\
		\left(\left(e^{a_1}\right)^2+\dotsc+\left(e^{a_n}\right)^2\right)\left(\left(e^{b_1}\right)^2+\dotsc+\left(e^{b_n}\right)^2\right)&\leq\left(e^{a_1}+\dotsc+e^{a_n}\right)^2\left(e^{b_1}+\dotsc+e^{b_n}\right)^2.\notag
	\end{align}
	This statement is true because
	\begin{align}
		\sum_{i=1}^nc_i^2\leq\left(\sum_{i=1}^nc_i\right)^2=\sum_{i=1}^nc_i^2+2\.(\underbrace{c_1c_2+c_1c_3+\dotsc+c_{n-1}c_n}_{>0 \text{ because } c_i>0})\,.
	\end{align}
	It only remains to show the convexity of the mapping
	\begin{align}
		f\col X\to\log\tr(\exp X)=\log\sum_{i=1}^ne^{\lambda_i}=\fpot(\lambda_1,\dotsc\lambda_n)\,.
	\end{align}
	In the symmetric eigenvalue representation $\fpot(\lambda_1,\dotsc\lambda_n)$ the Hessian matrix is given by
	\begin{align}
		\D\.\grad\fpot(\lambda_1,\dotsc,\lambda_n)=\frac{1}{\left(\sum_{i=1}^ne^{\lambda_i}\right)^{2}}\matr{\sum_{j\neq 1} e^{\lambda_1+\lambda_j} & -e^{\lambda_1+\lambda_2}& -e^{\lambda_1+\lambda_3}&\dotsc\\
			-e^{\lambda_2+\lambda_1}&\sum_{j\neq 2} e^{\lambda_2+\lambda_j}&-e^{\lambda_2+\lambda_3}&\dotsc\\
			-e^{\lambda_3+\lambda_1}& -e^{\lambda_3+\lambda_2}&\sum_{j\neq 3} e^{\lambda_3+\lambda_j} & \dotsc\\
			\vdots&\vdots&\vdots&\ddots}.
	\end{align}
	It is easy to see \cite{rivin2010golden} that the diagonal entries are always the sum of the absolute values of the off-diagonal entries in that row. Therefore, Gershgorin's circle theorem \cite{gerschgorin1931sa} yields that all eigenvalues of $\D\.\grad\fpot$ are non-negative, which implies positive semi-definiteness of $\D\.\grad\fpot$. This shows convexity of $\fpot$ and the Chandler Davis theorem implies that $X\to\log\tr(\exp X)$ is convex.
\end{appendix}
\end{document}